\documentclass{amsart}
\usepackage{amssymb,stmaryrd}
\usepackage[shortlabels]{enumitem}

\newtheorem{thm}{Theorem}[section]
\newtheorem{mainthm}{Theorem}

\newtheorem{lemma}[thm]{Lemma}
\newtheorem{cor}[thm]{Corollary}
\newtheorem{claim}{Claim}[thm]
\newtheorem{prop}[thm]{Proposition}
\newtheorem{fact}[thm]{Fact}

\newtheorem{theorem}[thm]{Theorem}

\theoremstyle{definition}
\newtheorem{defn}[thm]{Definition}
\newtheorem{question}[thm]{Question}

\theoremstyle{remark}
\newtheorem{recall}[thm]{Recall}
\newtheorem{remark}[thm]{Remark}
\newtheorem{conv}[thm]{Convention}

\newcommand\diagonal{\bigtriangleup}
\newcommand\s{\subseteq}
\newcommand\sq{\sqsubseteq}
\newcommand*\sqleft[1]{\mathrel{_{#1}{\sqsubseteq}}}
\newcommand\sqx{\sqleft{\chi}}
\newcommand\sqw{\sqleft{\omega}}
\newcommand\br{\blacktriangleright}
\newcommand\symdiff{\mathbin\triangle}
\DeclareMathOperator\suc{succ}
\DeclareMathOperator\ord{Ord}
\DeclareMathOperator\cl{cl}
\newcommand\ubd{\textsf{unbounded}}
\newcommand\onto{\textsf{onto}}
\newcommand\half{\frac{1}{2}}

\renewcommand\restriction{\mathbin\upharpoonright}
\renewcommand\mid{\mathrel{|}\allowbreak}
\DeclareMathOperator{\non}{non}
\DeclareMathOperator{\cf}{cf}

\DeclareMathOperator{\otp}{otp}
\DeclareMathOperator{\acc}{acc}
\DeclareMathOperator{\cg}{CG}
\DeclareMathOperator{\nacc}{nacc}
\DeclareMathOperator{\Tr}{Tr}

\DeclareMathOperator{\p}{P}
\DeclareMathOperator\ssup{ssup}

\newcommand*\axiomfont[1]{\textsf{\textup{#1}}}
\newcommand\zfc{\axiomfont{ZFC}}
\newcommand\gch{\axiomfont{GCH}}
\newcommand\ssh{\axiomfont{SSH}}

\newcommand\ns{\textup{NS}}
\newcommand\bd{\textup{bd}}
\newcommand\stick{{{\ensuremath \mspace{2mu}\mid\mspace{-12mu} {\raise0.6em\hbox{$\bullet$}}}}}
\setlist[enumerate,1]{label={(\roman*)}}
\newenvironment{why}[1][Proof]{\proof[#1]\mbox{}}{\endproof}

\title{A club guessing toolbox I}
\date{Preliminary preprint as of July 8, 2022. Comments are most welcome! For the latest version, visit \textsf{http://p.assafrinot.com/46}.}
\author{Tanmay Inamdar}

\address{Department of Mathematics, Bar-Ilan University, Ramat-Gan 5290002, Israel.}
\author{Assaf Rinot}
\address{Department of Mathematics, Bar-Ilan University, Ramat-Gan 5290002, Israel.}
\urladdr{http://www.assafrinot.com}

\begin{document} 
\begin{abstract} Club guessing principles were introduced by Shelah as a weakening of Jensen's diamond. Most spectacularly, they were used to prove Shelah's $\zfc$ bound on $2^{\aleph_\omega}$.
These principles have found many other applications: in cardinal arithmetic and PCF theory; in the construction of combinatorial objects on uncountable cardinals such as J\'onsson algebras, strong colourings, Souslin trees, 
and pathological graphs; to the non-existence of universals in model theory; to the non-existence of forcing axioms at higher uncountable cardinals; and many more. 

In this paper, the first part of a series, 
we survey various forms of club-guessing that have appeared in the literature,
and then systematically study the various ways in which a club-guessing sequences can be improved,
especially in the way the frequency of guessing is calibrated.

We include an expository section intended for those unfamiliar with club-guessing and which can be read independently of the rest of the article.
\end{abstract}
\maketitle
\section{Introduction}
\subsection{Motivation}\label{subsectionsurvey}
In this paper, the first of a series, we initiate a study of various aspects and forms of \emph{club guessing}. Our definitions are quite general, and in order to motivate them we start with a brief survey of the various forms of club guessing that have appeared in the literature as well as their applications. All undefined notation can be found in Section~\ref{notation}, but we remind the reader right away that for a pair $\lambda < \kappa$ of infinite regular cardinals, $E^\kappa_\lambda:= \{\delta < \kappa \mid \cf(\delta) = \lambda\}$, and $E^\kappa_{\neq \lambda}:= \{\delta < \kappa \mid \cf(\delta) \neq \lambda\}$.

Shortly after Jensen constructed Souslin trees in $L$ \cite{MR3618579}, he isolated a combinatorial principle named \emph{diamond}, which is sufficient for the construction. 

\begin{fact}[\cite{jensen}] In G{\"o}del's constructible universe $L$, for every regular uncountable cardinal $\kappa$ and every stationary $S\s \kappa$, 
there is a sequence $\vec A=\langle A_\delta \mid \delta \in S\rangle$ which is a $\diamondsuit(S)$-sequence, that is:
\begin{enumerate}
\item for every $\delta \in S$, $A_\delta$ is subset of $\delta$;
\item for every subset $A \s \kappa$, the set $\{\delta \in S \mid A_\delta=A\cap \delta\}$ is stationary in $\kappa$.
\end{enumerate}
\end{fact}
It is easy to see that for $S\s\kappa$ as above, $\diamondsuit(S)$ implies that $2^{<\kappa} = \kappa$,
and hence diamond is not a consequence of $\zfc$. In contrast, the following result of Shelah \cite{Sh:365}, which is the most well-known club-guessing result, is a theorem of $\zfc$. 
\begin{fact}[\cite{Sh:365}]\label{shelahcg}
Suppose that $\lambda< \lambda^+ < \kappa$ are infinite regular cardinals.
 Then there is a sequence $\vec C=\langle C_\delta \mid \delta \in E^\kappa_\lambda\rangle$ such that
\begin{enumerate}
\item for every $\delta \in E^\kappa_\lambda$, $C_\delta$ is a club in $\delta$;
\item for every club $D \s \kappa$, 
the set $\{\delta \in S \mid C_\delta \s D\}$ is stationary in $\kappa$.
\end{enumerate}
\end{fact}
In particular, unlike the $\diamondsuit$ principle or its descendants the $\clubsuit$ and $\stick$ principles (see \cite{Ostaszewski_space,MR511564}),
the focus is not on predicting arbitrary or even just cofinal subsets of $\kappa$, but rather only the closed and unbounded subsets of $\kappa$. 
This makes the task of guessing easier, since the collection of club subsets of $\kappa$ generate a normal $\kappa$-complete filter. 

The most famous application of Fact~\ref{shelahcg} is Shelah's PCF bound  (see \cite{Sh:400} or \cite[Theorem~7.3]{MR2768693}): $$2^{\aleph_\omega} \leq \max\{2^{\aleph_0}, \aleph_{\omega_4}\}.$$
 Apart from the upper bound on cardinal exponentiation, Fact~\ref{shelahcg} has many other uses in PCF theory. 
As an example from the basic theory, in obtaining exact upper bounds for sequences of ordinal functions (see \cite[Lemma~2.19]{MR2768693}), 
in fact showing that there are stationary sets consisting of points of large cofinality in the approachability ideal $I[\lambda]$ (see \cite[\S3]{MR2768694}). 
Outside of PCF theory, there are applications of Fact~\ref{shelahcg} to the universality spectrum of models \cite{MR1187454,MR2160658, Sh:1151}, 
cardinal invariants of the continuum \cite{MR2023448, MR2391923}, cardinal invariants at uncountable cardinals \cite{hbz, MR4144135}, to the study of the Boolean algebra $\p(\lambda)/[\lambda]^{<\lambda}$ for $\lambda$ singular of countable cofinality \cite{MR1835242}, to showing the incompactness of chromatic number \cite{Sh:1006}, 
to obtain a refinement of the downwards L{\"o}wenheim-Skolem Theorem \cite{MR2115072}, to study the saturation of the non-stationary ideal on $\p_\lambda(\kappa)$ \cite{MR1738689}, 
to obtain two-cardinal diamond principles in $\zfc$ \cite{MR1903856,MR2449470,MR2502486, Sh:e},
to obtaining consequence of forcing axioms \cite{Sh:794}, to obtain limitative results on forcing \cite{MR1420260, MR3749406}, to constructing graphs with a prescribed rate of growth of the chromatic number of its finite subgraphs \cite{MR4092983}. That Fact~\ref{shelahcg} is a theorem of $\zfc$ also imposes important limitations on the theory of forcing axioms at $\aleph_2$ (see for example \cite{MR3201836}).

While club guessing was motivated by finding a weak substitute for the diamond principle, in \cite{Sh:922}, Shelah, using arguments that materialized through the development of the theory of club guessing,
proved the next theorem on diamond, concluding a 40 year old search for such a result (see \cite{paper_r01}).
\begin{fact}[\cite{Sh:922}]
Let $\lambda$ be an uncountable cardinal, and let $S \s E^{\lambda^+}_{\neq\cf(\lambda)}$ be stationary. The following are equivalent:
\begin{enumerate}[(1)]
\item $2^\lambda= \lambda^+$;
\item $\diamondsuit(S)$.
\end{enumerate}
\end{fact}

Apart from Fact~\ref{shelahcg}, there are other, finer, forms of club guessing which are less well-known and yet altogether have a variety of applications. 
For instance, Fact~\ref{shelahcg} says nothing about the case $\kappa=\lambda^+$. For this, we have the following result of Shelah.

\begin{fact}[{\cite[Claim~2.4]{Sh:365}}]\label{shelahs21relative}
There is a sequence $\vec C=\langle C_\delta \mid \delta \in E^{\aleph_2}_{\aleph_1}\rangle$ such that 
\begin{enumerate}
\item for every $\delta \in E^{\aleph_2}_{\aleph_1}$, $C_\delta$ is a club in $\delta$ of ordertype $\omega_1$;
\item for every club $D \s \aleph_2$, there is a $\delta \in E^{\aleph_2}_{\aleph_1}$ such that the following set is cofinal in $\delta$: 
$$\{\beta <\delta \mid \min(C_\delta \setminus (\beta+1)) \in D \cap E^{\aleph_2}_{\aleph_1}\}.$$
\end{enumerate}
\end{fact}

To compare the preceding with Fact~\ref{shelahcg},
we see two differences in the corresponding Clause~(ii).
The first, here we require a single $\delta$ instead of stationarily many, however,
this is easily seen to be equivalent.\footnote{This equivalence is true in greater generality; see the discussion at the end of Section~\ref{sectionwarmingup}.} Second, which is more important,
instead of requiring $C_\delta$ to be a subset of $D$,
we now merely require that the intersection of $D$ with the set $\nacc(C_\delta)$ of all non-accumulation points of $C_\delta$ be cofinal in $\delta$. 
This choice is not arbitrary. For club many $\delta\in E^{\aleph_2}_{\aleph_1}$, 
both the set $\acc(C_\delta)$ of all accumulation points of $C_\delta$ 
and the set $D\cap\delta$ are clubs in $\delta$,
and hence $\acc(C_\delta)\cap D$ is trivially cofinal in $\delta$.

Consider now another example due to Shelah concerning the case $\kappa=\lambda^+$ (see \cite{cgdummies}, for a short proof):

\begin{fact}[{\cite[Claim 3.5]{Sh:413}}]\label{shelahs21ns}
There is a sequence $\vec C=\langle C_\delta \mid \delta \in E^{\aleph_2}_{\aleph_1}\rangle$ such that 
\begin{enumerate}
\item for every $\delta \in E^{\aleph_2}_{\aleph_1}$, $C_\delta$ is a club in $\delta$ of ordertype $\omega_1$;
\item for every club $D \s \aleph_2$, there is a $\delta \in S$ such that the following set is stationary in $\delta$:
$$\{\beta <\delta \mid \min(C_\delta \setminus (\beta+1)) \in D\}.$$
\end{enumerate}
\end{fact}

Comparing the two, we see that Fact~\ref{shelahs21relative} features a sequence where the guessing is measured
against the ideal $J^{\bd}[\delta]$ of bounded subsets of $\delta$,
whereas here we measure against the nonstationary ideal $\ns_\delta$. 
However, Fact~\ref{shelahs21relative} features a sequence which guesses clubs \emph{relative to the set $E^{\aleph_2}_{\aleph_1}$}, 
and for this reason the two results are incomparable.
In this paper, a join of the two results is obtained.

Note that so far we have always required that the clubs $C_\delta$ have the minimal possible ordertype of $\cf(\delta)$.
The small ordertype requirement is trivially gotten for a sequence that guesses as in Clause~(ii) of Fact~\ref{shelahcg}. 
In other cases, however, obtaining that the local clubs have a small ordertype requires extra care (see for example \cite[Theorem~9]{kojmanabc}, where the weaker form of Fact~\ref{shelahs21relative} is proved where one requires each of the $C_\delta$ to merely have \emph{size} $\aleph_1$). 
As a sample application,  we mention that
in \cite{MR2533525}, a strong form of club-guessing at $\aleph_1$ with minimal ordertype is used to construct a small Dowker space.

However, attention should not be restricted to guessing with minimal ordertypes.
At the level of $\aleph_1$, the ordertypes of guessing sequences play a crucial role in separating forcing axioms at $\aleph_1$ in \cite[Chapter XVII]{Sh:f}, and later in \cite{MR3093385}.
At higher cardinals, guessing sequences $\vec C$ with very large ordertypes are useful
for getting a pathological graph $G(\vec C)$ with maximal chromatic number \cite{paper12}.
An open question concerning guessing sequences of maximal ordertype is stated in \cite[Question~2]{paper11}.
For an extended discussion, see the introduction to \cite{paper19}.

At a cardinal $\kappa$ that is a limit or a successor of a limit, 
another type of relative club-guessing has shown to be useful,
where the guessing feature stipulates additional conditions on the sequence $\langle \cf(\gamma) \mid \gamma \in \nacc(C_\delta) \rangle$. 
In \cite{Sh:365, MR3321938}, the additional condition is that this sequence is strictly increasing and converging to $|\delta|$. 
This is used to construct colourings satisfying strong negative square bracket partition relations \cite{EiSh:535,EiSh:819}. 
An earlier construction (see \cite{Sh:365} or \cite[Theorem~5.19]{MR2768694}) requires that the sequence $\langle \cf(\gamma) \mid \gamma \in \nacc(C_\delta) \rangle$ have cofinally many cardinals carrying a J\'onsson algebra. This is used to construct J\'onsson algebras at $\kappa$.
Note that the existence of a club-guessing sequence of large ordertypes in $\zfc$ would have given rise to such sequences,
in particular, solving \cite[Question~2.4]{EiSh:819} in the affirmative.

We move to the next example, this time a question of Shelah.
\begin{question}[{\cite[Question 5.4]{Sh:666}}]\label{shelahquestion}
Let $\lambda < \lambda^+ = \kappa$ be regular uncountable cardinals. Is there a sequence $\vec C=\langle C_\delta \mid \delta \in E^\kappa_\lambda\rangle$ such that
\begin{enumerate}
\item for every $\delta \in E^\kappa_\lambda$, $C_\delta$ is a club in $\delta$ of ordertype $\lambda$;
\item for every club $D \s \kappa$, there is a $\delta \in E^\kappa_\lambda$ such that the following set is stationary in $\delta$:
$$\{\beta <\delta \mid \beta_1, \beta_2 \in D\text{ where }\beta_1:=\min(C_\delta \setminus(\beta+1))\ \&\ \beta_2:=\min(C_\delta \setminus (\beta_1+1))\}.$$
\end{enumerate}
\end{question}

Compared to Fact~\ref{shelahs21ns}, 
here we require that for stationarily many $\beta < \delta$, \emph{two consecutive non-accumulation points} following $\beta$ are in the club $D$.
Shelah mentions (without proof) that this slight strenghening of Fact~\ref{shelahs21ns}
combined with $\gch$ allows for the construction of a $\kappa$-Souslin tree. This is related to the open problem of whether $\gch$ implies the existence of an $\aleph_2$-Souslin tree (see \cite{paper37}),
and the earlier work of Kojman and Shelah on that matter \cite{KoSh:449}.

Here we shall prove that at the level of $\aleph_2$, an affirmative answer to Shelah's question follows from the existence of a sequence as in Fact~\ref{shelahs21relative} 
in which $D \cap E^{\aleph_2}_{\aleph_1}$ in Clause~(ii) is replaced by $D \cap E^{\aleph_2}_{\aleph_0}$.
However, Asper{\'o} has answered Shelah's question negatively \cite{MR3307877}. 
Getting this failure together with the $\gch$ remains open.

The feature of guessing consecutive points has other applications in the construction of Souslin trees: 
in \cite[\S5]{paper23}, the feature of guessing with two consecutive points allows to reduce a $\diamondsuit(\kappa)$ hypothesis from \cite{paper22}
to just $\kappa^{<\kappa}=\kappa$. 
In \cite{paper20}, a feature of guessing with $\omega$-many consecutive points is used to construct Souslin trees with precise control over their reduced powers.

Returning to the discussion after Fact~\ref{shelahs21ns}, there is another way to impose that the set of good guesses be `large'.
Here is an example, again due to Shelah.
\begin{fact}[{\cite[Claim~3.10]{Sh:572}}]\label{shelahpartitioning}
Suppose that $\kappa=\lambda^+$ for a regular uncountable cardinal $\lambda$ that is not strongly inaccessible. 
Then, there is a sequence $\langle h_\delta :C_\delta \rightarrow \omega \mid \delta \in E^\kappa_\lambda \rangle$ such that
\begin{enumerate}[(1)]
    \item for every $\delta \in E^\kappa_\lambda$, $C_\delta$ is a club in $\delta$ of ordertype $\lambda$;
    \item for every club $D \s \kappa$, there is a $\delta \in E^\kappa_\lambda$ such that  
$$\bigwedge_{n<\omega}\sup\{\beta< \delta \mid \min(C_\delta \setminus(\beta+1))\in D \cap h_\delta^{-1}\{n\}\} = \delta.$$
\end{enumerate}
\end{fact}

This result is used in \cite[\S2]{paper15} in producing a strong oscillation with $\omega$ many colours,
sufficient to derive strong colorings \cite[\S3]{paper15} and transformations of the transfinite plane \cite{paper44}.
Any improvement of the above result that partitions the club-guessing into $\theta$ many pieces
immediately translates to getting a strong oscillation with $\theta$ colours. 
This also connects to our previous discussion on guessing sequences $\langle C_\delta\mid\delta\in S\rangle$ with very large ordertypes,
since the number of pieces into which $C_\delta$ may be partitioned is bounded by $|C_\delta|$.

We shall later show that for the purpose of obtaining such partitioned club-guessings, the move from the unbounded ideal to the non-stationary ideal as in Fact~\ref{shelahs21relative} to Fact~\ref{shelahs21ns} is beneficial. Sufficient conditions and applications to an even stronger form of partitioned club-guessing 
in which there is global function $h:\kappa\rightarrow\theta$
such that $h_\delta=h\restriction C_\delta$ for all $\delta$ may be found in \cite[\S3]{Sh:365} and \cite[Theorem~4.20]{paper34}.

A very useful feature of club-guessing sequences we have so far ignored is \emph{coherence}. 
Coherent club-guessing sequences have been applied to set theory of the real line \cite{she1056}, and to cardinal invariants of the continuum  \cite{MR1906062}.
Coherent club-guessing sequences were also used to show the non-existence of a natural forcing axiom \cite{Sh:784}
and to construct strong colourings \cite{paper18,paper52}. 
Weakly coherent club-guessing at the level of $\aleph_1$ have been used to define a pathological topology on the real line \cite{MR1814230},
and weakly coherent club-guessing at the level of a successor $\lambda^+$ of a singular cardinal $\lambda$ was used in \cite[\S2.1]{paper29} to prevent $\lambda$-distributive $\lambda^+$-trees from having a cofinal branch,
thus, yielding nonspecial $\lambda^+$-Aronszajn trees.

The above is hardly an exhaustive list of applications of club guessing but merely a selection biased by the themes of this paper. 
Additional key results, including those from \cite{GiSh:577,MR2194236,MR2151585,MR2652193}  will be discussed in Part~II of this series.
Another caveat is that in this paper we shall only be concerned with getting club-guessing results at $\kappa\ge\aleph_2$.
The behavior of club-guessing at the level of $\aleph_1$ is entirely independent of $\zfc$,
and we refer the reader to \cite{MR2342443,MR2444284,MR2529910,MR2963020,MR2963019,MR3312316,MR4078213} for more on that matter.

\subsection{The results}\label{mainconventions}
Throughout the paper,
$\kappa$ stands for an arbitrary regular uncountable cardinal;
$\theta,\mu,\chi$ are (possibly finite) cardinals $\le\kappa$,
$\lambda$ and $\nu$ are infinite cardinals $<\kappa$,
$\xi,\sigma$ are ordinals $\le\kappa$,
and $S$ and $T$ are stationary subsets of $\kappa$ consisting of limit ordinals.

\begin{defn} A \emph{$C$-sequence over $S$} is a sequence $\vec C=\langle C_\delta\mid\delta\in S\rangle$ such that, for every $\delta\in S$, $C_\delta$ is a closed subset of $\delta$ with $\sup(C_\delta)=\sup(\delta)$. It is said to be \emph{$\xi$-bounded} if $\otp(C_\delta)\le\xi$ for all $\delta\in S$.
\end{defn}

Our first main result fulfills the promise of finding a join of Facts \ref{shelahs21relative} and \ref{shelahs21ns}.

\begin{mainthm}\label{THMA} For every successor cardinal $\lambda$, there exists a $\lambda$-bounded $C$-sequence $\vec C=\langle C_\delta\mid \delta\in E^{\lambda^+}_\lambda\rangle$
satisfying the following.
For every club $D \s \lambda^+$, there is a $\delta \in E^{\lambda^+}_\lambda$ such that the following set is stationary in $\delta$:
$$\{\beta <\delta \mid \min(C_\delta \setminus(\beta+1))\in D\cap E^{\lambda^+}_\lambda\}.$$
\end{mainthm}

Our next result deals with coherent guessing sequences.
\begin{mainthm}\label{THMB} For every cardinal $\lambda\ge\beth_\omega$ such that $\square(\lambda^+)$ holds,
 for all stationary subsets $S,T$ of $\lambda^+$,
there exists an $\sq^*$-coherent $C$-sequence $\vec C=\langle C_\delta\mid\delta<\lambda^+\rangle$ such that,
for every club $D\s\lambda^+$, there exists $\delta\in S$ such that $\sup(\nacc(C_\delta)\cap D\cap T)=\delta$.
\end{mainthm}

The next two results address the problem of partitioning a given club-guessing sequence into $\theta$ many pieces as in Fact~\ref{shelahpartitioning}.\footnote{\label{footnote2}The related problem of using a given club-guessing sequence to produce another club-guessing sequence that admits a partition is also addressed in this paper. See Theorem~\ref{812} and the introduction to Section~\ref{movingsection}.} 

\begin{mainthm}\label{THMC} Suppose that for each $\delta\in E^\kappa_\lambda$, $J_\delta$ is some $\lambda$-complete ideal over $\delta$,
and suppose that $\vec C=\langle C_\delta\mid\delta\in E^\kappa_\lambda\rangle$ is a given $\lambda$-bounded $C$-sequence satisfying that
for every club $D\s\lambda$, there exists $\delta\in E^\kappa_\lambda$ such that
$$\{\beta<\delta\mid \min(C_\delta\setminus(\beta+1))\in D\cap T\}\in J_\delta^+.$$

Any of the following hypotheses  imply that there exists a map $h:\lambda\rightarrow\theta$ 
such that for every club $D\s\lambda$, there exists $\delta\in E^\kappa_\lambda$ such that, for every $\tau<\theta$,
$$\{\beta<\delta\mid h(\otp(C_\delta\cap\beta))=\tau\ \&\ \min(C_\delta\setminus(\beta+1))\in D\cap T\}\in J_\delta^+.$$
\begin{enumerate}
\item $\theta=\lambda=\lambda^{<\lambda}$ and $\lambda$ is a successor cardinal;
\item $\theta=\lambda$, $\diamondsuit(\lambda)$ holds, and $\lambda$ is not Mahlo;
\item $\theta=\lambda$, $\diamondsuit^*(\lambda)$ holds, and each $J_\delta$ is normal;
\item $\theta<\lambda$ is regular, and $\lambda$ is not greatly Mahlo.
\end{enumerate}
\end{mainthm}
\begin{mainthm}\label{THMD}
Under the same setup of the previous theorem,
any of the following hypotheses  imply that there exists a sequence of maps $\langle h_\delta:\lambda\rightarrow\theta\mid \delta\in E^\kappa_\lambda\rangle$ 
such that for every club $D\s\lambda$, there exists $\delta\in E^\kappa_\lambda$ such that, for every $\tau<\theta$,
$$\{\beta<\delta\mid h_\delta(\otp(C_\delta\cap\beta))=\tau\ \&\ \min(C_\delta\setminus(\beta+1))\in D\cap T\}\in J_\delta^+.$$
\begin{enumerate}
\item $\theta=\lambda$ and $\diamondsuit^*(\lambda)$ holds;
\item $\lambda$ is not strongly inaccessible, and $\theta$ is the least to satisfy $2^\theta\ge\lambda$;
\item $\theta^+=\lambda$;
\item $\theta=\omega$ and $\lambda$ is not ineffable.
\end{enumerate}
\end{mainthm}

Our last result fulfills the promise to show that at the level of $\aleph_2$ an affirmative answer to Question~\ref{shelahquestion} follows from the existence of a sequence as in Fact~\ref{shelahs21relative} 
in which $D \cap E^{\aleph_2}_{\aleph_1}$ in Clause~(ii) is replaced by $D \cap E^{\aleph_2}_{\aleph_0}$. 
\begin{mainthm}\label{THME} For every successor cardinal $\lambda$, 
if there exists a $\lambda$-bounded $C$-sequence $\vec C=\langle C_\delta\mid \delta\in E^{\lambda^+}_{\lambda}\rangle$ such that for every club $D\s\lambda^+$,
there is a $\delta\in E^{\lambda^+}_\lambda$ such that $\sup(\nacc(C_\delta)\cap D\cap E^{\lambda^+}_{<\lambda})=\delta$,
then there exists a $\lambda$-bounded $C$-sequence $\vec C=\langle C_\delta\mid \delta\in E^{\lambda^+}_{\lambda}\rangle$ such that for every club $D\s\lambda^+$,
there is a $\delta \in E^\kappa_\lambda$ such that the following set is stationary in $\delta$:
$$\{\beta <\delta \mid \beta_1, \beta_2 \in D\text{ where }\beta_1:=\min(C_\delta \setminus(\beta+1))\ \&\ \beta_2:=\min(C_\delta \setminus (\beta_1+1))\}.$$
\end{mainthm}

\subsection{Organization of this paper}
The aim of Section~\ref{sectionwarmingup} is to give the reader a tour of the basic methods for
proving club-guessing theorems.
The purpose is introductory, and with one exception, all the results we prove are not new,
though, some of them are not widely known. In particular, we give a proof of Fact~\ref{shelahcg}. In Subsection~\ref{sectionpreliminaries} our main definition, Definition~\ref{maindefn}, can be found.

In Section~\ref{sectioncoherent}, our theme is to obtain club-guessing sequences with additional \emph{coherence} properties. This is done by starting with an arbitrary $C$-sequence with some degree of coherence, and then improving it to make it guess clubs as well, all the while preserving the coherence. This allows us to obtain `coherent forms' of known results such as Fact~\ref{shelahs21relative}. 
At the end of the section, we record the results which can be obtained by the same proofs, but without any assumptions of coherence on the initial $C$-sequence. A proof of Theorem~\ref{THMB} can be found in this section. 

In Section~\ref{sectionpartitioning}, we consider partitioned club-guessing. We show how the colouring principles of \cite{paper47,paper53} allow us to not just obtain \emph{partitioned} club guessing, but in fact \emph{partition} club guessing (recall Footnote~\ref{footnote2}). 
Furthermore, using these colouring principles allows us to separate the combinatorial content from the club guessing content in previous results about partitioned club guessing (see \cite[Lemma 3.10]{Sh:572}). 
A proof of Theorems \ref{THMC} and \ref{THMD} can be found here.

In Section~\ref{sectionsigma}, we turn to the problem of guessing many consecutive non-accumulation points as in the discussion surrounding Question~\ref{shelahquestion}. We show how a sequence guessing clubs relative to points of small cofinality can be modified for this purpose. A proof of Theorem~\ref{THME} can be found here.

In the last section, Section~\ref{movingsection}, our focus is on improving the quality of the guessing calibrated against the ideal. Mainly, our focus is moving from the unbounded ideal to the non-stationary as in the move from Fact~\ref{shelahs21relative} to Fact~\ref{shelahs21ns}. A proof of Theorem~\ref{THMA} can be found here. Similar ideas also allow us to improve some results from Section~\ref{sectionpartitioning}.

	\subsection{Notation and conventions}\label{notation} We have already listed some conventions in the beginning of Subsection~\ref{mainconventions}. Here, we list some more. 
	$\log_\chi(\lambda)$ stands for the least cardinal $\theta\le\lambda$ such that $\chi^\theta\ge\lambda$.
	For sets of ordinals $A,B$, we denote $A\circledast B:=\{(\alpha,\beta)\in A\times B\mid \alpha<\beta\}$
	and we identify $[B]^2$ with $B\circledast B$.
	For $\theta>2$, $[\kappa]^\theta$ simply stands for the collection of all subsets of $\kappa$ of size $\theta$.

	Let $E^\kappa_\theta:=\{\alpha < \kappa \mid \cf(\alpha) = \theta\}$,
	and define $E^\kappa_{\le \theta}$, $E^\kappa_{<\theta}$, $E^\kappa_{\ge \theta}$, $E^\kappa_{>\theta}$,  $E^\kappa_{\neq\theta}$ analogously.
For a stationary $S\s \kappa$, we write
$\Tr(S):= \{\alpha \in E^\kappa_{>\omega}\mid  S\cap \alpha\text{ is stationary in }\alpha\}$.

For a set of ordinals $A$, we write $\ssup(A) := \sup\{\alpha + 1 \mid \alpha \in A\}$, $\acc^+(A) := \{\alpha < \ssup(A) \mid \sup(A \cap \alpha) = \alpha > 0\}$,
$\acc(A) := A \cap \acc^+(A)$, $\nacc(A) := A \setminus \acc(A)$, and $\cl(A):= A \cup \acc^+(A)$.
A function $f:A\rightarrow\ord$ is \emph{regressive}  iff $f(\alpha)<\alpha$ for all nonzero $\alpha\in A$.
A function $f:[A]^2\rightarrow\ord$ is \emph{upper-regressive} iff $f(\alpha,\beta)<\beta$ for every $(\alpha,\beta)\in[A]^2$.

\section{Warming up}\label{sectionwarmingup}
In this introductory section we have two tasks. The first is to introduce our main definition, Definition~\ref{maindefn}, and the second is to familiarise the reader with the basic idea of all club-guessing proofs, the method of `collecting counterexamples'.\footnote{As Shelah puts it in \cite{Sh:365}: ``The moral is quite old-fashioned: if you work hard and continue to try enough times, correcting and recorrecting yourself you will eventually succeed."} 
The former is achieved in Subsection~\ref{sectionpreliminaries}.

For the latter purpose, we provide in Subsection~\ref{sectiontour} a proof of some known club-guessing results including the most famous, Fact~\ref{shelahcg}. In the process we hope to make the reader comfortable with Definition~\ref{maindefn}. However, we stress that the full generality of Definition~\ref{maindefn} is not needed in Subsection~\ref{sectiontour}.
\subsection{Preliminaries}\label{sectionpreliminaries}
The aim of our main definition, Definition~\ref{maindefn}, is to provide a language that is able to differentiate between
all of the club-guessing principles we have met in Subsection~\ref{subsectionsurvey}.
While there are numerous parameters in the definition, we hope to have convinced the reader with the examples from Section~\ref{subsectionsurvey} that all of them have been found fruitful from the point of view of applications.
\begin{defn}[\cite{paper22}]
For a set of ordinals $C$, write $$\suc_\sigma(C) := \{ C(j+1)\mid j<\sigma\ \&\allowbreak\ j+1<\otp(C)\}.$$
\end{defn}
In particular, for all $\gamma\in C$ such that $\sup(\otp(C\setminus\gamma))\ge\sigma$,
$\suc_\sigma(C\setminus\gamma)$ consists of the next $\sigma$-many successor elements of $C$ above $\gamma$.

\medskip

Throughout the paper, we shall be working with some sequence
$\vec J=\langle J_\delta\mid \delta\in S\rangle$ such that,
for each $\delta\in S$, $J_\delta$ is a $\cf(\delta)$-additive proper ideal over $\delta$ extending $J^{\bd}[\delta]:=\{ B\s\delta\mid \sup(B)<\delta\}$.

\begin{defn}[Main definition]\label{maindefn}
$\cg_\xi(S, T, \sigma,\vec J)$ asserts the existence of
a $\xi$-bounded $C$-sequence, $\vec C=\langle C_\delta\mid\delta \in S\rangle$ such that,
for every club $D\s\kappa$ there is $\delta \in S$ such that
    \[\{\beta < \delta \mid \suc_\sigma(C_\delta \setminus \beta) \subseteq D\cap T\}\in J_\delta^+.\]
\end{defn}

\begin{conv}\label{convention-omissions}  We will often simplify the notation by omitting some parameters,
in which case, these parameters take their weakest possible values. Specifically,
if we omit $\xi$, then we mean that $\xi=\kappa$;
if we omit $\vec J$, then we mean that $\vec J=\langle J^{\bd}[\delta]\mid \delta\in S\rangle$,
if we omit $\vec J$ and $\sigma$, then we mean that $\sigma=1$ and  $\vec J=\langle J^{\bd}[\delta]\mid \delta\in S\rangle$.
\end{conv}

The following collects some trivial monotonicity results on $\cg$.
\begin{prop} \label{monotonicty}
Suppose that $\cg_\xi(S, T, \sigma,\allowbreak \langle J_\delta \mid \delta \in S\rangle)$ holds. Then $\cg_{\xi'}(S', T', \sigma',\allowbreak \langle J'_\delta \mid \delta \in S'\rangle)$ also holds assuming all of the following conditions are satisfied:
\begin{enumerate}
    \item $\xi \leq \xi'$;
    \item $S \subseteq S'$;
    \item $T \subseteq T'$;
    \item $\sigma \geq \sigma'$;
    \item for each $\delta \in S'$, $J'_\delta \supseteq J_\delta$.
\end{enumerate}
In fact, in this case, if $\vec{C}$ witnesses that $\cg_\xi(S, T, \sigma,\langle J_\delta \mid \delta \in S\rangle)$, then it is also a witness for $\cg_{\xi'}(S', T', \sigma',\langle J'_\delta \mid \delta \in S'\rangle)$. 
\end{prop}

\subsection{A tour of club-guessing}\label{sectiontour}

Most of the results in this article will have the following format: we shall assume the existence of a $C$-sequence $\vec C$ witnessing a certain form of club guessing, 
and then we shall improve or modify this $\vec C$ so that it satisfies another form of club guessing, or such that it has some other properties. 
For example, Proposition~\ref{monotonicty} suggests to us that starting from $\vec C$ witnessing $\cg_\xi(S, T, \sigma,\langle J_\delta \mid \delta \in S\rangle)$, 
we may look into the possibility of reducing $\xi$, or shrinking $S$ or $T$, or increasing $\sigma$, 
or enlarging the ideals in $\langle J_\delta \mid \delta \in S\rangle$, all the while preserving the guessing properties of $\vec C$. 
We shall be considering these problems and other similar ones in this article.

What is important in all this is our ability to modify a given $C$-sequence to satisfy other, or additional, properties. In this section we shall present some of the standard techniques that one uses to make such modifications, 
and we do this by giving a proof of Fact~\ref{shelahcg} (Corollary~\ref{corfullcg} below). As our purpose is introductory, we avoid giving the most direct proofs and focus instead upon the gradual process of improving the guessing.

We then move on to proving  a less-known theorem of Shelah that $\cg(S,\kappa)$ holds for every stationary subset $S$ of every regular cardinal $\kappa\ge\aleph_2$ (Theorem~\ref{thm212} below).

We finish by giving in Proposition~\ref{prop217} an example of how a prediction principle weaker than $\diamondsuit$ consisting of a matrix of sets can be modified to obtain a club-guessing principle.

\medskip

We begin by considering a very weak variation of $\cg(S, T, \sigma,\vec J)$.
\begin{defn} $\cg_\xi(S,T,-)$ asserts
the existence of a $\xi$-bounded $C$-sequence, $\langle C_\delta\mid\delta \in S\rangle$ such that,
for every club $D\s\kappa$ there is a $\delta \in S$ with
$\sup(C_\delta \cap D\cap T)=\delta$.
\end{defn}

The following might be obvious, but since we have just begun, we give a detailed proof.
\begin{prop}\label{prop22} For every triple of regular cardinals $\mu<\lambda<\kappa$, for every stationary  $S\s E^\kappa_\lambda$,
any $\lambda$-bounded $C$-sequence over $S$ witnesses $\cg_\lambda(S,E^\kappa_\mu,-)$.
\end{prop}
\begin{proof}
Let $S \subseteq E^\kappa_\lambda$ be stationary, and let $\langle C_\delta \mid \delta \in S \rangle$ be a $\lambda$-bounded $C$-sequence. 
Given a club $D$ in $\kappa$, pick $\delta \in S \cap \acc(D)$. Since $\cf(\delta) = \lambda$ which is a regular uncountable cardinal, $D\cap\delta$ is club in $\delta$.
Pick a closed and unbounded subset $B_\delta$ of $\delta$ of ordertype $\lambda$ such that $B_\delta \subseteq D$. 
Since $C_\delta$ is also club in $\delta$, and $\delta$ has uncountable cofinality, $B_\delta \cap C_\delta$ is as well club in $\delta$.
Let $\langle a_\delta(i) \mid i < \lambda\rangle$ be the increasing enumeration of $C_\delta \cap D_\delta$. As this is an increasing and continuous sequence, it is clear then that for every $j \in \acc(\lambda)$, $\cf(a_\delta(j)) = \cf(j)$. Since $\mu < \lambda$, the set $E^\lambda_\mu$ is cofinal in $\lambda$, and so for every $j \in E^\lambda_\mu$, $a_\delta(j) \in E^\kappa_\mu$. 
It follows that $\{a_\delta(j) \mid j\in E^\lambda_\mu\}$ is a subset of $C_\delta \cap D \cap E^\lambda_\mu$ which is unbounded in $\delta$.
\end{proof}

Our goal now is to show that if $\lambda^+<\kappa$, then $\cg_\lambda(S,T,-)$ implies
that there is a $C$-sequence $\langle C_\delta\mid\delta\in S\rangle$ with the property that for every club $D\s\kappa$, the set $\{\delta\in S\mid C_\delta\s D\}$ is stationary. 
In doing so, the challenge lies in improving instances of ``$\sup(C_\delta\cap D)=\delta$'' into instances of ``$C_\delta\s D$''.
A natural approach is to shrink each club $C_\delta$ into a smaller club in $\delta$, say $\Phi(C_\delta)$. In an ideal scenario, a single such act of shrinking will be enough and we will have our result. If the scenario is not so ideal, we would hope that $\Phi(C_\delta)$ is at least `better' than $C_\delta$, or `takes care of the requirements imposed by more clubs' than $C_\delta$ (we will be more precise momentarily). A common strategy in club guessing is to assume that there are no such ideal scenarios, and then in this case perform this shrinking process (equivalently, improvement process) iteratively for long enough that a contradiction results.

We return to precision. We shall need the following operator in what follows for purposes we have already hinted at.
\begin{defn}\label{firstoperatordefn}
For a subset $B\s\kappa$, we define the operator $\Phi^B:\mathcal P(\kappa)\rightarrow\mathcal P(\kappa)$ by letting for all $x \subseteq \kappa$, 
    \[\Phi^B(x):=\begin{cases}
\cl(x\cap B),&\text{if } \sup(x\cap B)=\sup(x);\\
x\setminus\sup(x\cap B),&\text{otherwise}.
\end{cases}\]
\end{defn}
We list a few useful properties of $\Phi^B$:
\begin{enumerate}
    \item $\sup(\Phi^B(x))=\sup(x)$;
    \item If $\sup(x\cap B)=\sup(x)$, then $\nacc(\Phi^B(x))$ is a cofinal subset of $x \cap B$;
    \item If $x$ is a closed subset of $\sup(x)$, then $\Phi^B(x) \subseteq x$ and $\otp(\Phi^B(x)) \leq \otp(x)$.
\end{enumerate}

\begin{lemma} \label{weaktotail}
Suppose that $\vec C=\langle C_\delta\mid \delta\in S\rangle$ witnesses $\cg_\lambda(S,T,-)$,
where $S$ and $T$ are stationary subsets of $\kappa$.
If $ \lambda^+<\kappa$,
then there exists a club $D\s\kappa$ such that $\langle \Phi^{D\cap T}(C_\delta)\mid \delta\in S\rangle$  witnesses $\cg_\lambda(S,T,\kappa)$.
\end{lemma}
\begin{proof}  Without loss of generality, $S\s\acc(\kappa)$.
Suppose that the conclusion does not hold. In this case, for every club $D\subseteq \kappa$, there is a club $F^D \subseteq \kappa$ such that, for every $\delta \in S$,
$$\sup(\nacc(\Phi^{D\cap T}(C_\delta))\setminus (F^D\cap T))=\delta.$$
Here, we use that 
$$\otp(\nacc(\Phi^{D\cap T}(C_\delta))) \leq \otp(\Phi^{D\cap T}(C_\delta)) \leq \otp(C_\delta) \leq \lambda<\kappa.$$
We construct now a $\subseteq$-decreasing sequence $\langle D_i \mid i\leq \lambda^+ \rangle$ of clubs in $\kappa$ as follows:
\begin{enumerate}
    \item $D_0 := \kappa$;
    \item $D_{i+1} := D_i \cap F^{D_i}$;
    \item for $i\in\acc(\lambda^++1)$, $D_i := \bigcap_{i'< i} D_{i'}$.
\end{enumerate}
Since $\lambda^+ < \kappa$, all these are club subsets of $\kappa$.
As $\vec C$ witnesses $\cg_\lambda(S,T,-)$, let us now pick $\delta\in S$ with $\sup(C_\delta\cap D_{\lambda^+}\cap T)=\delta$.
In particular, for all $i<\lambda^+$, $\sup(C_\delta\cap D_i\cap T)=\sup(C_\delta)$,
so that $\Phi^{D_i\cap T}(C_\delta)=\cl(C_\delta\cap D_i\cap T)$. Now, as
$\langle D_i \mid i<\lambda^+ \rangle$ is $\subseteq$-decreasing,
so is  $\langle C_\delta\cap D_i\cap T \mid i<\lambda^+ \rangle$.
But $\otp(C_\delta)\le\lambda$, so we may find some $i<\lambda^+$ such that $C_\delta\cap D_i\cap T=C_\delta\cap D_{i+1}\cap T$.

By the choice of $F^{D_i}$, and since $D_{i+1} \s F^{D_i}$, we have that 
$$\sup(\nacc(\Phi^{D_i\cap T}(C_\delta))\setminus (D_{i+1}\cap T))=\delta.$$
However, $\nacc(\Phi^{D_i\cap T}(C_\delta))\s C_\delta \cap D_i\cap T= C_\delta \cap D_{i+1}\cap T$, which is a contradiction.
\end{proof}

So $\cg_\lambda(S,T,-)$ implies $\cg_\lambda(S,\allowbreak T,\kappa)$,
provided that $ \lambda^+<\kappa$. 
Likewise, $\cg(S,\allowbreak T,\kappa)$ holds whenever there is a witness $\vec C=\langle C_\delta\mid\delta\in S\rangle$ to $\cg(S,T)$
such that $|C_\delta|<\delta$ for club many $\delta\in S$.

The instance $\cg(S,T,\sigma)$ with $\sigma=\kappa$ is sometimes dubbed \emph{tail club guessing}.
The next lemma derives a stronger form of guessing from tail club guessing.
\begin{lemma} \label{taillemma} $\cg(S,T,\kappa)$ holds iff there is a $C$-sequence $\langle C_\delta\mid\delta\in S\rangle$ such that:
\begin{enumerate}
\item for every $\delta\in S$, $\otp(C_\delta)=\cf(\delta)$;
\item for every club $D\s\kappa$, the set $\{ \delta\in S\mid \nacc(C_\delta)\s D\cap T\}$ is stationary.
\end{enumerate}
\end{lemma}
\begin{proof} Only the forward implication requires an argument. Let $\vec C=\langle C_\delta\mid\delta\in S\rangle$ be a $\cg(S,T,\kappa)$-sequence.
For every $i<\kappa$, we define the operator $\Phi_i:\mathcal P(\kappa)\rightarrow\mathcal P(\kappa)$ by letting for all $x \subseteq \kappa$, 
    \[\Phi_i(x):=\begin{cases}
x,&\text{if } x\s i;\\
x\setminus i,&\text{otherwise}.
\end{cases}\]

It is clear that $\Phi_i(x)$ is a cofinal subset of $x$,
and $\nacc(\Phi_i(x))\s\nacc(x)$.
Furthermore, if $x$ is club in its supremum, then so is $\Phi_i(x)$.

\begin{claim} \label{warmfirstrefine}There exists $i<\kappa$ such that, for every club $D\s\kappa$, the set $\{ \delta\in S \mid \nacc(\Phi_i(C_\delta))\s D\cap T\}$ is stationary.
\end{claim}
\begin{why} Suppose not. For each $i<\kappa$, fix a sparse enough club $D_i\s\kappa$ for which
$\{ \delta\in S \mid \nacc(\Phi_i(C_\delta))\s D_i\cap T\}$ is disjoint from $D_i$.
Let $D:=\diagonal_{i<\kappa}D_i$. By the choice of $\vec C$, there are $\delta\in S\cap D$ and $\beta<\delta$ such that $\suc_\kappa(C_\delta\setminus\beta)\s D\cap T$.
As $\otp(C_\delta) < \kappa$, we can find an $i<\delta$ such that $\nacc(C_\delta)\setminus i\s D\cap T$.
Then $\nacc(\Phi_i(C_\delta))\s D\cap T$. But $i<\delta$ and $\delta\in D$, so that $\delta\in D_i$. 
This is a contradiction.
\end{why}

Let $i$ be given by the preceding claim. 
The sequence $\langle \Phi_i(C_\delta)\mid \delta\in S\rangle$ satisfies Clause~(ii) of the Lemma.
In order to incorporate Clause~(i),
for each $\delta\in S$, 
we simply pick a club $C_\delta^\bullet$ in $\delta$ of ordertype $\cf(\delta)$ such that $\nacc(C_\delta^\bullet)\s\nacc(\Phi_i(C_\delta))$.
Evidently, $\langle C_\delta^\bullet\mid \delta\in S\rangle$ is as sought.
\end{proof}

Putting everything together, we arrive at the following striking conclusion.

\begin{cor}[\cite{Sh:365}]\label{thm23} For every regular uncountable cardinal $\lambda$ such that $\lambda^+<\kappa$,
for every stationary $S\s E^\kappa_\lambda$, there is a $C$-sequence $\vec C=\langle C_\delta\mid\delta\in S\rangle$ such that the following two hold:
\begin{enumerate}
    \item for every $\delta\in S$, $\otp(C_\delta)=\lambda$;
    \item for every club $D\s\kappa$, the set $\{ \delta\in S\mid C_\delta\s D\}$ is stationary.
\end{enumerate}
\end{cor}
\begin{proof} By Proposition~\ref{prop22}, in particular, $\cg_\lambda(S,E^\kappa_{\aleph_0},-)$ holds.
Then, Lemma~\ref{weaktotail} implies that so does $\cg_\lambda(S,E^\kappa_{\aleph_0},\kappa)$.
Now, appeal to Lemma~\ref{taillemma}.
\end{proof}

At this point, it is natural to ask whether it is possible to waive the uncountability hypothesis on $\lambda$ in the preceding theorem.
We shall show that this is indeed the case, by invoking an operation different than that of $\Phi^B$.

\begin{defn}\label{dropdefn}
For a subset $D\s\kappa$, we define the operator $\Phi_D:\mathcal P(\kappa)\rightarrow\mathcal P(\kappa)$ by letting for all $x \subseteq \kappa$, 
    \[\Phi_D(x):=\begin{cases}
\{\sup(D \cap \eta )\mid \eta \in x, \eta > \min(D)\},&\text{if } \sup(D\cap\sup(x))=\sup(x);\\
x\setminus\sup(D\cap\sup(x)),&\text{otherwise}.
\end{cases}\]
\end{defn}

We list a few useful properties of $\Phi_D$:
\begin{enumerate}
    \item $\sup(\Phi_D(x))=\sup(x)$;
    \item $\otp(\Phi_D(x)) \leq \otp(x)$;
    \item If $\sup(D\cap\sup(x))=\sup(x)$, then $\acc^+(\Phi_D(x))\s \acc^+(D)\cap\acc^+(x)$.
    If in addition, $D$ is closed below $\sup(x)$, then $\Phi_D(x)\s D$.
\end{enumerate}

\begin{lemma} \label{tailomega}Suppose that $\kappa\ge\aleph_2$, and that $\langle C_\delta\mid \delta\in S\rangle$
is an $\omega$-bounded $C$-sequence over a stationary $S\s E^{\kappa}_{\aleph_0}$.

Then there is a club $D\s\kappa$ such that $\langle \Phi_D(C_\delta)\mid \delta\in S\rangle$ witnesses $\cg_\omega(S,\kappa,\kappa)$.
\end{lemma}
\begin{proof} Suppose not. In this case, for every club $D \subseteq \kappa$, there is a club $F^D \subseteq \kappa$ such that, for every $\delta \in S$,
$$\sup(\Phi_{D}(C_\delta)\setminus F^D)=\delta.$$
Here we have used that since $C_\delta$ has ordertype $\omega$, $\Phi_D(C_\delta)$ has ordertype $\omega$ as well, and hence all of its points are nonaccumulation points.

As $\kappa>\aleph_1$, we may construct a $\subseteq$-decreasing sequence $\langle D_i \mid i\leq \omega_1 \rangle$ of clubs in $\kappa$ as follows:
\begin{enumerate}
    \item $D_0 := \kappa$;
    \item $D_{i+1} := D_i \cap F^{D_i}$;
    \item for $i\in\acc(\omega_1+1)$, $D_i := \bigcap_{i'< i} D_{i'}$.
\end{enumerate}

Pick $\delta \in S \cap \acc(D_{\omega_1})$. For each $i< \omega_1$, since $D_i\cap \delta$ is a closed unbounded subset of $\delta$,
it is the case that $\Phi_{D_i}(C_\delta) = \{\sup(D_i \cap \eta )\mid \eta \in C_\delta, \eta > \min(D_i)\}$,
and $\Phi_{D_i}(C_\delta)\s D_i$.

As  $\langle D_i \mid i\leq \omega_1 \rangle$ is $\s$-decreasing, for each $\eta \in C_\delta$, $\langle \sup(D_i \cap \eta) \mid i< \omega_1\rangle$ is a weakly decreasing sequence of ordinals. By well-foundedness of the ordinals, for each $\eta\in C_\delta$,
there must be some $i_{\eta} < \omega_1$ such that $\sup(D_i\cap \eta) = \sup(D_j \cap \eta)$
whenever $i_\eta\le i<j<\omega_1$.
Let $i^* := \sup_{\eta \in C_\delta}i_\eta$, which is a countable ordinal as $C_\delta$ is a countable set. It follows that for any $i\in[i^*,\omega_1)$,
$\Phi_{D_i}(C_\delta) = \Phi_{D_{i+1}}(C_\delta)$.
However, $\Phi_{D_{i+1}}(C_\delta)\s  D_{i+1} \subseteq F^{D_i}$, contradicting the choice of $F^{D_i}$.
\end{proof}

Putting everything together:

\begin{cor}[\cite{Sh:365}]  \label{corfullcg}For every pair of infinite regular cardinals $\lambda<\kappa$ 
and every stationary $S\s E^\kappa_\lambda$,
if $\lambda^+<\kappa$, then
there is a $C$-sequence $\vec C=\langle C_\delta\mid\delta\in S\rangle$ with the property that for every club $D\s\kappa$,
the set $\{\delta\in S\mid C_\delta\s D\}$ is stationary.\qed
\end{cor}
 
We now move on to prove a lesser-known theorem of Shelah concerning club-guessing.
Unlike the previous result, in the following, $S$ is not assumed to be a subset of $E^\kappa_\lambda$ for some fixed cardinal $\lambda<\kappa$. 
So, for instance, $S$ could be the set of regular cardinals below a Mahlo cardinal $\kappa$.
\begin{theorem}[Shelah]\label{thm212} Suppose $\kappa\ge\aleph_2$.

For every stationary $S\s\kappa$, $\cg(S,\kappa,1)$ holds.
\end{theorem}

Our proof of Theorem~\ref{thm212} goes through the notion of an amenable $C$-sequence,
which is a strengthening of $\otimes_{\vec C}$ from \cite[p.~134]{Sh:365}.
\begin{defn}[{\cite[Definition~1.3]{paper29}}]\label{defame} For a stationary $S\s\kappa$, a $C$-sequence $\langle C_\delta\mid\delta\in S\rangle$ is \emph{amenable}
iff for every club $D\s\kappa$, the set $\{ \delta\in S\mid \sup(D\cap\delta\setminus C_\delta)<\delta\}$ is nonstationary in $\kappa$.
\end{defn}

\begin{fact}[{\cite[Corollary~3.11]{paper47}}]\label{amenfact} For every stationary $S\s\kappa$, there exists a stationary $S'\s S$ such that $S'$ carries an amenable $C$-sequence.
\end{fact}

\begin{lemma}\label{lemma215} Suppose that $S\s\kappa$ is stationary and $\vec C=\langle C_\delta\mid \delta\in S\rangle$ is an amenable $C$-sequence.
If $\kappa\ge\aleph_2$, then there exists a club $D\s\kappa$ for which 
$\langle \Phi_D(C_\delta)\mid \delta\in S\rangle$ witnesses $\cg(S,\kappa,1)$.
\end{lemma}
\begin{proof} Suppose not. In this case, for every club $D \subseteq \kappa$, there is a club $F^D \subseteq \kappa$ such that, for every $\delta \in S$,
$$\sup(\nacc(\Phi_{D}(C_\delta))\cap F^D)<\delta.$$
As $\kappa\ge\aleph_2$, we may construct a $\subseteq$-decreasing sequence $\langle D_i \mid i\leq \omega_1 \rangle$ of clubs in $\kappa$ as follows:
\begin{enumerate}
    \item $D_0 := \kappa$;
    \item $D_{i+1} := D_i \cap F^{D_i}$;
    \item for $i\in\acc(\omega_1+1)$, $D_i := \bigcap_{i'< i} D_{i'}$.
\end{enumerate}

As $\vec C$ is amenable and $D_{\omega_1}$ is club in $\kappa$, we may pick some $\delta\in S$ such that $\sup(D_{\omega_1}\cap\delta\setminus C_\delta)=\delta$.
For each $i< \omega_1$, since $D_i\cap \delta$ is a closed unbounded subset of $\delta$,
it is the case that $$\Phi_{D_i}(C_\delta) = \{\sup(D_i \cap \eta )\mid \eta \in C_\delta, \eta > \min(D_i)\}.$$
So $\Phi_{D_i}(C_\delta)\s D_i$ and $\acc(\Phi_{D_i}(C_\delta))\s \acc(D_i)\cap\acc(C_\delta)$.

In addition, for each $i<\omega_1$, since $D_{i+1}\s F^{D_i}$, the following ordinal is smaller than $\delta$:
$$\epsilon_i:=\sup(\nacc(\Phi_{D_i}(C_\delta))\cap D_{i+1}).$$
\begin{claim}\label{2131} There exists $I\s\omega_1$ of ordertype $\omega$ such that $\sup\{\epsilon_i\mid i\in I\}<\delta$.
\end{claim}
\begin{why} If $\cf(\delta)>\omega_1$, then just let $I:=\omega$. 
If $\cf(\delta)=\omega$, then pick a countable cofinal subset $E$ of $\delta$ and for each $i\in\omega_1$,
find the least $\epsilon\in E$ such $\epsilon_i\le\epsilon$. By the pigeonhole principle, 
there is an $\epsilon\in E$ for which $\{ i\in I\mid \epsilon_i\le\epsilon\}$ is uncountable. In particular, this set contains a subset of ordertype $\omega$.
\end{why}

Fix $I$ as in the claim, and then pick $\gamma\in D_{\omega_1}\cap\delta\setminus C_\delta$ above $\sup\{ \epsilon_i\mid i\in I\}$.
As $\gamma\notin C_\delta$, $\eta:=\min(C_\delta\setminus\gamma)$ is in $\nacc(C_\delta)$.
As $\langle \sup(D_i \cap \eta) \mid i\in I\rangle$ is a weakly decreasing sequence of ordinals, by well-foundedness 
there must be a pair of ordinals $i<j$ in $I$ such that $\beta_i:=\sup(D_i\cap \eta)$ is equal to  $\beta_j:=\sup(D_j \cap \eta)$.

As $\gamma\in D_{\omega_1}\s D_i$, $\epsilon_i<\gamma\le\beta_i\le\eta$, so $\beta_i\in\Phi_{D_i}(C_\delta)\cap(\epsilon_i,\eta]$.
Likewise, $\beta_j\in\Phi_{D_j}(C_\delta)\cap(\epsilon_j,\eta]$.
Recalling that $\beta_i=\beta_j\in D_j\s D_{i+1}$, it follows that $\beta_i$ is an element of $\Phi_{D_i}(C_\delta)\cap D_{i+1}$ above $\epsilon_i$
and hence $\beta_i\in\acc(\Phi_{D_i}(C_\delta))$.
However, $\acc(\Phi_{D_i}(C_\delta))\s \acc(D_i)\cap\acc(C_\delta)$, and hence $\beta_i\in\acc(C_\delta)$.
But $\gamma\le\beta_i\le\eta$ and $C_\delta\cap[\gamma,\eta]=\{\eta\}$, and hence $\beta_i=\eta$,
contradicting the fact that $\eta\in\nacc(C_\delta)$.
\end{proof}

\begin{proof}[Proof of Theorem~\ref{thm212}] Given a stationary $S\s\kappa$, 
appeal to Fact~\ref{amenfact} to find an amenable $C$-sequence $\langle C_\delta\mid\delta\in S'\rangle$ for some stationary $S'\s S$.
Then, by Lemma~\ref{lemma215}, $\cg(S',\kappa,1)$ holds.
So $\cg(S,\kappa,1)$ holds as well.
\end{proof}

As pointed out in the introduction, if $\diamondsuit(S)$ holds for a given stationary subset $S$ of $\kappa$,
then, for every stationary $T\s\kappa$, $\cg(S,T,\kappa)$ holds. The next result
shows how to get $\cg(S,T,\kappa)$ from a principle weaker than $\diamondsuit(S)$ and even weaker than $\clubsuit(S)$
and of which many instances hold true in $\zfc$.

\begin{defn}[\cite{paper07}]\label{clubminus} For a stationary subset $S$ of a regular uncountable cardinal $\kappa$,
$\clubsuit^-(S)$ asserts the existence of a sequence $\langle \mathcal A_\delta\mid \delta\in S\rangle$ such that:
\begin{enumerate}
\item for all $\delta\in S$, $\mathcal A_\delta\s [\delta]^{<|\delta|}$ and $|\mathcal A_\delta|\le|\delta|$;
\item for every cofinal $Z\s\kappa$, there are $\delta\in S$ and $A\in\mathcal A_\delta$ with $\sup(A\cap Z)=\delta$.
\end{enumerate}
\end{defn}
\begin{remark}
Note that if $\clubsuit^-(S)$ holds, then $\{\delta\in S\mid \cf(\delta)<|\delta|\}$ must be stationary.
\end{remark}
\begin{fact}[\cite{paper07}]  For an infinite cardinal $\lambda$ and a stationary $S\s\lambda^+$:
\begin{itemize}
\item If $S\cap E^{\lambda^+}_{\neq\cf(\lambda)}$ is stationary, then  $\clubsuit^-(S)$ holds;
\item If $\square_\lambda^*$ holds and $S$ reflects stationarily often, then  $\clubsuit^-(S)$ holds.
\end{itemize}
\end{fact}

In reading the statement of the next two propositions, keep in mind Lemma~\ref{taillemma}.
\begin{prop}\label{prop217} Suppose that $\clubsuit^-(S)$ holds for a given stationary $S\s \kappa$. 

	Then, for every stationary $T\s\kappa$, $\cg(S,T,\kappa)$ holds. 
\end{prop}
\begin{proof}
Let $\langle \mathcal A_\delta\mid \delta\in S\rangle$ be a $\clubsuit^-(S)$-sequence. For each $\delta\in S$,
fix an enumeration $\{ A_{\delta,i}\mid i<\delta\}$ of $\mathcal A_\delta$.
\begin{claim} There exists $i<\kappa$ such that, for every club $E\s\kappa$, there is $\delta\in S$ with $\sup(A_{\delta,i}\cap E\cap T)=\delta$.
\end{claim}
\begin{why} Otherwise, for each $i<\kappa$, we may pick a counterexample $E_i$. Let $Z:=T\cap \diagonal_{i<\kappa}E_i$.
Pick $\delta\in S$ and $i<\delta$ such that $\sup(A_{\delta, i}\cap Z)=\delta$. Since $Z\cap(i,\delta)\s T\cap E_i$, we have that $\sup(A_{\delta,i}\cap E_i\cap T)=\delta$. This contradicts the choice of $E_i$.
\end{why}

Fix $i$ as given by the preceding claim, and denote $A_\delta:=A_{\delta,i}$. We shall now make use of the operator $\Phi^B$ from Definition~\ref{firstoperatordefn}.
\begin{claim} There exists a club $D\s\kappa$ such that, for every club $E\s\kappa$, there exists $\delta\in S$ with $\sup(A_{\delta})=\delta$ and $\nacc(\Phi^{D\cap T}(A_{\delta}))\s E$.
\end{claim}
\begin{why} 
    Suppose not. In this case, for every club $D\subseteq \kappa$, there is a club $F^D \subseteq \kappa$ such that for every $\delta \in S$,
either $\sup(A_\delta)<\delta$ or $\nacc(\Phi^{D\cap T}(A_{\delta}))\nsubseteq F^D$.
We construct a $\subseteq$-decreasing sequence $\langle D_i \mid i<\kappa\rangle$ of clubs in $\kappa$ as follows:
\begin{enumerate}
    \item $D_0 := \kappa$;
    \item $D_{i+1} := D_i \cap F^{D_i}$;
    \item for $i\in\acc(\kappa)$, $D_i := \bigcap_{i'< i} D_{i'}$.
\end{enumerate}

Let $E:=\diagonal_{i<\kappa}D_i$. Pick $\delta\in S$ with $\sup(A_\delta\cap E\cap T)=\delta$.
For every $i<\delta$, $\delta\in\acc^+(D_i\cap T)$ and $\nacc(\Phi^{D_i\cap T}(A_\delta))\nsubseteq D_{i+1}$,
so that we may pick $\beta_i\in (A_\delta\cap D_i\cap T)\setminus D_{i+1}$.
As $|A_\delta|<\delta$, let us fix $i<j<\delta$ such that $\beta_i=\beta_j$.
So $\beta_i\notin D_{i+1}$ while $\beta_j\in D_j\s D_{i+1}$. This is a contradiction.
\end{why}

Let $D$ be given by the preceding claim. Write $C_\delta:=\Phi^{D\cap T}(A_\delta)$.
Then $\langle C_\delta\mid\delta\in S\rangle$ witness $\cg(S,T,\kappa)$.
\end{proof}
\begin{cor}\label{cor36} For every uncountable cardinal $\lambda$, every stationary $S\s E^{\lambda^+}_{\neq\cf(\lambda)}$,
and every stationary $T\s\lambda^+$, $\cg(S,T,\kappa)$ holds.\qed
\end{cor}

We end this section by saying a few words about the following natural generalisation of Definition~\ref{maindefn}.
\begin{defn} \label{defnjkappa}For an ideal $J_\kappa$ over $\kappa$,
$\cg_\xi(J_\kappa, T, \sigma,\vec J)$ asserts the existence of
a $C$-sequence, $\vec C=\langle C_\delta\mid\delta <\kappa\rangle$ such that
$\{\delta<\kappa\mid \otp(C_\delta)>\xi\}\in J_\kappa$, and such that
for every club $D\s\kappa$, 
$$\{\delta<\kappa\mid \{\beta < \delta \mid \suc_\sigma(C_\delta \setminus \beta) \subseteq D\cap T\}\in J_\delta^+\}\in J_\kappa^+.$$
\end{defn}

First, the proof of Lemma~\ref{taillemma} makes it clear that 
the usual principle $\cg_\xi(S, T, \sigma,\vec J)$ coincides with $\cg_\xi(J_\kappa, T, \sigma,\vec J)$ for $J_\kappa:=\ns_\kappa\restriction S$. That is, obtaining a single witness $\delta \in S$ to an instance of guessing a club is equivalent to obtaining stationarily-many witnesses $\delta \in S$.

Second, in any of the upcoming results that involve pumping an instance $\cg_\xi(S,\ldots)$ into a better instance $\cg_{\bar\xi}(S,\ldots)$,
no new ideas would be needed in order to get the analogous result where $S$ is replaced by an abstract $\kappa$-complete ideal $J_\kappa$ over $\kappa$.
In fact, for many of the results, letting $J_\kappa$ be an $\aleph_1$-indecomposable ideal (see for instance \cite[\S2]{MR2652193}) over $\kappa$ would be sufficient. For this reason we eschew the added generality 
of Definition~\ref{defnjkappa} and focus on Definition~\ref{maindefn}.

\section{Coherent sequences}\label{sectioncoherent}
Let us point out some commonalities in the proofs of Lemmas \ref{weaktotail}, \ref{taillemma}, \ref{tailomega} and \ref{lemma215}. 
In all of these, we started with a $C$-sequence, and then we improved it using some operation $\Phi:\mathcal P(\kappa)\rightarrow\mathcal P(\kappa)$.
These similarities leads one to describe abstractly the class of such operations to which the examples we've met  belong, in the hope that known members or properties of this class might be of assistance in future endeavours. 
This class has in fact already been delineated in work of Brodsky and Rinot in \cite{paper29}, where they occurred in the work on constructing trees with prescribed properties by studying how the properties of a $C$-sequence affect the properties of the trees derived from walks on ordinals.

\begin{defn}[\cite{paper29}]\label{defpp}
Let $\mathcal K(\kappa):=\{ x\in\mathcal P(\kappa)\mid x\neq\emptyset\ \&\ \acc^+(x)\s x\ \&\allowbreak\ \sup(x) \notin x\}$
be the set of all closed subsets of some nonzero limit ordinal $\le\kappa$.

An operator $\Phi:\mathcal K(\kappa)\rightarrow\mathcal K(\kappa)$ is a \emph{postprocessing function}  if for every $x\in\mathcal K(\kappa)$:
\begin{enumerate}
\item  $\Phi(x)$ is a club in $\sup(x)$;
\item $\acc(\Phi(x)) \s \acc(x)$;
\item $\Phi(x)\cap\bar\alpha=\Phi(x\cap\bar\alpha)$ for every $\bar\alpha\in\acc(\Phi(x))$.
\end{enumerate}
\end{defn}
\begin{remark} By the first clause, $\otp(\Phi(x))\ge\cf(\sup(x))$, and by the second clause, $\otp(\Phi(x))\le\otp(x)$.
\end{remark}
It is easy to verify that the three operations we met in Section~\ref{sectionwarmingup} --- when their domains are restricted to $\mathcal K(\kappa)$ --- are postprocessing functions.
What's nice about postprocessing functions is that requirement~(iii) implies that they maintain \emph{coherence} features of $C$-sequences. 
Indeed, the theme of this section is to obtain club-guessing sequences which have additional coherence  features. The particular coherence features we consider can be found in the following definition and in Definition~\ref{defnstarcoherence} below.

\begin{defn} \label{defncoherence}
Let $\vec C=\langle C_\delta\mid\delta<\kappa\rangle$ be a $C$-sequence.
\begin{enumerate}
\item For an infinite cardinal $\chi\le\kappa$, we say that $\vec C$ is \emph{$\sqx$-coherent} iff for every $\delta<\kappa$ and $\bar\delta\in\acc(C_\delta)\cap E^\kappa_{\ge\chi}$,
it is the case that $C_{\bar\delta}=C_\delta\cap\bar\delta$;
\item We say that $\vec C$ is \emph{coherent} iff it is $\sqw$-coherent;
\item We say that $\vec C$ is \emph{weakly coherent}
iff for every $\alpha<\kappa$, $|\{ C_\delta\cap\alpha\mid \delta<\kappa\}|<\kappa$.
\end{enumerate}
\end{defn}
It is routine to verify that if $\Phi: \mathcal K(\kappa) \rightarrow \mathcal K(\kappa)$ is a postprocessing function, and $\langle C_\delta \mid \delta < \kappa \rangle$ satisfies any of the above coherence properties, then $\langle \Phi(C_\delta) \mid \delta< \kappa\rangle$ satisfies the same coherence property as well, Clause~(iii) being key in the verification.\footnote{Strictly speaking, $\Phi(C_{\gamma})$ is undefined for $\gamma\in\nacc(\kappa)$, so we make the ad hoc choice that 
$\Phi(C_0):=\emptyset$ and $\Phi(C_{\gamma+1}) := \{\gamma\}$ for all $\gamma<\kappa$ in order to ensure formal correctness of this statement.}

In particular, as a consequence of the use of postprocessing functions, in each of Lemmas \ref{weaktotail}, \ref{taillemma}, \ref{tailomega}, and \ref{lemma215}, if we start with a $C$-sequence with one of the coherence properties above, the exact same proof ensures that the guessing $C$-sequence obtained satisfies the same coherence property. 

As a concrete example, by \cite[Lemma~1.23]{paper29}, every transversal to a $\square_\xi(\kappa,{<}\mu)$-sequence with $\xi<\kappa$ or $\mu<\kappa$ gives an amenable $C$-sequence $\langle C_\delta \mid \delta \in S\rangle$ on a stationary $S \s \kappa$, which can then be supplied to the machinery in Lemma~\ref{lemma215}.
Indeed, in \cite{paper29}, solving Question~16 from \cite{paper_s01} in the affirmative,
a \emph{wide club guessing} theorem was proven using Lemma~\ref{lemma215}.\footnote{The statement of \cite[Lemma~2.5]{paper29} does not mention the parameter $\xi$,
however, its proof is an application of a postprocessing function and hence does not increase ordertypes.}
\begin{fact}[{\cite[Lemma~2.5]{paper29}}]\label{widecg} If $\square_\xi(\kappa,{<}\mu)$ holds for a regular cardinal $\kappa\ge\aleph_2$ and a cardinal $\mu<\kappa$,
then for every stationary $S\s\kappa$, $\square_\xi(\kappa,{<}\mu)$ may be witnessed by a sequence $\langle\mathcal C_\delta\mid\delta<\kappa\rangle$ 
with the added feature that for every club $D\s\kappa$, there exists $\delta\in S$ such that, for every $C\in\mathcal C_\delta$, $\sup(\nacc(C)\cap D)=\delta$.
\end{fact}
\begin{recall} $\square_\xi(\kappa,{<}\mu)$ asserts the existence of a sequence $\langle \mathcal C_\alpha\mid\alpha<\kappa\rangle$ satisfying all of the following:
\begin{itemize}
\item for every limit ordinal $\alpha<\kappa$, $1<|\mathcal C_\alpha|<\mu$, and each $C\in\mathcal C_\alpha$ is club in $\alpha$ with $\otp(C)\le\xi$;
\item for every $\alpha<\kappa$, $C\in\mathcal C_\alpha$, and $\bar\alpha\in\acc(C)$, $C\cap\bar\alpha\in\mathcal C_{\bar\alpha}$;
\item for every club $D$ in $\kappa$, there exists some $\alpha\in\acc(D)$ such that $D\cap\alpha\notin\mathcal C_\alpha$.
\end{itemize}
\end{recall}
\begin{remark} 
The instance $\square_\kappa(\kappa,{<}2)$ is better known as $\square(\kappa)$,
the instance $\square_\lambda(\lambda^+,{<}2)$ is better known as $\square_\lambda$,
and the instance $\square_\lambda(\lambda^+,{<}\lambda^+)$ is better known as $\square_\lambda^*$.
\end{remark}

Note that $\square_\lambda$ holds iff there exists a coherent $\lambda$-bounded $C$-sequence over $\lambda^+$,
and that $\square^*_\lambda$ holds iff there exists a weakly coherent $\lambda$-bounded $C$-sequence over $\lambda^+$.
The following terminology is also quite useful.
\begin{defn} A $C$-sequence $\vec C=\langle C_\delta\mid\delta<\lambda^+\rangle$ is a \emph{transversal for $\square^*_\lambda$}
iff it is $\lambda$-bounded and weakly coherent.
\end{defn}

A special case of Fact~\ref{widecg} states that if $\square(\kappa)$ holds and $\kappa\ge\aleph_2$,
then for every stationary $S\s\kappa$,
there exists a $\square(\kappa)$-sequence $\vec C$ such that 
$\vec C\restriction S$  witnesses $\cg(S,\kappa)$.
Replacing $\square(\kappa)$ by $\square_\lambda$, better forms of guessing are available:

\begin{fact}[{\cite[Corollary~2.4]{paper11}}] Suppose that $\lambda$ is an uncountable cardinal.

Then $\square_\lambda$ holds iff there exists a coherent $\lambda$-bounded $C$-sequence $\langle C_\delta\mid\delta<\lambda^+\rangle$
with the feature that
for every club $D\s\lambda^+$ and every $\theta\in\acc(\lambda)$,
there exists some $\delta<\lambda^+$ with $\otp(C_\delta)=\theta$ such that $C_\delta\s D$.
\end{fact}
\begin{remark}
Note that for a  $\vec C$ as above,
the map $\delta\mapsto\otp(C_\delta)$
yields a canonical partition of $\acc(\lambda^+)$
into $\lambda$-many pairwise disjoint stationary sets.
\end{remark}

An inspection of the proofs of \cite[Lemma~2.8]{paper11} and Proposition~\ref{prop217} makes it clear that the following holds true.

\begin{thm}\label{thm310} Suppose that $\lambda$ is an uncountable cardinal, and $S,T$ are stationary subsets of $\lambda^+$.
Suppose also that either $S\cap E^{\lambda^+}_{\neq\cf(\lambda)}$ is stationary or that $\Tr(S)$ is stationary.
Then:
\begin{enumerate}[(1)]
\item $\square_\lambda$ holds iff there exists a coherent $\lambda$-bounded $C$-sequence $\langle C_\delta\mid\delta<\lambda^+\rangle$
with the feature that
for every club $D\s\lambda^+$, there exists a $\delta\in S$ such that $\nacc(C_\delta)\s D\cap T$;
\item $\square^*_\lambda$ holds iff there exists a weakly coherent $\lambda$-bounded $C$-sequence $\langle C_\delta\mid\delta<\lambda^+\rangle$
with the feature that
for every club $D\s\lambda^+$, there exists a $\delta\in S$ such that $\nacc(C_\delta)\s D\cap T$.\qed
\end{enumerate}
\end{thm}

We now turn to present another postprocessing function.

\begin{lemma}[see {\cite[Lemma~4.9]{paper23}}] For every function $f:\kappa\rightarrow[\kappa]^{<\omega}$, 
the operator $\Phi_f:\mathcal K(\kappa)\rightarrow\mathcal K(\kappa)$ defined via:
$$\Phi_f(x):= x\cup \bigcup\{f(\gamma)\cap (\sup(x\cap \gamma), \gamma) \mid \gamma \in \nacc(x)\}$$
is a postprocessing function.\qed
\end{lemma}

The definition of $\Phi_f$ is motivated by the \emph{regressive functions ideal} $J[\kappa]$ from \cite{paper24},
and the following extension of it from \cite{paper51}.

\begin{defn}[\cite{paper51}]\label{paper51} $J_\omega[\kappa]$
stands for the collection of all subsets $S\s\kappa$ for which there exist a club $C\s\kappa$ and a sequence of functions  $\langle f_i:\kappa\rightarrow[\kappa]^{<\omega}\mid i<\kappa\rangle$ 
with the property that for every $\delta\in S\cap C$, every regressive function $f:\delta\rightarrow\delta$,
and every cofinal subset $\Gamma\s\delta$, there exists an $i<\delta$ such that $$\sup\{ \gamma\in \Gamma\mid f(\gamma)\in f_i(\gamma)\}=\delta.$$
\end{defn}

The next lemma gives a sufficient condition for moving from 
$\cg_\xi(S,\kappa)$ to $\cg_\xi(S,T)$.
\begin{lemma} 
Suppose that $\vec C=\langle C_\delta\mid\delta\in S\rangle$ is a $C$-sequence witnessing $\cg_\xi(S,\kappa)$.
If $S\in J_\omega[\kappa]$,
then for every stationary $T\s\kappa$, there exists a function $f:\kappa\rightarrow[\kappa]^{<\omega}$
such that $\langle \Phi_f(C_\delta)\mid \delta\in S\rangle$ witnesses $\cg_\xi(S,T)$
\end{lemma}
\begin{proof} Without loss of generality, $S\s\acc(\kappa)$.
Suppose that $S\in J_\omega[\kappa]$,
and fix a club $C\s\kappa$ and a sequence of functions  $\langle f_i:\kappa\rightarrow[\kappa]^{<\omega}\mid i<\kappa\rangle$ 
as in Definition~\ref{paper51}.
For all $i<\kappa$ and $\delta\in S$, denote $C_\delta^i:=\Phi_{f_i}(C_\delta)$, so that $\otp(C_\delta^i)\le\otp(C_\delta)$.

Let $T$ be an arbitrary stationary subset of $\kappa$.

    \begin{claim}\label{claim961}
    	There is an $i< \kappa$ such that $\langle C^i_\delta \mid \delta \in S\rangle$ witnesses $\cg(S,T)$.
    \end{claim}
\begin{why}
	Suppose not. For each $i< \kappa$, pick a club $D_i \s \kappa$ such that for every $\delta \in S$, 
	$$\sup(\nacc(C^i_\delta) \cap D_{i} \cap T)< \delta.$$
Consider the two clubs $D:= C\cap\diagonal_{i< \kappa} D_i$ and $D':=\acc^+(D\cap T)$.
	By the choice of $\vec C$, pick $\delta\in S$ such that 	$\Gamma:=\nacc(C_\delta)\cap D'$ is cofinal in $\delta$.
	As $\Gamma$ is a subset of $\nacc(C_\delta)\cap D'$, 
we may define a regressive function $f:\Gamma\rightarrow\delta$ via:
	$$f(\gamma):=\min\{ \beta\in D\cap T\mid \sup(C_\delta\cap\gamma)<\beta<\gamma\}.$$

As $\gamma\in S\cap C$, 
find $i<\delta$ such that $\Gamma':=\{ \gamma\in \Gamma\mid f(\gamma)\in f_i(\gamma)\}$ is cofinal in $\delta$.
By possibly omitting an initial segment of $\Gamma'$, 
we may assume that $\sup(C_\delta\cap\min(\Gamma'))>i$.
Recalling the definition of $D$, it follows that for every $\gamma\in\Gamma'$, $f(\gamma)\in D_i\cap T$.
So, for every $\gamma\in\Gamma'$, if we let $\beta:=\sup(C_\delta\cap\gamma)$,
then $C_\delta^i\cap(\beta,\gamma)$ is equal to the finite set $f_i(\gamma)\cap(\beta,\gamma)$ that contains $f(\gamma)$ which is an element of $D_i\cap T$.
So $\sup(\nacc(C^i_\delta) \cap D_{i} \cap T)= \delta$, contradicting the choice of $D_i$.
\end{why}
Let $i$ be given by the preceding claim. Then $f:=f_i$ is as sought.
\end{proof}
By \cite[Proposition~3.3]{paper51}, $J_\omega[\lambda^+]$ contains no stationary subsets of $E^{\lambda^+}_{\cf(\lambda)}$. In particular, $J_\omega[\omega_1]$ is empty. So, 
unlike Fact~\ref{widecg},
in the following we don't need to explicitly require $\kappa$ to be $\ge\aleph_2$.

\begin{cor} If $\square(\kappa)$ holds,
then for every stationary $S\in J_\omega[\kappa]$ and every stationary $T\s\kappa$,
there exists a $\square(\kappa)$-sequence $\vec C$ such that 
$\vec C\restriction S$  witnesses $\cg(S,T)$.\qed
\end{cor}

We have described a way for moving from $\cg(S,\kappa)$ to $\cg(S,T)$.
Our next goal is to describe a way for moving from $\cg(\kappa,T)$ to $\cg(S,T)$.
First, a definition.

\begin{defn} \label{defnstarcoherence}We say that a $C$-sequence $\vec C=\langle C_\delta\mid\delta<\kappa\rangle$ 
is \emph{$\sq^*$-coherent} iff for all $\delta<\kappa$ and $\bar\delta\in\acc(C_\delta)$,
it is the case that $\sup(C_{\bar\delta}\symdiff (C_\delta\cap \bar\delta)) < \bar \delta$.
\end{defn}
\begin{remark} Every $\sq^*$-coherent $C$-sequence is weakly coherent.
\end{remark}

To preserve $\sq^*$-coherence, one needs to consider a strengthening of Definition~\ref{defpp}.

\begin{defn}
An operator $\Phi:\mathcal K(\kappa)\rightarrow\mathcal K(\kappa)$ is a \emph{postprocessing$^*$ function}  if for every $x\in\mathcal K(\kappa)$,
Clause (i)--(iii) of Definition~\ref{defpp} hold true, and, in addition:
\begin{enumerate}
\item[(iv)] for every $\bar x\in\mathcal K(\kappa)$ such that $\bar x\sq^* x$ and $\sup(\bar x)\in\acc(\Phi(x))$, $\Phi(\bar x)\sq^*\Phi(x)$.
\end{enumerate}
\end{defn}

It is readily checked that all the postprocessing functions we have met so far are moreover postprocessing$^*$ functions.
An example of a postprocessing function that is not postprocessing$^*$ function may be found in \cite[Lemma~3.8]{paper28}.  
In fact, it is unknown at present how to obtain the same effect of that map using a postprocessing$^*$ function.

\begin{thm}\label{sqstartthm} Suppose that $\kappa\ge\aleph_2$,
and that there exists an $\sq^*$-coherent $C$-sequence witnessing $\cg(\kappa,T)$.
	
For every stationary $S\s \kappa$, there exists an $\sq^*$-coherent $C$-sequence over $\kappa$
whose restriction to $S$ witnesses  $\cg(S,T)$.
\end{thm}
	\begin{proof}The proof will be an adaptation of the proof of \cite[Theorem 4.13]{paper24}. Suppose towards a contradiction that $S$ is a counterexample.
Fix 	an $\sq^*$-coherent $C$-sequence $\vec e=\langle e_\delta\mid\delta<\kappa\rangle$  witnessing $\cg(\kappa,T)$.
By recursion on $i< \omega_1$, we construct a club $D_i\s\kappa$ and an $\sq^*$-coherent $C$-sequence $\vec{C^i}=\langle C^i_\alpha \mid \alpha< \kappa \rangle$. 
Our construction will have the property that for all $\alpha< \kappa$
and $j< i<\omega_1$, $C^j_\alpha \s C^i_\alpha$.

$\br$ For $i=0$, let $D_0 := \kappa$ and $\vec{C^0}:=\vec e$.

$\br$ For every $i<\omega_1$ such that a club $D_i\s\kappa$ and an $\sq^*$-coherent $C$-sequence $\vec{C^i}$ have been constructed, by the assumption we can find a club $F^{D_i}$ such that, for every $\delta \in S$, 
        $$\sup(\nacc(C^i_\delta) \cap F^{D_i} \cap T) < \delta.$$
        So let $D_{i+1}:= D_i \cap F^{D_i}$. As for constructing $\vec{C^{i+1}}$, we do this by recursion on $\alpha< \kappa$. To start, let $C^{i+1}_0:= \emptyset$, and for every $\alpha< \kappa$ let $C^{i+1}_{\alpha+1}:=\{\alpha\}$. 
        Next, for $\alpha\in\acc(\kappa)$, 
			$$C^{i+1}_\alpha:= C^i_\alpha \cup \bigcup \{C^{i+1}_\gamma \setminus\sup(C^i_\alpha\cap \gamma) \mid \gamma\in\nacc(C^i_\alpha)\setminus (F^{D_i} \cap T) \}.$$

			It is clear that $\vec{C^{i+1}}=\langle C^{i+1}_\alpha \mid \alpha< \kappa\rangle$ is $\sq^*$-coherent as well.

$\br$ For every $i\in\acc(\omega_1)$ such that $\langle (D_j,\vec{C^j})\mid j<i\rangle$ has been constructed as required, let $D_i:= \bigcap_{j< i} D_j$ 
and for every $\alpha< \kappa$, let $C^i_\alpha:= \bigcup_{j< i}C^j_\alpha$.

			\begin{claim} Let $\alpha<\kappa$. Then 
			$\acc^+(C^i_\alpha)= \bigcup_{j< i}\acc(C^j_\alpha)$.
			\end{claim}
			\begin{why} Let $\beta \in \acc^+(C^i_\alpha)$.			
				The sequence $\langle \min(C^j_\alpha\setminus\beta)\mid j< i\rangle$ is weakly decreasing, and as $i$ is a nonzero limit ordinal, it stabilizes at some $j^*< i$. 
				Let $\beta^+$ be this stable value (which may be equal to $\beta$).
				If there exists $j \in [j^*, i)$ such that $\beta^+\in \acc(C^{j}_\alpha)$, then in fact $\beta^+ = \beta$ and we can finish. So, suppose that this is not so. 
				If there exists $j \in [j^*, i)$ such that $\beta^+\notin F^{D_j}\cap T$, then $\beta^+ \in \acc(C^{j+1}_\alpha)$ and again we can finish. So suppose this is not so.
				Let $\beta^- : =\sup(C^{j^*}_\alpha \cap \beta^+)$ so that $\beta^-<\beta \leq \beta^+$. 
				Now examining the construction of $C^j_\alpha$ for $j \in (j^*, i)$, it is clear that for every $j\in[j^*, i)$,
				$$(\beta^-,\beta^+] \cap C^{j^*}_\alpha = (\beta^-,\beta^+] \cap C^{j}_\alpha.$$
				However, this implies that $\beta\notin\acc(C^i_\alpha)$, which is a contradiction.
			\end{why}
			
So 				$\vec{C^i}=\langle C^i_\alpha\mid \alpha< \kappa\rangle$ is a $C$-sequence.
Furthermore, since,  				$\vec{C^j}$ is $\sq^*$-coherent for every $j<i$, 
the preceding claim altogether implies that	$\vec{C^i}$ is $\sq^*$-coherent.

Consider the club $D:=\bigcap_{i< \omega_1}D_i$. As $\vec C$ witnesses $\cg(\kappa, T)$, the following set is stationary:
    $$B:= \{\beta< \kappa \mid \sup(\nacc(C_\beta)\cap D \cap T)=\beta\}.$$

Notice that by the nature of our recursive construction, for all $i< \omega_1$,
$$\nacc(C_\beta)\cap D \cap T) \s \nacc(C^i_\beta)\cap D \cap T,$$
and hence, for all $\beta\in B$ and $i<\omega_1$,
$$\sup(\nacc(C^i_\beta)\cap D \cap T)=\beta.$$
    
Pick $\delta\in S\cap\acc^+(B)$.
    For each $i< \omega_1$ the following ordinal is smaller than $\delta$
    $$\epsilon_i := \sup(\nacc(C^i_\delta)\cap F^{D_i} \cap T).$$
    
We now perform a case analysis to reach a contradiction.

\medskip
    
		\underline{\textbf{CASE 1}. $\cf(\delta)> \aleph_0$.} Let $\epsilon^*:= \sup_{i< \omega} \epsilon_i$, so that $\epsilon^*< \delta$. 
		Pick $\beta \in (\epsilon^*, \delta)\cap B$.
		For every $i< \omega$, let $\gamma_i:= \min(C^i_\delta\setminus \beta)$ so that $\langle \gamma_i \mid i< \omega\rangle$ is weakly decreasing. 
Then pick $i< \omega$ such that $\gamma_i= \gamma_{i+1}$. 

\underline{\textbf{SUBCASE 1.1}. $\gamma_i > \beta$.} It follows that $\gamma_i \in \nacc(C^i_\delta)$. As $\epsilon^*< \beta< \gamma_i$, we have that $\gamma_i \notin F^{D_i}\cap T$. 
It follows from our recursive construction that $C^{i+1}_\delta\cap [\beta, \gamma_i) = C^{i+1}_{\gamma_i} \cap [\beta, \gamma_i)$ and the latter set is nonempty so that $\gamma_{i+1} < \gamma_i$ which is a contradiction. 

		\underline{\textbf{SUBCASE 1.2}. $\gamma_i = \beta$ and $\beta \in \nacc(C^i_\delta)$.} So $\gamma_i$ is an element of $\nacc(C^i_\delta)$ above $\epsilon^*$,
		and we are back to Subcase~1.1.

\underline{\textbf{SUBCASE 1.3}. $\gamma_i = \beta$ and $\beta \in \acc(C^i_\delta)$.} In this case, by $\sq^*$-coherence, 
we have that $\sup((C^i_{\beta}\symdiff C^i_\delta)\cap \beta) < \beta$. This implies that 
$$\sup(\nacc(C^i_\delta)\cap \beta \cap D\cap T) = \beta.$$
But $D\s F^{D_i}$, contradicting the fact that $\epsilon_i< \beta$.

\medskip

		\underline{\textbf{SUBCASE 2}. $\cf(\delta)= \aleph_0$.} Find an uncountable $I \s \omega_1$ such that $$\epsilon^*:=\sup\{\max\{\epsilon_i, \epsilon_{i+1}\}\mid i\in I\}$$ is smaller than $\delta$. 
Pick $\beta \in (\epsilon^*, \delta)\cap B$. For every $i< \omega_1$, let $\gamma_i:= \min(C^i_\delta\setminus \beta)$ so that $\langle \gamma_i \mid i< \omega_1\rangle$ is weakly decreasing. 
Pick a large enough $i\in I$ such that $\gamma_i= \gamma_{i+1}$. 
		
		\underline{\textbf{SUBCASE 2.1}. $\gamma_i > \beta$.} 
Same as in Subcase~1.1.

		\underline{\textbf{SUBCASE 2.2}. $\gamma_i = \beta$ and $\beta \in \nacc(C^i_\delta)$.}  
Same as in Subcase~1.2.
				
		\underline{\textbf{SUBCASE 2.3}. $\gamma_i = \beta$ and $\beta \in \acc(C^i_\delta)$.} 
Same as in Subcase~1.3.		
	\end{proof}

\begin{cor}\label{cor318} If $\square(\kappa)$ holds and $J_\omega[\kappa]$ contains a stationary set,
then for all stationary subsets $S,T$ of $\kappa$,
there exists an $\sq^*$-coherent $C$-sequence $\vec C=\langle C_\delta\mid\delta<\kappa\rangle$ such that 
$\vec C\restriction S$  witnesses $\cg(S,T)$.\qed
\end{cor}
	
	Recalling Theorem~\ref{thm310},
we now turn to deal with stationary subsets of $E^{\lambda^+}_{\cf(\lambda)}$,
dividing the results into two, depending on whether $\lambda$ is regular or singular.

\begin{thm}\label{weaksquarethm} Suppose that $\lambda$ is a regular uncountable cardinal, and $\square^*_\lambda$ holds.

For every stationary $S\s E^{\lambda^+}_\lambda$, there exists a transversal $\vec C=\langle C_\delta\mid \delta<\lambda^+ \rangle$ for $\square^*_\lambda$ such that
$\vec C\restriction S$ witnesses $\cg_\lambda(S,E^{\lambda^+}_\lambda)$.
\end{thm}
\begin{proof} 
Suppose not, and fix a stationary $S\s E^{\lambda^+}_\lambda$ that constitutes a counterexample.
As $\square^*_\lambda$ holds, we may fix a  transversal 
$\vec e=\langle e_\delta\mid \delta<\lambda^+ \rangle$ for $\square^*_\lambda$.
We shall recursively construct a sequence $\langle (D_n,\vec{e^n},\vec{C^n})\mid n<\omega\rangle$ 
such that $D_n$ is a club in $\lambda^+$,
and $\vec{e^n}$ and $\vec{C^n}$ are transversals for $\square^*_\lambda$.

Set $D_0:=\acc(\lambda^+)$, $\vec{e^0}:=\vec e$ and $\vec{C^0}:=\vec e$.
Next, suppose that $n<\omega$ and that $\langle (D_j,\vec{e^j},\vec{C^j})\mid j\le n\rangle$ has already been successfully defined.
As $\vec{C^n}=\langle C_\delta^n\mid \delta<\lambda^+\rangle$ is a transversal for $\square^*_\lambda$,
by the choice of the stationary set $S$, it follows that we may pick a subclub $D_{n+1}\s D_n$ such that, for every $\delta\in S$,
$$\sup(\nacc(C_\delta^n)\cap D_{n+1}\cap E^{\lambda^+}_\lambda)<\delta.$$

Consider the postprocessing function $\Phi_{D_{n+1}}$ from Definition~\ref{dropdefn}.
For every $\delta\in\lambda^+\setminus S$, let $e_\delta^{n+1}:=e_\delta$ and $C_\delta^{n+1}:=e_\delta$.
For every $\delta\in S$, let 
$e_\delta^{n+1}:=\Phi_{D_{n+1}}(C^n_\delta)$, and then let
$$C_\delta^{n+1}:=e_\delta^{n+1}\cup\{ e_\gamma\setminus \sup(e_\delta^{n+1}\cap\gamma)\mid \gamma\in\nacc(e^{n+1}_\delta)\cap E^{\lambda^+}_{<\lambda}\}.$$

By \cite[Lemma~2.8]{paper32}, $\vec{e^{n+1}}:=\langle e_\delta^{n+1}\mid \delta<\lambda^+\rangle$ is again a transversal for $\square^*_\lambda$.
\begin{claim}\label{claim3111} $\vec{C^{n+1}}$ is a transversal for $\square^*_\lambda$.
\end{claim}
\begin{why} Since $\vec{e^{n+1}}$ is a $\lambda$-bounded $C$-sequence,
the definition of $\vec{C^{n+1}}$ makes it clear that it is as well.
Suppose that $\vec{C^{n+1}}$ is not weakly coherent,
and pick the least $\alpha<\lambda^+$ such that
$$|\{ C^{n+1}_\delta\cap\alpha\mid \delta<\lambda^+\}|=\lambda^+.$$
As $\vec{C^{n+1}}\restriction(\lambda^+\setminus S)=\vec{e^{n+1}}\restriction(\lambda^+\setminus S)$,
and $\vec{e^{n+1}}$ is a transversal for $\square^*_\lambda$, it follows that 
$$|\{ e^{n+1}_\delta\cap\alpha\mid \delta\in S\}|<|\{ C^{n+1}_\delta\cap\alpha\mid \delta\in S\}|=\lambda^+,$$
so we may fix $\Delta\in[S]^{\lambda^+}$ such that:
\begin{itemize}
\item $\delta\mapsto C^{n+1}_\delta\cap\alpha$ is injective over $\Delta$, but
\item $\delta\mapsto e^{n+1}_\delta\cap\alpha$ is constant over $\Delta$.
\end{itemize}
Fix $\epsilon<\alpha$ such that $\sup(e^{n+1}_\delta\cap\alpha)=\epsilon$ for all $\delta\in \Delta$. By minimality of $\alpha$, and by possibly shrinking $\Delta$ further, we may also assume that
\begin{itemize}
\item $\delta\mapsto C^{n+1}_\delta\cap\epsilon$ is constant over $\Delta$.
\end{itemize}
It thus follows from the definition of $\vec{C^{n+1}}$ that the map $\delta\mapsto C^{n+1}_\delta\cap[\epsilon,\alpha)$ is injective over $\Delta$, and
that, for every $\delta\in \Delta$, $C^{n+1}_\delta\cap[\epsilon,\alpha)=e_\gamma\cap[\epsilon,\alpha)$ for $\gamma:=\min(e_\delta^{n+1}\setminus(\epsilon+1))$.
In particular, $$|\{ e_\gamma\cap[\epsilon,\alpha)\mid \gamma<\lambda^+\}|=\lambda^+,$$
contradicting the fact that $\vec e$ is a transversal for a $\square^*_\lambda$-sequence.
\end{why}

This completes the construction of the sequence $\langle (D_n,\vec{e^n},\vec{C^n})\mid n<\omega\rangle$.
Now, let $D:=\bigcap_{n<\omega}D_n$. Pick $\delta\in S$ such that $\otp(D\cap\delta)=\omega^\delta>\lambda$.
Recall that, for every $n<\omega$, the following ordinal is smaller than $\delta$:
$$\epsilon_n:=\sup(\nacc(C_\delta^n)\cap D_{n+1}\cap E^{\lambda^+}_\lambda).$$

Since $\cf(\lambda)>\omega$,
for every $\alpha<\delta$, $\otp(\bigcup_{n<\omega}C_\delta^n\cap\alpha)<\lambda$.
So, $\otp(\bigcup_{n<\omega}C_\delta^n)=\lambda<\omega^\delta=\otp(D\cap \delta)$,
and we may fix $\beta\in D\setminus \bigcup_{n<\omega}C_\delta^n$ above $\sup_{n<\omega}\epsilon_n$.
Clearly, for each $n<\omega$, 
$\gamma_n:=\min(C_\delta^n\setminus\beta)$ is an element of $\nacc(C_\delta^n)$ above $\beta$.

Let $n<\omega$. Since $e^{n+1}_\delta=\Phi_{D_{n+1}}(C^n_\delta)$ and $\sup(D_{n+1}\cap\delta)=\delta$, 
$$e^{n+1}_\delta=\{\sup(D_{n+1}\cap \eta )\mid \eta \in C^n_\delta, \eta > \min(D_{n+1})\}.$$
In particular, $\sup(D_{n+1}\cap \gamma_n)\in e_\delta^{n+1}$. 
As $\gamma_n>\beta$ and $\beta\in D\s D_{n+1}$,
it is the case that $\beta\le \sup(D_{n+1}\cap \gamma_n)$. 
As $e^{n+1}_\delta\s C^{n+1}_\delta$, altogether, 
$$\beta <\gamma_{n+1}=\min(C^{n+1}_\delta\setminus\beta)\le\min(e^{n+1}_\delta\setminus\beta)=\sup(D_{n+1}\cap \gamma_n)\le\gamma_n.$$

Now, pick $n<\omega$ such that $\gamma_{n+1}=\gamma_n$. There are two options, each leads to a contradiction.

$\br$ If $\gamma_{n+1}\in e^{n+1}_\delta$, then since $e^{n+1}_\delta\s C^{n+1}_\delta$,
and $\gamma_{n+1}\in\nacc(C^{n+1}_\delta)$, $\gamma_{n+1}\in\nacc(e^{n+1}_\delta)$.
As, $\gamma_{n+1}\in e^{n+1}_\delta\s D_{n+1}\s D_0$, $\gamma$ is a limit ordinal.
Since $C_\delta^{n+1}\cap[\beta,\gamma_{n+1})$ is empty,
the definition of $C^{n+1}_\delta$ implies that $\cf(\gamma_{n+1})=\lambda$.
Altogether, $\gamma_{n+1}\in\nacc(C^{n+1}_\delta)\cap D_{n+1}\cap E^{\lambda^+}_\lambda$,
contradicting the fact that $\gamma_{n+1}>\beta>\epsilon_{n+1}$.

$\br$ If $\gamma_{n+1}\notin e^{n+1}_\delta$, then since $\gamma_{n+1}=\gamma_n\in C^n_\delta$,
the definition of $e^{n+1}_\delta$ implies that $\gamma_{n+1}<\sup(D_{n+1}\cap\gamma_n)$.
So, this time,
$$\gamma_{n+1}=\min(C^{n+1}_\delta\setminus\beta)<\min(e^{n+1}_\delta\setminus\beta)=\sup(D_{n+1}\cap \gamma_n)\leq\gamma_n,$$
contradicting the choice of $n$.
\end{proof}

In order to obtain a correct analogue of the preceding result,
we introduce the following natural strengthening of Definition~\ref{maindefn},
in which we replace the stationary set $T\s\kappa$ by a sequence 
$\vec T = \langle T_i \mid i< \theta\rangle$ of stationary subsets of $\kappa$.
\begin{defn}
$\cg_\xi(S, \vec T, \sigma,\vec J)$ asserts the existence
of a $\xi$-bounded $C$-sequence, $\vec C=\langle C_\delta\mid\delta \in S\rangle$ such that,
for every club $D\s\kappa$ there is $\delta \in S$ such that for every $i< \min\{\delta, \theta\}$,
    $$\{\beta < \delta \mid \suc_\sigma(C_\delta \setminus \beta) \subseteq D\cap T_i\}\in J_\delta^+.$$
\end{defn}

\begin{conv}
Convention~\ref{convention-omissions} applies to the above definition, as well.
\end{conv}

\begin{thm}\label{weaksquarethm2} 
Suppose that $\lambda$ is  a singular cardinal of uncountable cofinality,
and $\square^*_\lambda$ holds.
Let $\langle \lambda_i\mid i<\cf(\lambda)\rangle$ be the increasing enumeration of a club in $\lambda$.

For every stationary $S \s E^{\lambda^+}_{>\omega}$,
there exists a transversal $\vec C=\langle C_\delta\mid \delta<\lambda^+ \rangle$ for $\square^*_\lambda$ such that $\vec C\restriction S$ witnesses
$\cg_\lambda(S, \langle E^{\lambda^+}_{\geq \lambda_i}\mid i<\cf(\lambda)\rangle)$.
\end{thm}
\begin{proof} 
Suppose not, and fix a stationary $S\s E^{\lambda^+}_{>\omega}$ that constitutes a counterexample.
Without loss of generality, $\min(S)\ge\lambda$.
As $\square^*_\lambda$ holds, we may fix a  transversal 
$\vec e=\langle e_\delta\mid \delta<\lambda^+ \rangle$ for $\square^*_\lambda$.
As $\lambda$ is singular, we may assume that $\otp(e_\delta)<\lambda$ for every $\delta<\lambda^+$ (e.g., by appealing to \cite[Lemma~3.1]{paper29} with $\Sigma:=\{\lambda_i\mid i<\cf(\lambda)\}$).

Following the proof approach of Theorem~\ref{weaksquarethm},
we shall recursively construct a sequence $\langle (D_n,\vec{e^n},\vec{C^n})\mid n<\omega\rangle$ 
such that $D_n$ is a club in $\lambda^+$,
and $\vec{e^n}$ and $\vec{C^n}$ are transversals for $\square^*_\lambda$.

Set $D_0:=\acc(\lambda^+)$, $\vec{e^0}:=\vec e$ and $\vec{C^0}:=\vec e$.
Next, suppose that $n<\omega$ and that $\langle (D_j,\vec{e^j},\vec{C^j})\mid j\le n\rangle$ has already been successfully defined.
As $\vec{C^n}=\langle C_\delta^n\mid \delta<\lambda^+\rangle$ is a transversal for $\square^*_\lambda$,
by the choice of the stationary set $S$, it follows that we may pick a subclub $D_{n+1}\s D_n$ such that, for every $\delta\in S$,
for some $i_\delta^{n+1}<\cf(\lambda)$,
$$\sup(\nacc(C_\delta^n)\cap D_{n+1}\cap E^{\lambda^+}_{\ge\lambda_{i_\delta^{n+1}}})<\delta.$$

Consider the postprocessing function $\Phi_{D_{n+1}}$ from Definition~\ref{dropdefn}.
For every $\delta\in\lambda^+\setminus S$, let $e_\delta^{n+1}:=e_\delta$ and $C_\delta^{n+1}:=e_\delta$.
For every $\delta\in S$, let 
$e_\delta^{n+1}:=\Phi_{D_{n+1}}(C^n_\delta)$, and then let
$$C_\delta^{n+1}:=e_\delta^{n+1}\cup\{ e_\gamma\setminus \sup(e_\delta^{n+1}\cap\gamma)\mid \gamma\in\nacc(e^{n+1}_\delta)\cap E^{\lambda^+}_{<\lambda_{i_\delta^{n+1}}}\}.$$

By \cite[Lemma~2.8]{paper32}, $\vec{e^{n+1}}:=\langle e_\delta^{n+1}\mid \delta<\lambda^+\rangle$ is again a transversal for $\square^*_\lambda$.
By the exactly same proof of Claim~\ref{claim3111},
also $\vec{C^{n+1}}:=\langle C_\delta^{n+1}\mid \delta<\lambda^+\rangle$ is a transversal for $\square^*_\lambda$.
Furthermore, $\otp(e_\delta^{n+1})\le\otp(C_\delta^{n+1})<\lambda$ for all $\delta<\lambda^+$.

This completes the construction of the sequence $\langle (D_n,\vec{e^n},\vec{C^n})\mid n<\omega\rangle$.
Now, let $D:=\bigcap_{n<\omega}D_n$. Pick $\delta\in S$ such that $\otp(D\cap\delta)=\omega^\delta>\lambda$.
Recall that, for every $n<\omega$, the following ordinal is smaller than $\delta$:
$$\epsilon_n:=\sup(\nacc(C_\delta^n)\cap D_{n+1}\cap E^{\lambda^+}_{\ge\lambda_{i^{n+1}_\delta}}).$$

As $\cf(\lambda)>\omega$, $\otp(\bigcup_{n<\omega}C_\delta^n)<\lambda$,
so, we may fix $\beta\in D\setminus \bigcup_{n<\omega}C_\delta^n$ above $\sup_{n<\omega}\epsilon_n$.
Clearly, for each $n<\omega$, 
$\gamma_n:=\min(C_\delta^n\setminus\beta)$ is an element of $\nacc(C_\delta^n)$ above $\beta$.

Let $n<\omega$. Since $e^{n+1}_\delta=\Phi_{D_{n+1}}(C^n_\delta)$ and $\sup(D_{n+1}\cap\delta)=\delta$, 
$$e^{n+1}_\delta=\{\sup(D_{n+1}\cap \eta )\mid \eta \in C^n_\delta, \eta > \min(D_{n+1})\}.$$
In particular, $\sup(D_{n+1}\cap \gamma_n)\in e_\delta^{n+1}$. 
As $\gamma_n>\beta$ and $\beta\in D\s D_{n+1}$,
it is the case that $\beta\le \sup(D_{n+1}\cap \gamma_n)$. 
As $e^{n+1}_\delta\s C^{n+1}_\delta$, altogether, 
$$\beta<\gamma_{n+1}=\min(C^{n+1}_\delta\setminus\beta)\le\min(e^{n+1}_\delta\setminus\beta)=\sup(D_{n+1}\cap \gamma_n)\le\gamma_n.$$

Now, pick $n<\omega$ such that $\gamma_{n+1}=\gamma_n$. There are two options, each leads to a contradiction.

$\br$ If $\gamma_{n+1}\in e^{n+1}_\delta$, then since $e^{n+1}_\delta\s C^{n+1}_\delta$,
and $\gamma_{n+1}\in\nacc(C^{n+1}_\delta)$, $\gamma_{n+1}\in\nacc(e^{n+1}_\delta)$.
As, $\gamma_{n+1}\in e^{n+1}_\delta\s D_{n+1}\s D_0$, $\gamma$ is a limit ordinal.
So, since $C_\delta^{n+1}\cap[\beta,\gamma_{n+1})$ is empty,
the definition of $C^{n+1}_\delta$ implies that $\cf(\gamma_{n+1})\ge\lambda_{i^{n+1}_\delta}$.
Altogether, $\gamma_{n+1}\in\nacc(C^{n+1}_\delta)\cap D_{n+1}\cap E^{\lambda^+}_{\lambda_{i^{n+1}_\delta}}$,
contradicting the fact that $\gamma_{n+1}>\beta>\epsilon_{n+1}$.

$\br$ If $\gamma_{n+1}\notin e^{n+1}_\delta$, then since $\gamma_{n+1}=\gamma_n\in C^n_\delta$,
the definition of $e^{n+1}_\delta$ implies that $\gamma_{n+1}<\sup(D_{n+1}\cap\gamma_n)$.
So, this time,
$$\gamma_{n+1}=\min(C^{n+1}_\delta\setminus\beta)<\min(e^{n+1}_\delta\setminus\beta)=\sup(D_{n+1}\cap \gamma_n)\leq\gamma_n,$$
contradicting the choice of $n$.
\end{proof}

By applying the proof of Proposition~\ref{prop217} on the $C$-sequence produced by the preceding, 
we get a somewhat cleaner form of guessing, as follows.

\begin{cor}
Suppose that $\lambda$ is  a singular cardinal of uncountable cofinality,
and $\square^*_\lambda$ holds.
For every stationary $S \s E^{\lambda^+}_{>\omega}$,
there exists a transversal $\vec C=\langle C_\delta\mid \delta<\lambda^+ \rangle$ for $\square^*_\lambda$ satisfying the following:
\begin{itemize}
\item  for every $\delta<\lambda^+$, $\otp(C_\delta)<\lambda$;
\item for every club $D\s\lambda^+$,
there exists $\delta\in S$ such that $C_\delta\s D$ and
$\sup(\nacc(C_\delta)\cap D\cap E^{\lambda^+}_{>\mu})=\delta$ for every $\mu<\lambda$.\qed
\end{itemize}
\end{cor} 

\subsection{When coherence is not available}\label{subsectionnocoherence}
By waiving any coherence considerations, 
the proofs of Theorems \ref{weaksquarethm}   and \ref{weaksquarethm2} (together with Proposition~\ref{prop217}) yield, respectively, the general case of the introduction's Fact~\ref{shelahs21relative}
and a result from \cite{EiSh:819}.
\begin{fact}[{\cite[Claim~2.4]{Sh:365}}]\label{shelahs21relativeb} For every regular uncountable cardinal $\lambda$,
for every stationary $S\s E^{\lambda^+}_\lambda$, $\cg_\lambda(S,E^{\lambda^+}_\lambda)$ holds.
\end{fact}

\begin{fact}[{\cite[Theorem~2]{EiSh:819}}] For every singular cardinal $\lambda$ of uncountable cofinality
and every stationary $S \s E^{\lambda^+}_{\cf(\lambda)}$,
there exists a $\cf(\lambda)$-bounded $C$-sequence $\langle C_\delta\mid\delta\in S\rangle$ satisfying the following.

For every club $D\s\lambda^+$,
there exists $\delta\in S$ such that $C_\delta\s D$ and $\langle \cf(\gamma)\mid \gamma\in\nacc(C_\delta)\rangle$ is strictly increasing and converging to $\lambda$.
\end{fact}

Likewise, by changing the choice of the initial $C$-sequence $\vec e$ in the proof of Theorem~\ref{sqstartthm},
one obtains a proof of the following.

\begin{thm}[Shelah]\label{nonreflectingthm} Suppose that $R,S,T$ are stationary subsets of a regular cardinal $\kappa\ge\aleph_2$.
\begin{enumerate}[(1)]
\item If $T$ is a nonreflecting stationary set, then $\cg(S,T)$ holds;
\item If $R$ is a nonreflecting stationary set, then $\cg(R,T,\sigma)$ implies $\cg(S,T,\sigma)$.\qed
\end{enumerate}
\end{thm}

Forgetting about coherence, Corollary~\ref{cor318} has the following strong consequence.

\begin{cor} Suppose that $\square(\lambda^+)$ holds, and one of the following:
\begin{itemize}
\item $\lambda\ge\beth_\omega$;
\item $\lambda^{\aleph_0}=\lambda$;
\item $\lambda=\mathfrak b=\aleph_1$;
\item $\lambda\ge 2^{\aleph_1}$ and Shelah's Strong Hypothesis ($\ssh$) holds;
\item There exists an infinite regular cardinal $\theta$ such that $2^\theta\le\lambda<\theta^{+\theta}$.
\end{itemize}

Then $\cg(S,T)$ holds for all stationary subsets $S,T$ of $\lambda^+$.
\end{cor}
\begin{proof} By Corollary~5.1, Corollary~5.3 and Corollary~5.7
of \cite{paper51}, any of the above hypotheses imply that $J_\omega[\lambda^+]$ contains a stationary set.
\end{proof}

\section{Partitioned the club-guessing} \label{sectionpartitioning}
The theme of this section is partitioned club guessing as in Fact~\ref{shelahpartitioning}. The main definition is Definition~\ref{defnpartitioning}, where various types of partitions are considered. The bulk of our results, with the exception of Subsection~\ref{subsectionpartitioned}, however are about the stronger partitioning club-guessing. The difference being that in the former we obtain a $C$-sequence and a partition, whereas in the latter we are given the $C$-sequence in advance and then are given the task of partitioning it. 

In Subsection~\ref{subsectioncolourings}, we partition club-guessing sequences using colouring principles from \cite{paper47,paper53} which were in fact discovered while working on partitioning club-guessing. We show how these colouring principles allow for an abstract approach to partitioning club-guessing, separating the club-guessing technology from the combinatorial technology given to us by the relevant hypothesis. 

In Subsection~\ref{subsectionpartitioned}, we construct partitioned club-guessing using these colouring principles. Furthermore, we can obtain partitioned club-guessing sequences satisfying coherence features as well. 

In Subsection~\ref{partitioningapplications}, we list the results from \cite{paper47,paper53} under which the colouring principles can be obtained, and draw conclusions.

We conclude with a higher analogue of a combinatorial construction on $\aleph_1$ due to Moore from \cite{MR2444284}.
\begin{defn} \label{defnpartitioning}
For a $C$-sequence $\vec C=\langle C_\delta\mid \delta\in S\rangle$, we define three sets of cardinals:
\begin{itemize}
\item $\Theta_0(\vec C,T,\sigma,\vec J)$ denotes the set of all cardinals $\theta$ for which there exists a function $h:\kappa\rightarrow\theta$ satisfying the following.

For every club $D\s\kappa$, there exists $\delta\in S$ such that, for every $\tau<\theta$,
$$\{\beta<\delta\mid \suc_\sigma(C_\delta\setminus\beta)\s D\cap T\cap h^{-1}\{\tau\}\}\in J_\delta^+.$$

\item $\Theta_1(\vec C,T,\sigma,\vec J)$ denotes the set of all cardinals $\theta$ for which there exists a function $h:\kappa\rightarrow\theta$ satisfying the following.

For every club $D\s\kappa$, there exists $\delta\in S$ such that, for every $\tau<\theta$,
$$\{\beta<\delta\mid h(\otp(C_\delta\cap\beta))=\tau \ \&\  \suc_\sigma(C_\delta\setminus\beta)\s D\cap T\}\in J_\delta^+.$$
\item $\Theta_2(\vec C,T,\sigma,\vec J)$ denotes the set of all cardinals $\theta$ 
for which there exists a sequence of functions $\langle h_\delta:\delta\rightarrow\theta\mid \delta\in S\rangle$ satisfying the following.

For every club $D\s\kappa$, there exists $\delta\in S$ such that, for every $\tau<\theta$,
$$\{\beta<\delta\mid h_\delta(\beta)=\tau \ \&\  \suc_\sigma(C_\delta\setminus\beta)\s D\cap T\}\in J_\delta^+.$$
\end{itemize}
\end{defn}
\begin{conv}
Convention~\ref{convention-omissions} applies to the above definition, as well.
In addition, if we omit $T$, then $T:=\kappa$.
\end{conv}
\begin{remark} $\Theta_0(\vec C,T,\sigma,\vec J)\s \Theta_2(\vec C,T,\sigma,\vec J)$ and
$\Theta_1(\vec C,T,\sigma,\vec J)\s \Theta_2(\vec C,T,\sigma,\vec J)$.
\end{remark}
\begin{prop}
 For a stationary $S\s E^\kappa_\theta$
and a sequence $\vec C$ witnessing $\cg(S,T,\kappa)$, $\theta\in\Theta_2(\vec C,T,\sigma)$ for any choice of $\sigma<\theta$.\qed
\end{prop}

\begin{lemma}\label{amhocg} 
Suppose that $\vec C=\langle C_\delta\mid \delta\in S\rangle$ is a $C$-sequence witnessing $\cg_\xi(S,T)$.
For every $\theta\in\Theta_2(\vec C,T)$ such that $\alpha+\beta<\xi$ for all $(\alpha,\beta)\in\theta\times\xi$, there exists 
a $C$-sequence $\vec C^\bullet=\langle C^\bullet_\delta\mid \delta\in S\rangle$ witnessing $\cg_\xi(S,T)$
for which $\theta\in\Theta_1(\vec{C^\bullet},T)$.
\end{lemma}
\begin{proof} Without loss of generality, $S\s\acc(\kappa)$.
Suppose that $\theta\in\Theta_2(\vec C,T)$ is such that $\alpha+\beta<\xi$ for all $(\alpha,\beta)\in\theta\times\xi$.
Let $h:\kappa\rightarrow\theta$ be a surjection such that, for every $\epsilon<\kappa$,
$\{ h(\iota+1)\mid \epsilon<\iota<\epsilon+\theta\}=\theta$.
If $\theta$ is uncountable, we also require that $\{ h(\omega\cdot\iota)\mid \epsilon<\iota<\epsilon+\theta\}=\theta$ for every $\epsilon<\kappa$.
It follows that for each $(\gamma,\epsilon,\tau)\in\kappa\times\kappa\times\theta$, we may fix $y_{\gamma,\epsilon,\tau}$ such that:
\begin{itemize}
\item $y_{\gamma,\epsilon,\tau}$ is a closed nonempty subset of $E^\kappa_{<\theta}$;
\item $\min(y_{\gamma,\epsilon,\tau})=\gamma$;
\item $\max(y_{\gamma,\epsilon,\tau})<\gamma+\theta$;
\item $h(\epsilon+\otp(y_{\gamma,\epsilon,\tau})-1)=\tau$.
\end{itemize}

Fix a sequence $\langle h_\delta:\delta\rightarrow\theta\mid \delta\in S\rangle$ witnessing that $\theta\in\Theta_2(\vec C,T)$.
We now construct $\vec {C^\bullet}=\langle C_\delta^\bullet\mid\delta\in S\rangle$ as follows.
Given $\delta\in S$,
let $\langle\delta_i\mid i<\otp(C_\delta)\rangle$ denote the increasing enumeration of $\{0\}\cup C_\delta$.
Construct a sequence $\langle x^\delta_i\mid i<\otp(C_\delta)\rangle$ 
by recursion on $i<\otp(C_\delta)$, as follows:

$\br$ if $\delta_{i+1}<\delta_i+\theta$, then set $x_i^\delta:=\{\delta_i\}$.

$\br$ if $\delta_{i+1}\ge\delta_i+\theta$, then set $x_i^\delta:=y_{\gamma,\epsilon,\tau}$,
for $\gamma:=\delta_i$, $\epsilon:=\otp(\bigcup_{i'<i}x^\delta_{i'})$,
and $\tau:=h_\delta(\delta_{i})$. In particular, $h(\otp(\bigcup_{i'\le i}x^\delta_{i'})-1)=h_\delta(\delta_{i})$.

Finally, let $C^\bullet_\delta:=\bigcup_{i<\otp(C_\delta)}x^\delta_i$, so that $C^\bullet_\delta$ is a club in $\delta$. 
Note that $\otp(C^\bullet_\delta)\le\xi$, since  $\otp(C_\delta)\le\xi$ and $\alpha+\beta<\xi$ for all $(\alpha,\beta)\in\theta\times\xi$.
Thus, to see that $\vec C^\bullet:=\langle C^\bullet_\delta\mid\delta\in S\rangle$ and $h$ are as sought,
let $D$ be a club in $\kappa$.  By possibly shrinking $D$, we may assume that every element of $D$ is an indecomposable ordinal greater than $\theta$.

Pick $\delta\in S$ such that 
for every $\tau<\theta$, the following set is cofinal in $\delta$:
$$B_\tau:=\{\beta<\delta\mid h_\delta(\beta)=\tau \ \&\  \min(C_\delta\setminus(\beta+1))\in  D\cap T\}$$
Let $\tau<\theta$ and let $\beta\in B_\tau$. Pick $i<\otp(C_\delta)$ such that $\beta=\delta_i$. Put $\beta':=\max(x^\delta_i)$ so that
$\beta\le\beta'<\delta_{i+1}=\min(C^\bullet_\delta\setminus(\beta'+1))$.
Since $\delta_{i+1}\in D$, we know that $\delta_i+\theta<\delta_{i+1}$. Consequently,
$$h(\otp(C^\bullet_\delta\cap\beta'))=h(\otp(\bigcup\nolimits_{i'\le i}x^\delta_{i'})-1)=h_\delta(\delta_{i})=\tau,$$
as sought.
\end{proof}
\begin{remark}\label{rmk86} The preceding lemma should not be interpreted as
saying that $\Theta_1(\ldots)$ and $\Theta_2(\ldots)$ are essentially the same,
since the move from $\vec C$ to $\vec{C^\bullet}$ may lead to the loss of coherence features of $\vec C$.
In addition, the above lemma is limited to $\sigma=1$, though a simple tweak yields that 
if $\vec C$ is a witness for $\cg_\xi(S,T,\sigma,\langle J_\delta\mid\delta\in S\rangle)$
with $\sigma\le\omega$, $\sigma<\theta$, and $\nacc(\delta)\in J_\delta$ for all $\delta\in S$,
then a $\vec C^\bullet$ may be cooked-up to satisfy $\theta\in\Theta_1(\vec{C^\bullet},T,\sigma)$.
\end{remark}

\subsection{Using colourings}\label{subsectioncolourings}
We now introduce two colouring principles from \cite{paper47} which we shall use in this subsection.
As explained in \cite[Remark 8.2]{paper47},
these principles are a spin-off of Sierpi\'nski's \emph{onto mapping principle}.

\begin{defn}[\cite{paper47}]\label{def1} Let $J$ be an ideal over $\lambda$,
and $\theta\le\lambda$ be some cardinal.
\begin{itemize}
\item $\onto(J,\theta)$  asserts the existence of a colouring $c:[\lambda]^2\rightarrow\theta$.
such that for every $B\in J^+$, there is an $\eta<\lambda$ such that $$c[\{\eta\}\circledast B]=\theta;$$
\item  $\ubd(J,\theta)$  asserts the existence of an upper-regressive colouring $c:[\lambda]^2\rightarrow\theta$ 
such that for every $B\in J^+$, there is an $\eta<\lambda$ such that $$\otp(c[\{\eta\}\circledast B])=\theta.$$
\end{itemize}
\end{defn}

Our first application of which will make use of the following pumping up result.
\begin{fact}[{\cite[\S4]{paper53}}] \label{jbduparrow}Let $\theta\le\lambda$ be a pair of infinite cardinals, with $\lambda$ regular.
\begin{enumerate}[(1)]
\item If $\onto(J^{\bd}[\lambda],\theta)$ holds,
then there exists a colouring $c:[\lambda]^2\rightarrow\theta$ 
such that for every $\lambda$-complete ideal $J$ on some ordinal $\delta$ of cofinality $\lambda$ and every map $\psi:\delta\rightarrow\lambda$ satisfying $\sup(\psi[B])=\lambda$ for all $B\in J^+$,
the following holds.
For all $B\in J^+$, there exists an $\eta<\lambda$ such that 
$$\{ \tau<\theta\mid \{ \beta\in B\mid \eta<\psi(\beta)\ \&\ c(\eta,\psi(\beta))=\tau\}\in J^+\}=\theta.$$
\item If $\ubd(J^{\bd}[\lambda],\theta)$ holds,
then there exists a colouring $c:[\lambda]^2\rightarrow\theta$ 
such that for every $\lambda$-complete ideal $J$ on some ordinal $\delta$ of cofinality $\lambda$ and every map $\psi:\delta\rightarrow\lambda$ satisfying $\sup(\psi[B])=\lambda$ for all $B\in J^+$,
the following holds.
For all $B\in J^+$, there exists an $\eta<\lambda$ such that 
$$\otp(\{ \tau<\theta\mid \{ \beta\in B\mid \eta<\psi(\beta)\ \&\ c(\eta,\psi(\beta))=\tau\}\in J^+\})=\theta.$$
\end{enumerate}
\end{fact}

\begin{thm}\label{partitioningjkk} Suppose that $\vec C$ witnesses $\cg_\xi(S,T,\sigma,\vec J)$ with $S\s E^\kappa_\lambda$,
and that $\theta\le\lambda$ is infinite.
\begin{enumerate}[(1)]
\item If $\onto(J^\bd[\lambda],\theta)$ holds and $\xi=\lambda$, then $\theta\in\Theta_1(\vec C,T,\sigma,\vec J)$;
\item If $\onto(J^\bd[\lambda],\theta)$ holds, then $\theta\in\Theta_2(\vec C,T,\sigma,\vec J)$;
\item If $\ubd(J^\bd[\lambda],\theta)$ holds and  $\theta< \lambda$, then $\theta\in \Theta_2(\vec C,T,\sigma,\vec J)$.
\end{enumerate}
\end{thm}
\begin{proof} For every $\delta \in S$, fix a club $e_\delta$ in $\delta$ of ordertype $\lambda$. In case that $\xi=\lambda$,
moreover set $e_\delta:=C_\delta$.
As $J_\delta$ is a $\lambda$-complete ideal on $\delta$ extending $J^\bd[\delta]$,
once we define $\psi_\delta:\delta\rightarrow\lambda$ via $\psi_\delta(\beta):=\otp(e_\delta\cap\beta)$,
then $\sup(\psi_\delta[B])=\lambda$ for every $B \in (J_\delta)^+$.

\medskip

(1) and (2): Suppose that $\onto(J^\bd[\lambda],\theta)$ holds, and fix a colouring $c:[\lambda]^2\rightarrow\theta$ as in Fact~\ref{jbduparrow}(1).
\begin{claim}\label{4101} There exists an $\eta<\lambda$ such that, for every club $D\s\kappa$,
there exists a $\delta\in S$, such that, for every $ \tau<\theta$:
$$\{\beta<\delta\mid \eta < \psi_\delta(\beta)\ \&\ c(\eta,\psi_\delta(\beta))=\tau\ \&\ \suc_\sigma(C_\delta\setminus \beta)\s D\cap T\}\in J_\delta^+.$$
\end{claim}
\begin{why} Suppose not. For every $\eta<\lambda$, fix a counterexample club $D_\eta\s\kappa$. Let $D:=\bigcap_{\eta<\lambda}D_\eta$.
By the choice of $\vec C$, let us now pick $\delta\in S$ such that the following set is in $J_\delta^+$:
$$B:=\{\beta<\delta\mid \suc_\sigma(C_\delta\setminus \beta)\s D\cap T\}.$$
Recalling that $c$ was given by Fact~\ref{jbduparrow}(1), there is an $\eta< \lambda$ such that 
$$\{ \tau<\theta\mid \{ \beta\in B\mid \eta<\psi(\beta)\ \&\ c(\eta,\psi(\beta))=\tau\}\in (J_\delta)^+\}=\theta.$$
However, as $D \s D_\eta$, this contradicts the choice of $D_\eta$.
\end{why}

Let $\eta<\lambda$ be given by the preceding claim. Choose $\vec h=\langle h_\delta:\delta\rightarrow\theta\mid\delta\in S\rangle$ satisfying 
$h_\delta(\beta)=c(\eta,\psi_\delta(\beta))$ for every $\delta\in S$ and $\beta <\delta$ such that $\eta<\psi_\delta(\beta))$. 
Then $\vec h$ witnesses that $\theta\in\Theta_2(\vec C,T,\sigma,\vec J)$.
Furthermore, in the special case that $\vec C$ is $\lambda$-bounded, 
any map $h:\kappa\rightarrow\theta$ satisfying $h(\bar\beta)=c(\eta,\bar\beta)$ for every $\bar\beta\in(\eta,\lambda)$
witnesses that $\theta\in\Theta_1(\vec C,T,\sigma,\vec J)$.

\medskip

(3) Suppose that $\ubd(J^\bd[\lambda],\theta)$ holds
with $\theta< \lambda$, and fix a colouring $c:[\lambda]^2\rightarrow\theta$ as in Fact~\ref{jbduparrow}(2).
For every club $D\subseteq\kappa$, for all $\delta\in S$ and $\eta< \lambda$, denote
$$D(\eta, \delta):=\{ \tau<\theta\mid \{\beta<\delta\mid \eta< \psi_\delta(\beta)\ \&\ c(\eta,\psi_\delta(\beta))=\tau\ \&\ \suc_\sigma(C_\delta\setminus \beta)\s D\cap T\}\in J_\delta^+\}.$$
\begin{claim}\label{claim892} There exists an $\eta<\lambda$ such that, for every club $D\s\kappa$,
there exists a $\delta\in S$, such that $|D(\eta,\delta)| = \theta$.
\end{claim}
\begin{why} Suppose not. For every $\eta<\lambda$, fix a counterexample club $D_\eta\s\kappa$. Let $D:=\bigcap_{\eta<\lambda}D_\eta$.
By the choice of $\vec C$, let us now pick $\delta\in S$ such that the following set is in $J_\delta^+$:
$$B:=\{\beta<\delta\mid \suc_\sigma(C_\delta\setminus \beta)\s D\cap T\}.$$
Recalling that $c$ was given by Fact~\ref{jbduparrow}(2), there is an $\eta< \lambda$ such that 
$$\otp(\{ \tau<\theta\mid \{ \beta\in B\mid \eta<\psi(\beta)\ \&\ c(\eta,\psi(\beta))=\tau\}\in (J_\delta)^+\})=\theta.$$
So $|D(\eta, \delta)| = \theta$. However, as $D \s D_\eta$, and $D_\eta$ was chosen so that $|D_\eta(\eta, \delta)|<\theta$, we reach a contradiction.
\end{why}

We fix from hereon an $\eta< \lambda$ as given by the previous claim, and for simplicity of notation, for $D\s\kappa$ a club and $\delta \in S$, we denote $D(\delta):= D(\eta, \delta)$.

\begin{claim}\label{claim893}  There exists a club $D^* \s\kappa$ such that for every club $D \s D^*$, there exists $\delta \in S$ such that $D(\delta)= D^*(\delta)$ and this set has size $\theta$.
\end{claim}
\begin{why}
Suppose this is not so. In that case, we can construct a $\subseteq$-decreasing sequence $\langle D_i \mid i \le\theta^+\rangle$ of clubs in $\kappa$ as follows:
\begin{enumerate}
\item $D_0 := \kappa$;
\item $D_{i+1} \subseteq D_i$ is some club such that for every $\delta \in S$, either $|D(\delta)|<\theta$ or $D_{i+1}(\delta) \subsetneq D_i(\delta)$;
\item for $i \in \acc(\theta^++1)$, $D_i:=\bigcap_{i' < i}D_{i'}$.
\end{enumerate}

Since $D_{\theta^+}$ is again a club in $\kappa$, we may fix a $\delta\in S$ such that $|D_{\theta^+}(\delta)|=\theta$. 
In particular, for every $i \leq \theta^+$, $|D_{i}(\delta)|=\theta$, and hence, by the construction, 
$\langle D_i(\delta) \mid i \leq \theta^+\rangle$ must be a strictly  $\s$-decreasing sequence of subsets of $D_0(\delta)$,
contradicting the fact that $|D_0(\delta)|=\theta$.
\end{why}

Let $D^* \s\kappa$ be given by the preceding claim.
Then any sequence $\vec h=\langle h_\delta:\delta\rightarrow\theta\mid\delta\in S\rangle$ satisfying that for all $\delta\in S$ and $\beta <\delta$ with $\psi_\delta(\beta)>\eta$,
$$h_\delta(\beta)=\otp(c(\eta,\psi_\delta(\beta))\cap D^*(\delta))$$
witnesses that $\theta\in\Theta_2(\vec C,T,\sigma,\vec J)$.
\end{proof}

We now move on to the case of normal ideals. 
We first need an analogue of Fact~\ref{jbduparrow}. 
In what follows, for a set of ordinals $A$, its \emph{collapsing map} is the unique function $\psi:A\rightarrow\otp(A)$ satisfying $\psi(\alpha)=\otp(A\cap\alpha)$ for all $\alpha\in A$.

\begin{thm}\label{normaluparrow}
Suppose that $\lambda$ is a regular uncountable cardinal, and $\theta\le\lambda$ is a cardinal.
 \begin{enumerate}[(1)]
\item If $\ubd(\ns_\lambda,\theta)$ holds,
then there exists a colouring $c:[\lambda]^2\rightarrow\theta$ 
such that for every $\lambda$-complete normal ideal $J$ on some ordinal $\delta$ of cofinality $\lambda$,
for every club $A$ in $\delta$ of ordertype $\lambda$,
for its collapsing map $\psi:A\rightarrow\lambda$ the following holds.
For all $B\in J^+$, there exists an $\eta<\lambda$ such that 
$$\otp(\{ \tau<\theta\mid \{ \beta\in B\cap A\mid \eta<\psi(\beta)\ \&\ c(\eta,\psi(\beta))=\tau\}\in J^+\})=\theta.$$
\item If $\onto(\ns_\lambda,\theta)$ holds,
then there exists a colouring $c:[\lambda]^2\rightarrow\theta$ 
then for every normal $\lambda$-complete ideal $J$ on some ordinal $\delta$ of cofinality $\lambda$,
for every club $A$ in $\delta$ of ordertype $\lambda$,
for its collapsing map $\psi:A\rightarrow\lambda$ the following holds.
For all $B\in J^+$, there exists an $\eta<\lambda$ such that 
$$\{ \tau<\theta\mid \{ \beta\in B\cap A\mid \eta<\psi(\beta)\ \&\ c(\eta,\psi(\beta))=\tau\}\in J^+\}=\theta.$$
\end{enumerate}
\end{thm}
\begin{proof} Clauses (1) and (2) follow from \cite[Proposition~2.25]{paper47}.
Since the proof of Clause (2) was omitted in \cite{paper47}, we give it here.

(2) Given a normal $\lambda$-complete ideal $J$ on some ordinal $\delta$ of cofinality $\lambda$,
and a club $A$ in $\delta$ of ordertype $\lambda$, it is the case that $A\in J^*$, since $J$ is normal.
Let $\psi$ denote the collapsing map of $A$. 
For all $B \s \delta$, $\eta< \lambda$ and $\tau< \theta$, denote
$$B^{\eta,\tau}:=\{\beta \in B\cap A\mid \eta<\psi(\beta)\ \&\ c(\eta,\psi(\beta)) = \tau\}.$$

Suppose $\onto(\ns_\lambda, \theta)$ holds,
and fix a  witnessing colouring $c:[\lambda]^2\rightarrow\theta$.
Towards a contradiction, suppose that there exists $B\in J^+$ 
such that, for every $\eta < \lambda$, there is a $\tau_\eta< \theta$ such that $B^{\eta, \tau_\eta} \in J$. 
As $J$ is normal, $E:= \psi^{-1}[\diagonal_{\eta <\lambda}(\lambda \setminus \psi[B^{\eta,\tau_\eta}]]$ is in $J^*$.
Note that
$$E=\{\beta\in A\mid \forall \eta<\psi(\beta)\,(\beta\notin B^{\eta,\tau_\eta})\}.$$
As $E\in J^*$ and $B\in J^+$, $B\cap E\in J^+$, so since $J$ is normal, 
$\psi[B\cap E]$ is stationary. It thus follows from the choice of $c$
that we may pick $\eta< \lambda$ such that $c[\{\eta\}\circledast \psi[B\cap E]]=\theta$.
Find $\beta\in B\cap E$ such that $\psi(\beta)>\eta$ and  $c(\eta, \psi(\beta)) = \tau_\eta$. 
Then $\beta\in B^{\eta,\tau_\eta}$, contradicting the fact that $\beta\in E$.
\end{proof}

\begin{thm}\label{partitioningnormal}Suppose that $\lambda$ is a regular uncountable cardinal, $\vec C$ witnesses $\cg_\xi(S,T,\sigma,\vec J)$ with $S\s E^\kappa_\lambda$ and for every $\delta \in S$, $J_\delta$ is a normal $\lambda$-complete ideal on $\delta$ extending $J^\bd[\delta]$.
Then:
\begin{enumerate}[(1)]
\item If $\onto(\ns_\lambda,\theta)$ holds and $\xi=\lambda$, then $\theta\in\Theta_1(\vec C,T,\sigma,\vec J)$;
\item If $\onto(\ns_\lambda,\theta)$ holds, then $\theta\in\Theta_2(\vec C,T,\sigma,\vec J)$;
\item If $\ubd(\ns_\lambda,\theta)$ holds and  $\theta< \lambda$, then $\theta_2\in \Theta_2(\vec C,T,\sigma,\vec J)$.
\end{enumerate}
\end{thm}
\begin{proof} Write $\vec C$ as $\langle C_\delta\mid\delta\in S\rangle$.
For each $\delta\in S$, 
if $\otp(C_\delta)=\lambda$, then set $A_\delta:=C_\delta$. Otherwise,
just let $A_\delta$ be some club in $\delta$ of ordertype $\lambda$.
Then, let $\psi_\delta:A_\delta \rightarrow \lambda$ be the corresponding collapsing map.
We can now repeat the proof of Theorem~\ref{partitioningjkk} except that we use Theorem~\ref{normaluparrow} instead of Fact~\ref{jbduparrow}.
\end{proof}

\subsection{Maintaining coherence}\label{subsectionpartitioned}
By Theorem~\ref{partitioningjkk},
$\onto(J^\bd[\lambda],\theta)$ implies that $\theta\in \Theta_1(\vec C,T)$.
In contrast, $\ubd(J^\bd[\lambda],\theta)$ gives $\theta\in \Theta_2(\vec C,T)$,
and then Lemma~\ref{amhocg} only yields another $C$-sequence $\vec{C^\bullet}$
such that $\theta\in\Theta_1(\vec{C^\bullet},T)$.

In the next theorem, we combine the two results carefully in order to 
obtain a $C$-sequence $\vec{C^\bullet}$
with $\theta\in\Theta_1(\vec{C^\bullet},T)$ while maintaining some coherence features of the original sequence $\vec C$.

 \begin{thm}\label{812}  Suppose that $\theta<\lambda<\kappa$ are infinite cardinals,
$\ubd(J^\bd[\lambda],\theta)$ holds,
and $S$ is a stationary subset of $E^\kappa_\lambda$. 

For every $C$-sequence 
$\vec C=\langle C_\delta\mid\delta<\kappa\rangle$ such that $\vec C\restriction S$ witnesses $\cg_\lambda(S,T)$,
there exists a corresponding $C$-sequence 
$\vec{C^\bullet}=\langle C^\bullet_\delta\mid\delta<\kappa\rangle$ such that:
\begin{itemize} 
\item $\vec{C^\bullet}\restriction S$ is $\lambda$-bounded;
\item If $\vec C$ is weakly coherent, then so is $\vec{C^\bullet}$;
\item For every infinite cardinal $\chi\in[\theta,\kappa)$, if $\vec C$ is $\sqx$-coherent, then so is $\vec{C^\bullet}$;
\item $\theta\in \Theta_1(\vec{C^\bullet}\restriction S,T)$.
\end{itemize}
\end{thm}
\begin{proof} 
Without loss of generality, $0\in C_\delta$ for all nonzero $\delta<\kappa$.
For every $\delta<\kappa$, denote $\xi_\delta:=\otp(C_\delta)$, let $\psi_\delta:C_\delta\rightarrow\xi_\delta$ be the collapsing map of $C_\delta$,
and let $\langle\delta_i\mid i<\xi_\delta\rangle$ denote the increasing enumeration of $C_\delta$ so that $\psi_\delta(\delta_i) = i$ for every $i< \xi_\delta$.

Fix a colouring $c:[\lambda]^2\rightarrow\theta$ as in Fact~\ref{jbduparrow}(2).
For every club $D\subseteq\kappa$, for all $\delta\in S$ and $\eta< \lambda$, denote
$$D(\eta, \delta):=\{ \tau<\theta\mid \sup\{\beta<\delta\mid c(\eta,\psi_\delta(\beta))=\tau\ \&\ \min(C_\delta\setminus(\beta+1))\in D\cap T\}=\delta\}.$$
By Claims \ref{claim892} and \ref{claim893}, 
we may pick $\eta^*<\lambda$ and a club $D^*\s\kappa$ 
such that for every club $D \s D^*$, there exists $\delta \in S$ such that $D(\eta^*,\delta)= D^*(\eta^*,\delta)$ and this set has size $\theta$.
By possibly shrinking $D^*$, we may assume that $D^*$ consists of indecomposable ordinals, and that $\min(D^*)>\theta$. 

For every $\delta\in S$, since $\cf(\delta)=\lambda>\theta$, the set
$$N_\delta:=\{\beta<\delta\mid c(\eta^*,\psi_\delta(\beta))\notin D^*(\eta^*,\delta)\ \&\ \min(C_\delta\setminus(\beta+1))\in D^*\cap T\}$$
is bounded in $\delta$. So, by one more stabilization argument, we may fix an $\varepsilon<\lambda$ such that
every club $D \s D^*$, there exists $\delta \in S$ such that $D(\eta^*,\delta)= D^*(\eta^*,\delta)$, $|D^*(\eta^*,\delta)|=\theta$,
and also $$\sup(\{\psi_\delta(\beta)\mid \beta\in N_\delta\})=\varepsilon.$$

For all $\delta\in\acc(\kappa)$ and $i\le\xi_\delta$, denote:
$$T_\delta^i:=\{ c(\eta^*,j)\mid \varepsilon<j<i,\ \eta^*<j,\ \delta_{j+1}\in D^*\cap T\}.$$

Fix a surjection $h:\kappa\rightarrow\theta$ 
and a sequence of sets $\langle y_{\gamma,\epsilon,\tau}\mid (\gamma,\epsilon,\tau)\in\kappa\times\kappa\times\theta\rangle$
as in the proof of Lemma~\ref{amhocg}.
We now construct the new $C$-sequence $\vec {C^\bullet}=\langle C_\delta^\bullet\mid\delta<\kappa\rangle$, as follows.
Set $C^\bullet_0:=\emptyset$ and $C^\bullet_{\gamma+1}:=\{\gamma\}$ for every $\gamma<\kappa$.
Next, given $\delta\in\acc(\kappa)$,
construct a sequence $\langle x_\delta^i\mid i<\xi_\delta\rangle$ 
by recursion on $i<\xi_\delta$, as follows:

$\br$ If $\delta_{i+1}\notin D^*$, then set $x^i_\delta:=\{\delta_i\}$.

$\br$ If $\delta_{i+1}\in D^*$, then in particular, $\delta_{i+1}\ge\delta_i+\theta$, so we set $x_\delta^i:=y_{\gamma,\epsilon,\tau}$,
for $\gamma:=\delta_i$, $\epsilon:=\otp(\bigcup_{i'<i}x_\delta^{i'})$,
and $$\tau:=\otp(c(\eta^*,i)\cap T^i_\delta).$$

Note that, for every $\delta\in\acc(\kappa)$, $C_\delta\s C^\bullet_\delta$,
and also $\acc(C^\bullet_\delta)\cap E^\kappa_{\ge\theta}=\acc(C_\delta)$,
since $y_{\gamma,\epsilon,\tau}\s E^\kappa_{<\theta}$ for every $(\gamma,\epsilon,\tau)\in\kappa\times\kappa\times\theta$.
In addition, for every $\delta\in\acc(\kappa)$, if $\alpha+\beta<\xi_\delta$ for all $(\alpha,\beta)\in\theta\times\xi_\delta$,
then $\otp(C_\delta^\bullet)=\xi_\delta$.
In particular,  $\otp(C^\bullet_\delta)=\lambda$ for all $\delta\in S$.

\begin{claim} Let $\chi\in[\theta,\kappa)$ be an infinite cardinal.

If $\vec C$ is $\sqx$-coherent, then so is $\vec{C^\bullet}$.
\end{claim}
\begin{why} Suppose that $\vec C$ is $\sqx$-coherent. Let $\delta<\kappa$ and $\bar\delta\in\acc(C^{\bullet}_\delta)\cap E^\kappa_{\ge\chi}$;
we need to verify that $C^\bullet_\delta\cap\bar\delta\sq C^\bullet_\delta$.
As $\acc(C^\bullet_\delta)\cap E^\kappa_{\ge\theta}=\acc(C_\delta)$ and $\chi\ge\theta$, we infer that $\bar\delta\in\acc(C_\delta)$,
so by $\sqx$-coherence of $\vec C$, $C_\delta\cap\bar\delta\sq C_\delta$.
It follows that:
\begin{itemize}
\item $\langle \bar\delta_i\mid i<\xi_{\bar\delta}\rangle=\langle \delta_i\mid i<\xi_{\bar\delta}\rangle$,
\item $\langle T^i_{\bar\delta}\mid i<\xi_{\bar\delta}\rangle=\langle T^i_\delta\mid i<\xi_{\bar\delta}\rangle$, and hence
\item $\langle x^i_{\bar\delta}\mid i<\xi_{\bar\delta}\rangle=\langle x^i_\delta\mid i<\xi_{\bar\delta}\rangle$,
\end{itemize}
so $C^\bullet_{\delta}\cap\bar\delta=\bigcup_{i<\xi_{\bar\delta}}x^i_{\delta}=\bigcup_{i<\xi_{\bar\delta}}x^i_{\bar\delta}=C^\bullet_{\bar\delta}$, as sought.
\end{why}
\begin{claim} If $\vec C$ is weakly coherent, then so is $\vec{C^\bullet}$.
\end{claim}
\begin{why} Towards a contradiction, suppose that $\vec C$ is weakly coherent, but $\vec{C^\bullet}$ is not. Fix the least $\alpha<\kappa$
such that $|\{ C_\delta^\bullet\cap\alpha\mid \delta<\kappa\}|=\kappa$.
So we may fix a cofinal subset $\Delta$ of $\acc(\kappa)$ such that:
\begin{itemize}
\item[(1)] $\delta\mapsto C^\bullet_\delta\cap\alpha$ is injective over $\Delta$, but
\item[(2)] $\delta\mapsto C_\delta\cap\alpha$ is constant over $\Delta$.
\end{itemize}
Fix $\gamma<\alpha$ such that $\sup(C_\delta\cap\alpha)=\gamma$ for all $\delta\in\Delta$. 
By minimality of $\alpha$, and by possibly shrinking $\Delta$ further, we may also assume that
\begin{itemize}
\item[(3)] $\delta\mapsto C^\bullet_\delta\cap\gamma$ is constant over $\Delta$.
\end{itemize}
It thus follows that the map $\delta\mapsto C^\bullet_\delta\cap[\gamma,\alpha)$ is injective over $\Delta$.
However, for every $\delta\in\Delta$, $C^\bullet_\delta\cap[\gamma,\alpha)$
is equal to $y_{\gamma,\epsilon,\tau}\cap\alpha$,
for $\epsilon:=\otp(C^\bullet_\delta\cap\gamma)$
and some $\tau<\theta$.
Recalling Clause~(3), there exists an $\epsilon<\kappa$ such that:
$$\{ C^\bullet_\delta\cap[\gamma,\alpha)\mid \delta\in\Delta\}\s\{y_{\gamma,\epsilon,\tau}\cap\alpha\mid \tau<\theta\},$$
contradicting the fact that the set on the right hand size has size $\le\theta<\kappa$.
\end{why}

Finally, to see that $\theta\in \Theta_1(\vec{C^\bullet}\restriction S,T)$,
let $D$ be a club in $\kappa$.  By possibly shrinking $D$, we may assume that $D\s D^*$. 
Pick $\delta \in S$ such that $D(\eta^*,\delta)= D^*(\eta^*,\delta)$, $|D^*(\eta^*,\delta)|=\theta$,
and also $$\sup(\{\psi_\delta(\beta)\mid \beta\in N_\delta\})=\varepsilon.$$
For any $D'\in\{D,D^*\}$,
$$D'(\eta^*, \delta)=\{ \tau<\theta\mid \sup\{i<\lambda\mid c(\eta^*,i)=\tau\ \&\ \delta_{i+1}\in D'\cap T\}=\lambda\}.$$
So, since $D(\eta^*,\delta)= D^*(\eta^*,\delta)$, the definition of $\varepsilon$ implies that
$$D(\eta^*, \delta)=\{ c(\eta^*,j)\mid \varepsilon<j<\lambda,\ \eta^*<j,\ \delta_{j+1}\in D^*\cap T\}=T_\delta^{\lambda}.$$
In particular, $|T_\delta^{\lambda}|=\theta$.
Now, given a prescribed colour $\tau^*$ and some $\alpha<\delta$,
we shall find a $\beta^*\in C^\bullet_\delta$ above $\alpha$ such that $\min(C^\bullet_\delta\setminus(\beta^*+1))\in D\cap T$ and $h(C^\bullet_\delta\cap\beta^*)=\tau^*$.
Fix the unique $\tau\in T^\lambda_\delta$ such that $\otp(T^\lambda_\delta\cap\tau)=\tau^*$.
Note that for a tail of $i<\lambda$, $T_\delta^\lambda\cap\tau=T_\delta^i\cap\tau$.

As $\tau\in D(\eta^*,\delta)$, there are cofinally many $\beta\in C_\delta$ such that $\min(C_\delta\setminus(\beta+1))\in D\cap T$,
$\eta^*<\psi_\delta(\beta)$ and $c(\eta^*,\psi_\delta(\beta))=\tau$. So, we may find  a large enough $i<\lambda$ such that:
\begin{itemize}
\item $\delta_{i+1}\in D\cap T$,
\item $\eta^*<i$,
\item $c(\eta^*,i)=\tau$, 
\item $\alpha<\delta_i$, and
\item $T_\delta^\lambda\cap\tau=T_\delta^i\cap\tau$.
\end{itemize}

So $\otp(c(\eta^*,i)\cap T^i_\delta)=\otp(\tau\cap T^i_\delta)=\otp(T^\lambda_\delta\cap\tau)=\tau^*$.
Put $\beta^*:=\max(x_\delta^i)$ so that
$\delta_i\le\beta^*<\delta_{i+1}=\min(C^\bullet_\delta\setminus(\beta^*+1))$.
Since $\delta_{i+1}\in D\s D^*$, we know that 
$x_\delta^i=y_{\gamma,\epsilon,\tau^*}$,
for $\gamma:=\delta_i$ and $\epsilon:=\otp(\bigcup_{i'<i}x_\delta^{i'})$.
Consequently,
$$h(\otp(C^\bullet_\delta\cap\beta^*))=h(\otp(\bigcup\nolimits_{i'\le i}x^\delta_{i'})-1)=\tau^*,$$
as sought.
\end{proof}

By \cite[Proposition~2.18 and Lemma~8.4]{paper47}, 
in G{\"o}del's constructible universe $L$,
for every weakly compact cardinal $\lambda$ that is not ineffable,
$\ubd(J^{\bd}[\lambda],\theta)$ fails for every cardinal $\theta\in[3,\lambda]$, but $\onto(\ns_\lambda,\lambda)$ holds.
Thus, it is easier to get $\ubd(J,\theta)$ with $J:=\ns_\lambda$ than with $J:=J^{\bd}[\lambda]$.
In the upcoming theorem, the hypothesis of 
$\ubd(J^{\bd}[\lambda],\theta)$ from Theorem~\ref{812} is reduced to $\ubd(\ns_\lambda,\theta)$
at the cost of requiring $\vec C\restriction S$ to witness $\cg_\lambda(S,T,1,\allowbreak\langle \ns_\delta\mid\delta\in S\rangle)$.

\begin{thm} Suppose that $\theta<\lambda<\kappa$ are infinite cardinals,
$\ubd(\ns_\lambda,\theta)$ holds,
and $S$ is a stationary subset of $E^\kappa_\lambda$. 

For every $C$-sequence 
$\vec C=\langle C_\delta\mid\delta<\kappa\rangle$ such that $\vec C\restriction S$ witnesses $\cg_\lambda(S,T,1,\allowbreak\langle \ns_\delta\mid\delta\in S\rangle)$,
there exists a corresponding $C$-sequence 
$\vec{C^\bullet}=\langle C^\bullet_\delta\mid\delta<\kappa\rangle$ such that:
\begin{itemize} 
\item $\vec{C^\bullet}\restriction S$ is $\lambda$-bounded;
\item If $\vec C$ is weakly coherent, then so is $\vec{C^\bullet}$;
\item For every infinite cardinal $\chi\in[\theta,\kappa)$, if $\vec C$ is $\sqx$-coherent, then so is $\vec{C^\bullet}$;
\item $\theta\in \Theta_1(\vec{C^\bullet}\restriction S,T)$.
\end{itemize}
\end{thm}
\begin{proof} 
For every $\delta<\kappa$, let $\psi_\delta:C_\delta\rightarrow\otp(C_\delta)$ denote the collapsing map of $C_\delta$.
Let $c:[\kappa]^2\rightarrow\theta$ be a colouring given by Theorem~\ref{normaluparrow}(1).
\begin{claim} There exists an $\eta<\lambda$ such that, for every club $D\s\kappa$,
there exists a $\delta\in S$, such that the following set has size $\theta$:
$$D(\eta, \delta):=\{ \tau<\theta\mid \sup\{\beta<\delta\mid c(\eta,\psi_\delta(\beta))=\tau\ \&\ \min(C_\delta\setminus(\beta+1))\in D\cap T\}=\delta\}.$$
\end{claim}
\begin{why} Suppose not. For every $\eta<\lambda$, fix a counterexample club $D_\eta\s\kappa$. Let $D:=\bigcap_{\eta<\lambda}D_\eta$.
By the hypothesis on $\vec C\restriction S$, let us now pick $\delta\in S$ such that the following set is stationary in $\delta$:
$$B:=\{\beta<\delta\mid \min(C_\delta\setminus(\beta+1))\in D\cap T\}.$$
By the choice of $c$, there is an $\eta< \lambda$ such that 
$$\otp(\{ \tau<\theta\mid \{ \beta\in B\mid \eta<\psi_\delta(\beta)\ \&\ c(\eta,\psi_\delta(\beta))=\tau\}\in (\ns_\delta)^+\})=\theta.$$
In particular, $|D(\eta, \delta)| = \theta$, contradicting the fact that $D \s D_\eta$.
\end{why}
The rest of the proof is now identical to that of Theorem~\ref{812}. 
\end{proof}
\begin{remark}
Since $\nacc(\delta)\in \ns_\delta$ for all $\delta\in S$,
for every $\sigma\le\omega$ such that $\sigma<\theta$,
if $\vec C\restriction S$ moreover witnesses $\cg_\lambda(S,T,\sigma,\allowbreak\langle \ns_\delta\mid\delta\in S\rangle)$,
then the preceding proof may be slightly tweaked to secure that 
$\theta\in \Theta_1(\vec{C^\bullet}\restriction S,T,\sigma)$.
The first change is to require that the surjection $h:\kappa\rightarrow\theta$ satisfies that for every $\epsilon<\kappa$,
for every $\tau<\theta$, there exists $\iota\in(\epsilon,\epsilon+\theta)$ such that
$\{ h(\iota+\varsigma)\mid \varsigma\le\sigma\}=\{\tau\}$. The second change is to impose $x^i_\delta=\{\delta_i\}$ for all $i\in\nacc(\xi_\delta)$.
The details are left to the reader.
\end{remark}

\subsection{Applications}\label{partitioningapplications}

We now utilize the results from \cite{paper47,paper53} in order to partition club-guessing sequences.

\begin{fact}[\cite{paper47}] Suppose that $\lambda$ is regular and uncountable. 

Any of the following implies that $\onto(J^\bd[\lambda], \theta)$ holds:
\begin{enumerate}[(1)]
\item $\theta=\lambda=\aleph_1=\non(\mathcal M)$;
\item $\theta=\lambda$ is a successor cardinal, and $\stick(\lambda)$ holds;
\item $\theta=\lambda$ and $\diamondsuit(T)$ holds for a stationary $T\s\lambda$ that does not reflect at regulars;
\item $\theta<\lambda$ and $\lambda\nrightarrow[\lambda]^2_\theta$;
\item $\theta<\lambda$ is regular and $\ubd(J^\bd[\lambda], \lambda)$ holds.
\end{enumerate}
\end{fact}
 
\begin{fact}[\cite{paper53}]\label{fact517} Suppose that $\lambda$ is regular and uncountable, and $\theta\le\lambda$. 

Any of the following implies that $\ubd(J^\bd[\lambda], \theta)$ holds:
\begin{enumerate}[(1)]
\item $\lambda$ admits a nontrivial $C$-sequence in the sense of \cite[Definition~6.3.1]{TodWalks};
\item $\square(\lambda,{<}\mu)$ holds for some $\mu<\lambda$;
\item $\lambda$ is not greatly Mahlo;
\item $\lambda$ is not weakly compact in $L$;
\item $\lambda$ is not weakly compact, and $\theta=\omega$;
\item $\lambda$ is not strongly inaccessible, and $\theta=\log_2(\lambda)$.
\end{enumerate}
\end{fact}

\begin{cor}\label{cor418} Suppose that $\lambda$ is a regular uncountable cardinal,
and $\vec C$ witnesses $\cg_\lambda(S,T,\sigma,\vec J)$ with $S\s E^\kappa_\lambda$.
\begin{enumerate}[(1)]
\item If $\lambda=\lambda^{<\lambda}=\theta$ is a successor cardinal, then $\lambda\in\Theta_1(\vec C,T,\sigma,\vec J)$;
\item If $\lambda=\theta^+$ and $\theta$ is regular, then $\theta\in\Theta_1(\vec C,T,\sigma,\vec J)$;
\item If $\lambda$ is not Mahlo and $\diamondsuit(\lambda)$ holds, then $\lambda\in\Theta_1(\vec C,T,\sigma,\vec J)$;
\item If $\lambda$ is not greatly Mahlo then $\theta\in\Theta_2(\vec C,T,\sigma,\vec J)$ for every cardinal $\theta<\lambda$;
\item If $\lambda$ is not strongly inaccessible, then $\log_2(\lambda)\in\Theta_2(\vec C,T,\sigma,\vec J)$;
\item If $\lambda$ is not weakly compact, then $\omega\in\Theta_2(\vec C,T,\sigma,\vec J)$.\qed
\end{enumerate}
\end{cor}

\begin{fact}[\cite{paper47}]\label{normalsampleresults}
Let $\lambda$ be a regular uncountable cardinal.
\begin{enumerate}[(1)]
\item if $\diamondsuit^*(\lambda)$ holds, then so does $\onto(\ns_\lambda, \lambda)$;
\item If $\lambda$ admits an amenable $C$-sequence, then $\onto(\ns_\lambda, \theta)$ holds for all regular $\theta<\lambda$;
\item $\ubd(\ns_\lambda, \omega)$ holds iff $\lambda$ is not ineffable.
\end{enumerate} 
\end{fact}

\begin{cor}\label{cor420} Suppose that $\lambda$ is a regular uncountable cardinal,
and $\vec C$ witnesses $\cg_\lambda(S,T,\sigma,\vec J)$ with $S\s E^\kappa_\lambda$ and $\vec J$ is a sequence of normal ideals.
\begin{enumerate}[(1)]
\item If $\diamondsuit^*(\lambda)$ holds, then $\lambda\in\Theta_1(\vec C,T,\sigma,\vec J)$;
\item If $\lambda$ admits an amenable $C$-sequence, then $\lambda\in\Theta_1(\vec C,T,\sigma,\vec J)$ for all regular $\theta<\lambda$;
\item If $\lambda$ is not ineffable, then $\omega\in\Theta_2(\vec C,T,\sigma,\vec J)$.\qed
\end{enumerate}
\end{cor}

In \cite{MR2444284},
the (weak) club-guessing principle $\mho$ was shown to give rise to a $C$-sequence $\vec C$ over $\omega_1$
for which the corresponding object $\mathcal T(\rho_0^{\vec C})$ is a special Aronszajn tree of pathological nature.
In the terminology developed in this paper, the key features of $\vec C$ sufficient for the construction are that 
$\vec C$ be a transversal for $\square^*_\omega$ and that $\omega\in\Theta_1(\vec C,\omega_1)$.
Arguably, the higher analog of this would assert the existence of a transversal $\vec C=\langle C_\delta\mid\delta<\lambda^+\rangle$ for $\square^*_{\lambda}$ such that $\log_2(\lambda)\in\Theta_1(\vec C,\lambda^+)$.
By the next corollary, if $\lambda$ is not strongly inaccessible (in particular, if $\lambda=\aleph_1$), then the existence of such  a $C$-sequence is in fact no stronger than $\square^*_{\lambda}$.

\begin{cor} Suppose that $\lambda$ is a regular uncountable cardinal, and $\theta=\log_2(\lambda)$.
\begin{enumerate}[(1)]
\item If $\square_{\lambda}^*$ holds and $\theta<\lambda$,
then  there exists a transversal $\vec C$ for $\square_\lambda^*$ such that $\theta\in \Theta_1(\vec C\restriction E^{\lambda^+}_{\lambda},E^{\lambda^+}_{\lambda})$;
\item If $\diamondsuit(\lambda)$ holds and $\lambda$ is not Mahlo, then  there exists a transversal $\vec C$ for $\square_\lambda^*$ such that
$\theta\in \Theta_1(\vec C\restriction E^{\lambda^+}_{\lambda},E^{\lambda^+}_\lambda)$;
\item If $\diamondsuit^*(\lambda)$ holds, then  there exists a transversal $\vec C$ for $\square_\lambda^*$ such that
$\theta\in \Theta_1(\vec C\restriction E^{\lambda^+}_{\lambda},\lambda^+)$.
\end{enumerate}
\end{cor}
\begin{proof} (1) Suppose that $\theta<\lambda$ and $\square^*_\lambda$ holds.
Appeal to Theorem~\ref{weaksquarethm} to find a transversal $\vec C$ for $\square_\lambda^*$ such that $\vec C\restriction E^{\lambda^+}_{\lambda}$  witnesses $\cg_\lambda(E^{\lambda^+}_{\lambda},E^{\lambda^+}_{\lambda})$.
By Corollary~\ref{cor418}(5), $\theta\in \Theta_1(\vec C\restriction E^{\lambda^+}_{\lambda},E^{\lambda^+}_{\lambda})$.

(2) Appeal to Fact~\ref{shelahs21relativeb} 
to pick a $\lambda$-bounded $C$-sequence $\vec C=\langle C_\delta\mid\delta<\lambda^+\rangle$ 
such that $C\restriction E^{\lambda^+}_\lambda$ witnesses $\cg_\lambda(E^{\lambda^+}_\lambda,E^{\lambda^+}_\lambda)$.
Suppose that $\diamondsuit(\lambda)$ holds. 
In particular, $\lambda^{<\lambda}=\lambda$, so $\vec C$ is trivially weakly coherent.
In addition, since $\diamondsuit(\lambda)$ holds, Corollary~\ref{cor418}(3) implies that $\lambda\in\Theta_1(E^{\lambda^+}_\lambda,E^{\lambda^+}_\lambda)$.

(3) By Fact~\ref{shelahs21relativeb} together with Corollary~\ref{thm42b} below,
we may fix a $\lambda$-bounded $C$-sequence $\vec C=\langle C_\delta\mid\delta<\lambda^+\rangle$ such that
$\vec C\restriction E^{\lambda^+}_\lambda$ witnesses $\cg_\lambda(E^{\lambda^+}_{\lambda},\lambda^+,1,\langle\ns_\delta\mid\delta\in S\rangle)$.
Suppose that $\diamondsuit^*(\lambda)$ holds. 
In particular, $\vec C$ is trivially weakly coherent.
In addition, since $\diamondsuit^*(\lambda)$ holds, Corollary~\ref{cor420}(1) implies that $\lambda\in\Theta_1(\vec C\restriction E^{\lambda^+}_\lambda,\lambda^+)$.
\end{proof} 

We also record the following variation.
\begin{cor} Suppose that $\lambda$ is a regular uncountable cardinal that is not strongly inaccessible. 
If $\square_\lambda$ holds,
then it may be witnessed by a $C$-sequence $\vec C$ such that $\log_2(\lambda)\in \Theta_1(\vec C\restriction E^{\lambda^+}_{\lambda},\lambda^+)$.
\end{cor}
\begin{proof} By Fact~\ref{widecg}, using $(\xi,\kappa,\mu,S):=(\lambda,\lambda^+,2,E^{\lambda^+}_\lambda)$,
and Corollary~\ref{cor418}(5).
\end{proof}

\section{Increasing $\sigma$}\label{sectionsigma}
In this section we are interested in improving the quality of guessing by guessing many consecutive nonaccumulation points as in Question~\ref{shelahquestion}. As we shall see, guessing clubs relative to points of prescribed cofinality turns out be of great help for this purpose. 
The main result of this section reads as follows.
\begin{cor}\label{cor51} Suppose $\nu<\xi\le\kappa$ are infinite cardinals,
and that $S\s E^\kappa_{>\nu}$.
	
	For every (possibly finite) cardinal $\sigma<\nu$,
	if $\cg_\xi(S, E^\kappa_{\ge\sigma}\cap E^\kappa_{\le\nu},1,\vec J)$ holds, then so does $\cg_\xi(S, \kappa, \sigma,\vec J)$.
\end{cor}
\begin{proof}
$\br$	If $\sigma$ is finite, then appeal to Theorem~\ref{thm42a} below.

$\br$	If $\sigma=\omega$ and $\cg_\xi(S, E^\kappa_{\sigma},1,\vec J)$ holds then appeal to Theorem~\ref{thma} below.

$\br$	If $\sigma=\omega$ and $\cg_\xi(S, E^\kappa_{\sigma},1,\vec J)$ fails,
then since $S\s E^\kappa_{>\nu}$, it is the case that $\cg_\xi(S, E^\kappa_{\bar\nu},1,\vec J)$ holds for some cardinal $\bar\nu$ with $\sigma<\bar\nu\le\nu$.
Now appeal to Theorem~\ref{thm45}(2) below with $T:=\kappa$.

$\br$	If $\sigma>\omega$, then appeal to Theorem~\ref{thm45} below with $T:=\kappa$.
\end{proof}

Another key result is Theorem~\ref{thm45} where a relative version of the main result is proved.

We commence with a result that pumps up $\sigma$ from $1$ to any prescribed positive integer
using a postprocessing$^*$ function (thereby, preserving coherence features).

\begin{thm}\label{thm42a}  Suppose $\sigma<\omega\le\nu<\kappa$ are cardinals.

Suppose that $\vec C=\langle C_\delta\mid\delta \in S\rangle$ witnesses $\cg_\xi(S, E^\kappa_{\le\nu},1,\vec J)$,
with $S\s E^\kappa_{>\nu}$.
Then there exists a postprocessing$^*$ function $\Phi:\mathcal K(\kappa)\rightarrow\mathcal K(\kappa)$
such that $\langle \Phi(C_\delta)\mid\delta \in S\rangle$ 
witnesses $\cg_\xi(S, \kappa, \sigma,\vec J)$.
\end{thm}
\begin{proof}

		Fix an auxiliary $C$-sequence $\vec e = \langle e_\gamma \mid \gamma <\kappa\rangle$ such that $\otp(e_\gamma) = \cf(\gamma)$ for every $\gamma<\kappa$. In what follows we shall use the operator $\Phi_D$ from Definition~\ref{dropdefn}.

\begin{claim}\label{claima}
	There is a club $D \subseteq \kappa$ such that for every club $E \subseteq \kappa$ there is a $\delta \in S$ such that
	\[\{\alpha<\delta\mid \gamma:= \min(C_\delta\setminus(\alpha+1))\ \&\ \gamma\in E^\kappa_{\leq \nu}\ \&\ |\Phi_D(e_\gamma)\cap (E\setminus \alpha)| > \sigma\}\in J_\delta^+.\]
\end{claim}
	\begin{why}
		Suppose that the claim does not hold. In this case, for every club $D\subseteq \kappa$, there is a club $F^D \subseteq \kappa$ such that for every $\delta \in S$,
		\[\{\alpha<\delta\mid \gamma:= \min(C_\delta\setminus(\alpha+1))\ \&\ \gamma\in E^\kappa_{\leq \nu}\ \&\ |\Phi_D(e_\gamma)\cap (F^D\setminus \alpha)| > \sigma\}\in J_\delta.\]
		
		Let $\mu:= \aleph_0$ so that $\mu\le\nu<\kappa$.
		We construct now a $\subseteq$-decreasing sequence $\langle D_i \mid i\leq \mu \rangle$ of clubs in $\kappa$ as follows:
		\begin{enumerate}
			\item $D_0 := \kappa$;
			\item $D_{i+1} := D_i \cap F^{D_i}$;
			\item for $i\in\acc(\mu+1)$, $D_i := \bigcap_{i'< i} D_{i'}$.
		\end{enumerate}
		Since $\mu< \kappa$, all these are club subsets of $\kappa$. Finally, consider the club
		$$D^* :=\{\gamma<\kappa\mid \otp(D_\mu\cap\gamma)=\gamma>\nu\}.$$
		
		As $\vec C=\langle C_\delta\mid\delta \in S\rangle$ witnesses $\cg_\xi(S, E^\kappa_{\le\nu},1,\vec J)$, let us pick $\delta \in S $ such that
		\[A:=\{\alpha< \delta \mid \min(C_\delta\setminus(\alpha+1))\in  D^* \cap E^\kappa_{\leq \nu}\}\]
		is in $J_\delta^+$.
		
		For every $i < \mu$, by the choice of $F^{D_i}$, the following set
		\[A_i:= \{\alpha<\delta\mid \gamma:= \min(C_\delta\setminus(\alpha+1))\ \&\ \gamma\in E^\kappa_{\leq \nu}\ \&\ |\Phi_{D_i}(e_\gamma)\cap (F^{D_i}\setminus \alpha)| >\sigma\}\]
		is in $J_\delta$.
		As $J_\delta$ is $\cf(\delta)$-additive, and $\cf(\delta)>\nu$, we may now fix $\alpha\in A\setminus\bigcup_{i<\mu}A_i$.
		Set $\gamma:=\min(C_\delta\setminus(\alpha+1))$, so that $\gamma\in D^*\cap E^\kappa_{\leq \nu}$. Since $\gamma \in D^* \subseteq \acc(D_{i})$ for every $i \leq \mu$, we have that $\Phi_{D_{i}}(e_\gamma)  \s D_i$. 
		
		As $\gamma\in D^*$, $\otp(D^*\cap\gamma)=\gamma>\nu\ge\otp(e_\gamma)$, so that $(D^*\cap\gamma)\setminus e_\gamma$ is cofinal in $\gamma$.
		Recursively construct $\langle (\delta^n,\beta^n) \mid n\le \sigma\rangle$ by letting 
		\begin{enumerate}
			\item $\delta^0:=\min( D^*\setminus (e_\gamma\cup\alpha))$, and 
			\item $\beta^0:=\min(e_\gamma\setminus\delta^0)$, and
			\item $\delta^{n+1}:=\min(D^*\setminus (e_\gamma\cup\beta^n))$, and \item $\beta^{n+1}:=\min(e_\gamma\setminus\delta^{n+1})$.
		\end{enumerate}
		Evidently, $\alpha\le\delta^n<\beta^n<\delta^{n+1}$.
		
		For every $n\le\sigma$, denote $\beta^n_i := \sup(\beta^n \cap D_i)$, and fix a large enough $j_n < \omega$ such that $\beta^n_i=\beta^n_{j_n}$ for every integer $i\ge j_n$.
		Set $i^*:=\max\{j_n\mid n\le \sigma\}$ which is finite as $\sigma$ is finite.
		Altogether, for every $n<\sigma$,
		$$\alpha\le\delta^n\le \beta^n_{i^*}=\beta^n_{i^*+1}\le\beta^n<\delta^{n+1}.$$
		
		It thus follows that $\{\beta^n_{i^*}\mid n\le\sigma\}$ consists of $\sigma+1$ many distinct elements of $\Phi_{D_{i^*}}(e_\gamma)\cap (D_{i^*+1}\setminus\alpha)$.
		But $D_{i^*+1}$ is a subset of $F^{D_{i^*}}$ so  $|\Phi_{D_{i^*}}(e_\gamma)\cap (F^{D_{i^*}}\setminus\alpha)| > \sigma$,
		and since $\gamma=\min(C_\delta\setminus(\alpha+1))$, we have contradicted the fact that $\alpha\not\in A_{i^*}$.
	\end{why}
	
	Let $D\subseteq \kappa$ be a club as given by the preceding claim.
	For all $\gamma\in E^\kappa_{\le\nu}$ and $z\in[\nu]^{<\omega}$, 
	we define a finite subset of $\gamma$:
	$$z_\gamma:=\{ \eta\in \Phi_D(e_\gamma)\mid \otp(\Phi_D(e_\gamma)\cap\eta)\in z\}.$$
	Now, for every $z\in[\nu]^{<\omega}$, define a postprocessing$^*$ function $\Phi^z:\mathcal K(\kappa)\rightarrow\mathcal K(\kappa)$ via:	
	$$\Phi^z(x):=x\cup \{ z_\gamma \setminus \sup(x\cap \gamma) \mid \gamma \in \nacc(x) \cap E^\kappa_{\le\nu}\}.$$

	\begin{claim} There exists $z\in[\nu]^{\sigma+1}$ such that
	$\vec{C^z}:=\langle \Phi^z(C_\delta)\mid\delta \in S\rangle$
witnesses $\cg_\xi(S, \kappa, \sigma,\vec J)$.
	\end{claim}
	\begin{why} 
	Suppose not. For each $z\in[\nu]^{\sigma+1}$ fix a counterexample $E_z$. Set $E:=\bigcap\{ E_z\mid z\in[\nu]^{\sigma+1}\}$.
Recalling the choice of $D$, let us now fix $\delta\in S$ for which 
		$$A:=\{\alpha<\delta\mid \gamma:=\min(C_\delta\setminus(\alpha+1))\ \&\  \gamma \in E^\kappa_{\le\nu}\ \&\ |\Phi_D(e_\gamma)\cap (E\setminus\alpha)|>\sigma\}$$
		is in $J_\delta^+$. For every $\alpha\in A$, let $\gamma_\alpha:=\min(C_\delta\setminus(\alpha+1))$ and fix $z_\alpha\in[\nu]^{\sigma+1}$ such that $e_{\gamma_\alpha}^{z_\alpha}\s E\setminus\alpha$.
		As $J_\delta$ is $\cf(\delta)$-additive and $\cf(\delta)>\nu=|[\nu]^{\sigma+1}|$,
		it follows that there exists some $z\in[\nu]^{\sigma+1}$ for which $\{\alpha\in A\mid z_\alpha=z\}$ is in $J_\delta^+$.
		As $E\s E_z$, this is a contradiction. 
	\end{why}
	Let $z$ be given by the preceding claim. Then $\Phi^z$ is as sought.
\end{proof}

\begin{cor}  Suppose $\nu<\lambda<\kappa$ are infinite regular cardinals, and $S\s E^\kappa_\lambda$.

If $\cg_\lambda(S, E^\kappa_{\nu})$ holds,
then there exists a $\lambda$-bounded $C$-sequence $\vec C=\langle C_\delta\mid\delta\in S\rangle$
satisfying the following. For every club $D\s\kappa$, for every $n<\omega$,
there exists a $\delta\in S$ such that
$\sup\{ \beta<\delta\mid \suc_n(C_\delta\setminus\beta)\s D\}=\delta$.\qed
\end{cor}

\begin{thm}\label{thma} Suppose $\xi\le\kappa$ are uncountable cardinals.

For every stationary $S\s E^\kappa_{>\omega}$,
		$\cg_\xi(S, E^\kappa_\omega,1,\vec J)$ implies $\cg_\xi(S, \kappa, \omega,\vec J)$;
\end{thm}
\begin{proof}
Suppose that $\vec C=\langle C_\delta\mid\delta \in S\rangle$ witnesses $\cg_\xi(S, E^\kappa_\omega,1,\vec J)$,
with $S\s E^\kappa_{>\omega}$. In particular, $\kappa\ge\aleph_2$.
	Fix an auxiliary $C$-sequence $\vec e = \langle e_\gamma \mid \gamma <\kappa\rangle$ such that $\otp(e_\gamma) = \cf(\gamma)$ for every $\gamma<\kappa$. In what follows we shall use the operator $\Phi_D$ from Definition~\ref{dropdefn}.
	
	\begin{claim}\label{claima}
	There is a club $D \subseteq \kappa$ such that for every club $E \subseteq \kappa$ there is a $\delta \in S$ such that
	\[\{\alpha<\delta\mid \gamma:= \min(C_\delta\setminus(\alpha+1))\ \&\ \gamma\in E^\kappa_\omega\ \&\ \Phi_D(e_\gamma)\s E\}\in J_\delta^+.\]
\end{claim}
	\begin{why}
	Suppose that the claim does not hold. In this case, for every club $D\subseteq \kappa$, there is a club $F^D \subseteq \kappa$ such that for every $\delta \in S$,
	\[\{\alpha<\delta\mid \gamma:= \min(C_\delta\setminus(\alpha+1))\ \&\ \gamma\in E^\kappa_\omega\ \&\ \Phi_D(e_\gamma)\s F^D\}\in J_\delta.\]
	
	Let $\mu:=\omega_1$ so that $\mu<\kappa$.
	We construct now a $\subseteq$-decreasing sequence $\langle D_i \mid i\leq \mu \rangle$ of clubs in $\kappa$ as follows:
	\begin{enumerate}
		\item $D_0 := \kappa$;
		\item $D_{i+1} := D_i \cap F^{D_i}$;
		\item for $i\in\acc(\mu+1)$, $D_i := \bigcap_{i'< i} D_{i'}$.
	\end{enumerate}
	Since $\mu< \kappa$, all these are club subsets of $\kappa$. Finally, consider the club
	$$D^* :=\{\gamma<\kappa\mid \otp(D_\mu\cap\gamma)=\gamma\}.$$
	
	As $\vec C=\langle C_\delta\mid\delta \in S\rangle$ witnesses $\cg_\xi(S, E^\kappa_\omega,1,\vec J)$, let us pick $\delta \in S$ such that
	\[A:=\{\alpha< \delta \mid \min(C_\delta\setminus(\alpha+1))\in  D^* \cap E^\kappa_\omega\}\]
	is in $J_\delta^+$.
	
	For every $i < \mu$, by the choice of $F^{D_i}$, the following set is in $J_\delta$:
	\[A_i:= \{\alpha<\delta\mid \gamma:= \min(C_\delta\setminus(\alpha+1))\ \&\ \gamma\in E^\kappa_\omega\ \&\ \Phi_{D_i}(e_\gamma)\s F^{D_i}\}.\]
	As $J_\delta$ is $\cf(\delta)$-additive, and $\cf(\delta)>\omega$, we may now fix $\alpha\in A\setminus\bigcup_{i<\mu}A_i$.
	Set $\gamma:=\min(C_\delta\setminus(\alpha+1))$, so that $\gamma\in D^*\cap E^\kappa_\omega$. Since $\gamma \in D^* \subseteq \acc(D_{i})$ for every $i \leq \mu$, we have that $\Phi_{D_{i}}(e_\gamma) \s D_i$ for every $i< \mu$. 
	
	For $\beta \in e_\gamma$, let $\beta_i := \sup(\beta \cap D_i)$. Since $\langle D_i \mid i\leq \mu \rangle$ is $\subseteq$-decreasing, $\langle\beta_i\mid i < \mu\rangle$ is a non-increasing sequence,
	and hence it must stabilize beyond some ordinal $j(\beta) < \mu$. That is, for every $i \geq j(\beta)$ we have $\beta_i = \beta_{j(\beta)}$. Let $i^* := \sup _{\beta \in e_\gamma}j(\beta)$,
	and note that $i^*<\mu$, since $\mu=\omega_1=|e_\gamma|^+$. In particular, this implies that \[\Phi_{D_{i^*}}(e_\gamma) = \Phi_{D_{i^*+1}}(e_\gamma) \s D_{i^*+1} \s D_{i^*} \cap F^{D_{i^*}}.\]
	On the other hand, since $\alpha\notin A_{i^*}$, we have that $\Phi_{D_{i^*}}(e_\gamma) \nsubseteq F^{D_{i^*}}$. This is a contradiction.
\end{why}

	Let $D\subseteq \kappa$ be a club as given by the preceding claim.
Consider a new $C$-sequence $\vec C^\bullet=\langle C^\bullet_\delta\mid\delta \in S\rangle$ defined via:
\[C_\delta^\bullet := C_\delta \cup \{\Phi_D(e_\gamma) \setminus \sup(C_\delta \cap \gamma) \mid \gamma \in \nacc(C_\delta) \cap E^\kappa_\omega\}.\]

Note that for any $\delta \in S$, since $\otp(C_\delta) \leq \xi$ and for each $\gamma \in E^\kappa_\omega$ we have $\otp(\Phi_D(e_\gamma)) \leq \otp(e) \leq \omega$, we have that the ordertype of every initial segment of $C^\bullet_\delta$ is strictly less than $\xi$, and hence $\otp(C^\bullet_\delta) \leq \xi$.

Now, if $E \subseteq \kappa$ is a club, then by the choice of $D$ there is some $\delta \in S$ such that
\[A:=\{\alpha<\delta\mid \gamma:=\min(C_\delta\setminus(\alpha+1))\ \&\ \gamma \in E^\kappa_\omega\ \&\ \Phi_D(e_\gamma) \subseteq E\}\]
is in $J_\delta^+$. For every $\alpha\in A$, if we let $\gamma_\alpha:=\min(C^\bullet_\delta\setminus(\alpha+1))$,
then $\gamma_\alpha\in E^\kappa_\omega$ and $C^\bullet_\delta \cap[\alpha,\gamma_\alpha)$ is equal to $\Phi_D(e_{\gamma_\alpha}) \setminus \alpha$,
which is an end segment of $\Phi_D(e_{\gamma_\alpha})$.
Since any end segment of $\Phi_D(e_{\gamma_\alpha})$ has ordertype $\omega$ as well, it follows that for any $\alpha\in A$,
there is an end segment of $C^\bullet_\delta \cap \gamma_\alpha$ of ordertype $\omega$ which is contained in $E$.
Since this interval of ordertype $\omega$ which is contained in $E$ also contains $\omega$ successive non-accumulation points of $C^\bullet_\delta \cap \gamma_\alpha$,
we infer that
$$B:=\{\beta < \delta \mid \suc_\omega(C_\delta \setminus \beta) \subseteq E\}$$
covers $A$. In particular, $B\in J_\delta^+$.
\end{proof}

\begin{thm}\label{thm45}
Let $\sigma< \nu < \xi \le \kappa$ be infinite cardinals. Suppose that $S\s E^\kappa_{>\nu}$ and $T\s\kappa$ are stationary sets.
\begin{enumerate}[(1)]
\item If $\cg_\xi(S, E^\kappa_\sigma\cap \Tr(T),1,\vec J)$ holds, then so does $\cg_\xi(S, T, \sigma,\vec J)$;
\item If $\cg_\xi(S, E^\kappa_\nu\cap \Tr(T),1,\vec J)$ holds, then so does $\cg_\xi(S, T, \sigma,\vec J)$.
\end{enumerate}
\end{thm}
\begin{proof} For the proof of both cases,
we fix an auxiliary $C$-sequence $\vec e = \langle e_\gamma \mid \gamma <\kappa\rangle$ such that $\otp(e_\gamma) = \cf(\gamma)$ for every $\gamma<\kappa$.

(1) Suppose that $\vec C=\langle C_\delta\mid\delta \in S\rangle$ witnesses $\cg_\xi(S, E^\kappa_{\sigma}\cap \Tr(T), 1,\vec J)$.
Let $\Phi^B$ be the operator from Definition~\ref{firstoperatordefn}.

\begin{claim}
There is a club $D \subseteq \kappa$ such that for every club $E \subseteq \kappa$ there is $\delta \in S$ such that
\[\{\alpha<\delta\mid \gamma:=\min(C_\delta\setminus(\alpha+1))\ \&\ \gamma \in E^\kappa_\sigma\cap\Tr(T)\text\ \&\ \Phi^{D\cap T}(e_\gamma) \subseteq E\}\in J_\delta^+.\]
\end{claim}

\begin{why}
Suppose that the claim does not hold. In this case, for every club $D\subseteq \kappa$, there is a club $F^D \subseteq \kappa$ such that, for every $\delta \in S$,
\[\{\alpha<\delta\mid \gamma:=\min(C_\delta\setminus(\alpha+1))\ \&\ \gamma \in E^\kappa_\sigma\cap\Tr(T)\ \&\ \Phi^{D\cap T}(e_\gamma) \subseteq F^D\}\in J_\delta.\]
Set $\mu:=\sigma^+$, so that $\mu\le\nu<\kappa$. We construct now a $\subseteq$-decreasing sequence $\langle D_i \mid i\leq \mu \rangle$ of clubs in $\kappa$ as follows:
\begin{enumerate}
    \item $D_0 := \kappa$;
    \item $D_{i+1} := D_i \cap F^{D_i}$;
    \item for $i\in\acc(\mu+1)$, $D_i := \bigcap_{i'< i} D_{i'}$.
\end{enumerate}

Next, let $D^* := \acc(D_{\mu})$ and fix $\delta \in S$ such that
\[A:=\{\alpha< \delta \mid \min(C_\delta\setminus(\alpha+1))\in  D^* \cap E^\kappa_\sigma\cap\Tr(T)\}\]
is in $J_\delta^+$.

For every $i < \mu$, as $D_{i+1}\s F^{D_i}$, the following set
\[A_i:= \{\alpha<\delta\mid \gamma:=\min(C_\delta\setminus(\alpha+1))\ \&\ \gamma \in E^\kappa_\sigma\cap\Tr(T)\ \&\ \Phi^{D_i\cap T}(e_\gamma) \subseteq D_{+i}\}\]
is in $J_\delta$.
As $\cf(\delta)>\nu\ge\mu$, we may fix $\alpha\in A\setminus\bigcup_{i<\mu}A_i$.
Set $\gamma:=\min(C_\delta\setminus(\alpha+1))$, so that $\gamma\in D^*\cap E^\kappa_\sigma\cap\Tr(T)$.

For every $i < \mu$, $\gamma\in\acc(D_i)\cap\Tr(T)$, so that $e_\gamma\cap D_i\cap T$ is stationary in $\gamma$, and hence $\Phi^{D_i\cap T}(e_\gamma)=\cl(e_\gamma\cap D_i\cap T)$.
Thus, for every $i<\mu$, as $\alpha\notin A_i$, it must be the case that $\Phi^{D_i\cap T}(e_\gamma) \nsubseteq D_{i+i}$;
but $D_{i+1}$ is closed, so that, in fact, $e_\gamma\cap D_i\cap T\nsubseteq D_{i+i}$.
For each $i<\mu$, pick $\beta_i\in (e_\gamma\cap D_i\cap T)\setminus D_{i+i}$.
As $|e_\gamma|<\mu$, we may now fix $(i,j)\in[\mu]^2$ such that $\beta_i=\beta_j$.
So $\beta_i\notin D_{i+1}$ while $\beta_j\in D_j\s D_{i+1}$. This is a contradiction.
\end{why}

Let $D\subseteq \kappa$ be given by the preceding claim.
It is clear that $\vec C^\bullet=\langle C^\bullet_\delta\mid\delta \in S\rangle$ 
witnesses $\cg_\xi(S, T, \sigma,\vec J)$.

\medskip

(2) Suppose that $\vec C=\langle C_\delta\mid\delta \in S\rangle$ witnesses $\cg_\xi(S, E^\kappa_{\nu}\cap \Tr(T), 1,\vec J)$.
For all $\gamma<\kappa$ and $\epsilon<\nu$, let $e_\gamma^\epsilon:=\{\beta\in e_\gamma\cap T\mid \otp(e_\gamma\cap\beta)<\epsilon\}$ so that it is an initial segment of $e_\gamma \cap T$.

\begin{claim} There is $\epsilon<\nu$ such that, for every club $E \subseteq \kappa$, there is $\delta \in S$ with
\[\{\alpha<\delta\mid \gamma:=\min(C_\delta\setminus(\alpha+1))\ \&\ \otp(e^\epsilon_\gamma\cap (E\setminus\alpha))>\sigma \}\in J_\delta^+.\]
\end{claim}
\begin{why} Otherwise, pick a counterexample $E_\epsilon$ for each $\epsilon<\nu$, and set $E:=\bigcap_{\epsilon<\nu}E_\epsilon$.
Pick $\delta\in S$ such that $$A:=\{\alpha<\delta\mid \min(C_\delta\setminus(\alpha+1))\in\acc(E)\cap\Tr(T)\cap E^\kappa_\nu\}$$ is in $J_\delta^+$.
For every $\alpha\in A$, if we let $\gamma_\alpha:=\min(C_\delta\setminus(\alpha+1))$, then $e_\gamma\cap E\cap T$ is a stationary subset of $\gamma$ of ordertype $\nu$,
so there exists some $\epsilon_\alpha<\nu$ such that $\otp(e^{\epsilon_\alpha}_{\gamma_\alpha}\cap E\setminus\alpha)>\sigma$.
As $J_\delta$ is $\cf(\delta)$-additive and $\cf(\delta)>\nu$, there must exist some $\epsilon<\nu$
for which $A_\epsilon:=\{\alpha\in A\mid \epsilon_\alpha=\epsilon\}$ is in $J_\delta^+$.
But $E\s E_\epsilon$. This is a contradiction.
\end{why}
Let $\epsilon$ be given by the claim. 
\begin{claim}
There is a club $D \subseteq \kappa$ such that for every club $E \subseteq \kappa$ there is $\delta \in S$ such that
\[\{\alpha<\delta\mid \gamma:=\min(C_\delta\setminus(\alpha+1))\ \&\  e^\epsilon_\gamma\cap D \subseteq E\ \&\ \otp( e^\epsilon_\gamma\cap (D\setminus \alpha))>\sigma\}\in J_\delta^+.\]
\end{claim}

\begin{why}
Suppose that the claim does not hold. In this case, for every club $D\subseteq \kappa$, there is a club $F^D \subseteq \kappa$ such that, for every $\delta \in S$,
\[\{\alpha<\delta\mid \gamma:=\min(C_\delta\setminus(\alpha+1))\ \&\  e^\epsilon_\gamma\cap D \subseteq F^D\ \&\ \otp( e^\epsilon_\gamma\cap (D\setminus \alpha))>\sigma\}\in J_\delta.\]
We construct now a $\subseteq$-decreasing sequence $\langle D_i \mid i\leq \nu\rangle$ of clubs in $\kappa$ as follows:
\begin{enumerate}
    \item $D_0 := \kappa$;
    \item $D_{i+1} := D_i \cap F^{D_i}$;
    \item for $i\in\acc(\nu+1)$, $D_i := \bigcap_{i'< i} D_{i'}$.
\end{enumerate}

Let $D^*:=D_\nu$ and fix $\delta\in S$ such that
$$A:=\{\alpha<\delta\mid \gamma:=\min(C_\delta\setminus(\alpha+1))\ \&\ \otp(e^\epsilon_\gamma\cap (D^*\setminus\alpha))>\sigma \}$$
is in $J_\delta^+$.

For every $i < \nu$, the following set
\[A_i:=\{\alpha<\delta\mid \gamma:=\min(C_\delta\setminus(\alpha+1))\ \&\  e^\epsilon_\gamma\cap D_i \subseteq F^{D_i}\ \&\ \otp( e^\epsilon_\gamma\cap (D_i \setminus \alpha))>\sigma\}\]
is in $J_\delta$.
Fix $\alpha\in A\setminus\bigcup_{i<\nu}A_i$.
Set $\gamma:=\min(C_\delta\setminus(\alpha+1))$, so that $\otp(e^\epsilon_\gamma\cap (D^*\setminus\alpha))>\sigma$.

For every $i < \nu$, since $D^* \s D_i$ and $D_{i+1}\s F^{D_i}$, we have that $\otp(e^\epsilon_\gamma\cap (D_{i}\setminus\alpha))>\sigma$, and hence it must be the case that $D_i\cap e^\epsilon_\gamma\nsubseteq F^{D_i}$, and therefore, $D_i\cap e^\epsilon_\gamma\nsubseteq D_{i+1}$.
So $\langle D_i\cap e^\epsilon_\gamma\mid i<\nu\rangle$ is a strictly $\s$-decreasing sequence of subsets of $e^\epsilon_\gamma$, contradicting the fact that $|e^\epsilon_\gamma|\le|\epsilon|<\nu$.
\end{why}

Let $D$ be given by the claim.
As $\epsilon<\nu$, for every $\gamma<\kappa$, $|e^\epsilon_\gamma|<\nu$.
It altogether follows that the $C$-sequence $\vec C^\bullet=\langle C^\bullet_\delta\mid\delta \in S\rangle$ defined via:
\[C_\delta^\bullet := C_\delta \cup \{\cl(D\cap e_\gamma^\epsilon)\setminus \sup(C_\delta \cap \gamma) \mid \gamma \in \nacc(C_\delta)\},\]
is as sought.
\end{proof}

\begin{cor}
Let $\mu <\sigma< \sigma^+ < \lambda < \kappa$ be infinite regular cardinals. 

Then $\cg_\lambda(E^\kappa_\lambda, E^\kappa_\sigma,1,\vec J)$ implies $\cg_\lambda(E^\kappa_\lambda, E^\kappa_\mu, \sigma,\vec J)$.
\end{cor}
\begin{proof} Appeal to Theorem~\ref{thm45}(1) with $\nu:=\sigma^+$, $\xi:=\lambda$, $S:=E^\kappa_\lambda$ and $T:=E^\kappa_\mu$.
\end{proof}

\begin{cor}\label{cor56} For every successor cardinal $\lambda$,
if $\cg_\xi(E^{\lambda^+}_\lambda, E^{\lambda^+}_{<\lambda},1)$ holds,
then so does $\cg_\xi(E^{\lambda^+}_\lambda, \lambda^+,2)$.
\end{cor}
\begin{proof} By Theorem~\ref{thm42a}, using $\sigma:=2$, $\lambda:=\nu^+$, $\kappa:=\lambda^+$, $S:=E^\kappa_\lambda$,
and $\vec J:=\langle J^{\bd}[\delta]\mid \delta\in S\rangle$.
\end{proof}
\begin{remark} This shows that Clause~(4) of Theorem~1.6 from \cite{MR3307877}
follows from Clause~(5) of the same theorem, provided that the cardinal $\kappa$ there is a successor cardinal.
\end{remark}

\section{Moving between ideals}\label{movingsection}

As shown in the Section~\ref{sectionpartitioning}, it is easier to partition a witness for $\cg_\xi(S,T,\sigma,\allowbreak\vec J)$ in the case that the ideals in $\vec J$ are normal.
So, in this section, we address the problem of deriving $\cg_\xi(S,T,\sigma,\langle\ns_\delta\mid\delta\in S\rangle)$
from $\cg_\xi(S,T,\sigma,\langle J^{\bd}[\delta]\mid\delta\in S\rangle)$. The key lemma is Lemma~\ref{claim521}. In Theorem~\ref{apartitioningjkk} it is used to improve results from Section~\ref{sectionpartitioning}. At successor cardinals Lemma~\ref{claim521} is particularly useful, the main result of this section is the following, which combines Theorems \ref{THMA} and \ref{THME}.

\begin{cor} Suppose that $\lambda$ is a successor cardinal,
and $S\s E^{\lambda^+}_\lambda$ is stationary.

Then:
\begin{enumerate}[(1)]
\item $\cg_\lambda(S,S,1,\langle \ns_\delta\mid \delta\in S\rangle)$ holds.
\item $\cg_\lambda(S,E^{\lambda^+}_{<\lambda})$ implies
$\cg_\lambda(S,\lambda^+,n,\langle \ns_\delta\mid \delta\in S\rangle)$ for every $n<\omega$;
\item If $\lambda>\aleph_1$, then $\cg_\lambda(S,E^{\lambda^+}_{<\lambda})$ implies
$\cg_\lambda(S,\lambda^+,\omega,\langle \ns_\delta\mid \delta\in S\rangle)$.
\end{enumerate}
\end{cor}
\begin{proof} Let $\nu$ denote the predecessor of $\lambda$. 

(1) By Fact~\ref{shelahs21relativeb} and Theorem~\ref{thm42} below.

(2) Then $\cg_\lambda(S,E^{\lambda^+}_{\ge\omega}\cap E^{\lambda^+}_{\le\nu})$ holds.
So, by, Corollary~\ref{cor51}, also $\cg_\lambda(S,\lambda^+,n)$ holds for every $n<\omega$.
Finally, By Theorem~\ref{thm42}(1) below, moreover $\cg_\lambda(S,\lambda^+,n,\langle \ns_\delta\mid \delta\in E^{\lambda^+}_{\lambda}\rangle)$ holds for every $n<\omega$.

(3) Assuming that $\nu$ is uncountable,
by Corollary~\ref{cor51}, also $\cg_\lambda(S,\lambda^+,\omega)$ holds.
Then, By Theorem~\ref{thm42}(1) below, moreover $\cg_\lambda(S,\lambda^+,\omega,\langle \ns_\delta\mid \delta\in E^{\lambda^+}_{\lambda}\rangle)$ holds.
\end{proof}

\begin{lemma}\label{lemma41} Let $\aleph_0 < \xi < \kappa$ and $S\s E^\kappa_\xi$ be stationary. 
Assume $1\le\sigma<\xi$.

If $\vec J=\langle J_\delta\mid\delta\in S\rangle$ is  such that each $J_\delta$ is a normal ideal on $\delta$,
then $\cg(S,T,\sigma,\vec J)$ implies $\cg_\xi(S,T,\sigma,\vec J)$.
\end{lemma}
\begin{proof} Let $\vec C=\langle C_\delta\mid\delta\in S\rangle$ be a witness to $\cg(S,T,\sigma,\vec J)$.
For each $\delta\in S$, define a function $f_\delta:\delta\rightarrow\delta$ via
$$f_\delta(\beta):=\sup(\suc_\sigma(C_\delta\setminus\beta))+1,$$
then fix a club $e_\delta$ in $\delta$ of ordertype $\cf(\delta)$ consisting of closure points of $f_\delta$,
and finally let $C_\delta^\bullet$ be the ordinal closure below $\delta$ of the following set:
$$\bigcup\{\suc_\sigma(C_\delta\setminus\beta)\mid \beta\in e_\delta\}.$$
To see that $\langle C_\delta^\bullet\mid \delta\in S\rangle$ witnesses $\cg_\xi(S,T,\sigma,\vec J)$,
let $D$ be a club in $\kappa$. Pick $\delta\in S$ for which the following set is stationary in $\delta$:
$$B:=\{\beta<\delta\mid \suc_\sigma(C_\delta\setminus\beta)\s D\cap T\}.$$
Then $e_\delta\cap B$ is stationary in $\delta$, and for every $\beta\in e_\delta\cap B$,
$\suc_\sigma(C^\bullet_\delta\setminus\beta)=\suc_\sigma(C_\delta\setminus\beta)$.
\end{proof}

\begin{lemma}\label{claim521}
Suppose that $\vec C=\langle C_\delta\mid\delta\in S\rangle$ witnesses $\cg(S,T,\sigma)$,
with $S\s E^\kappa_{>\omega}$ and $T\s\kappa$.
Then there exists a $C$-sequence $\vec e=\langle e_\delta\mid \delta\in S\rangle$ such that, for every club $D\s\kappa$,
there exists $\delta\in S$ such that the following set is stationary in $\delta$:
$$\{ \alpha\in e_\delta\mid \exists\beta\in C_\delta\,[\alpha\le\beta<\min(e_\delta\setminus(\alpha+1))\ \&\ \suc_\sigma(C_\delta\setminus\beta)\s D\cap T]\}.$$
\end{lemma}
\begin{proof} Suppose not. 
For every $\delta\in S$, let $e_\delta^0:=C_\delta$.
Next, suppose that $i<\omega$ and that $\langle e_\delta^i\mid \delta\in S\rangle$ has already been defined.
By assumption, we can find a club $D_i\s\kappa$ such that,
for every $\delta\in S$, the following set is nonstationary in $\delta$:
$$\{ \alpha\in e^i_\delta\mid \exists\beta\in C_\delta\,[\alpha\le\beta<\min(e^i_\delta\setminus(\alpha+1))\ \&\ \suc_\sigma(C_\delta\setminus\beta)\s D_i\cap T]\},$$
so let us pick a subclub $e_\delta^{i+1}$ of $e_\delta^i$ disjoint from it.

Put $D:=\bigcap_{i<\omega}D_i$. By the choice of $\vec C$, let us now pick $\delta\in S$ such that 
$$\sup\{ \beta<\delta\mid \suc_\sigma(C_\delta\setminus\beta)\s D\cap T\}=\delta.$$

Pick $\beta<\delta$ above $\min(\bigcap_{i<\omega}e_\delta^i)$ such that $\suc_\sigma(C_\delta\setminus\beta)\s D\cap T$.
Consider the ordinal $$\gamma:=\min(\suc_\sigma(C_\delta\setminus\beta)),$$
and then, for every $i<\omega$, let $\alpha_i:=\sup(e^i_\delta\cap\gamma)$. As $\acc(e^i_\delta)\s\acc(C_\delta)$,
and $\gamma$ is in $\nacc(C_\delta)$ and above $\min(e^i_\delta)$,
we infer that $\alpha_i\in e^i_\delta\cap \gamma$.
As $\langle e^i_\delta\mid i<\omega\rangle$ is a $\s$-decreasing chain, $\langle \alpha_i\mid i<\omega\rangle$ is $\le$-decreasing,
so we may find a large enough $i<\omega$ such that $\alpha_{i+1}=\alpha_i$.
In particular, $\alpha_i\in e^{i+1}_\delta$, so by the choice of $e^{i+1}_\delta$, 
$$\forall\beta\in C_\delta\,[\alpha\le\beta<\min(e^i_\delta\setminus(\alpha+1))\rightarrow\suc_\sigma(C_\delta\setminus\beta)\nsubseteq D_i\cap T].$$

On the other hand, since $\alpha_i=\sup(e^i_\delta\cap\gamma)$, it is the case that $\min(e_\delta^i\setminus(\alpha_i+1))\ge\gamma$.
Recalling also that $e_\delta^0\s C_\delta$,
altogether $\alpha_i\le\beta<\min(e_\delta^i\setminus(\alpha_i+1))$, and
$$\suc_\sigma(C_\delta\setminus\beta)\s D\cap T\s D_i\cap T.$$
This is a contradiction.
\end{proof}

An immediate consequence of the preceding lemma is an improvement of Clauses (2) and (3) of Theorem~\ref{partitioningjkk},
for the special case of $\vec J=\langle J^{\bd}[\delta]\mid \delta\in S\rangle$. 

\begin{thm}\label{apartitioningjkk} Suppose that $\vec C$ witnesses $\cg(S,T,\sigma)$ with $S\s E^\kappa_\lambda$.
\begin{enumerate}[(1)]
\item If $\onto(\ns_\lambda,\lambda)$ holds, then $\lambda\in\Theta_2(\vec C,T,\sigma)$;
\item If $\ubd(\ns_\lambda,\theta)$ holds and  $\theta< \lambda$, then $\theta\in \Theta_2(\vec C,T,\sigma)$.
\end{enumerate}
\end{thm}
\begin{proof} Let $\vec e=\langle e_\delta\mid \delta\in S\rangle$ be the corresponding $C$-sequence given by Lemma~\ref{claim521}.
Without loss of generality, $\otp(e_\delta)=\lambda$ for all $\delta\in S$.
Define $\psi_\delta:\delta\rightarrow\lambda$ via $\psi_\delta(\beta):=\otp(e_\delta\cap\beta)$, so that, for every $A\s\delta$,
$\psi_\delta[A]$ is stationary in $\lambda$ iff $A$ is stationary in $\delta$.

\medskip

(1): Suppose that $\onto(\ns_\delta,\lambda)$ holds, and fix a colouring $c:[\lambda]^2\rightarrow\lambda$ as in Theorem~\ref{normaluparrow}(2).
As $\nacc(\kappa)$ is in $\ns_\lambda$,
we may assume that for all $\eta<\alpha<\kappa$, $c(\eta,\alpha+1)=c(\eta,\alpha)$.
Now, a proof nearly identical to that of Claim~\ref{4101} yields an 
$\eta<\lambda$ such that, for every club $D\s\kappa$,
there exists a $\delta\in S$, such that, for every $ \tau<\lambda$:
$$\sup\{\beta<\delta\mid \eta < \psi_\delta(\beta)\ \&\ c(\eta,\psi_\delta(\beta))=\tau\ \&\ \suc_\sigma(C_\delta\setminus \beta)\s D\cap T\}=\delta.$$
Choose $\vec h=\langle h_\delta:\delta\rightarrow\lambda\mid\delta\in S\rangle$ satisfying 
$h_\delta(\beta)=c(\eta,\psi_\delta(\beta))$ every $\delta\in S$ and $\beta <\delta$ such that $\eta<\psi_\delta(\beta))$. 
Then $\vec h$ witnesses that $\lambda\in\Theta_2(\vec C,T,\sigma)$.

\medskip

(2) Suppose that $\ubd(J^\bd[\lambda],\theta)$ holds
with $\theta< \lambda$, and fix a colouring $c:[\lambda]^2\rightarrow\theta$ as in Theorem~\ref{normaluparrow}(1).
For every club $D\subseteq\kappa$, for all $\delta\in S$ and $\eta< \lambda$, let $D(\eta,\delta)$ denote the set:
$$\{ \tau<\theta\mid \sup\{\beta<\delta\mid \eta< \psi_\delta(\beta)\ \&\ c(\eta,\psi_\delta(\beta))=\tau\ \&\ \suc_\sigma(C_\delta\setminus \beta)\s D\cap T\}=\delta\}.$$
A proof nearly identical to that of Claim~\ref{claim892} 
yields an $\eta<\lambda$ such that, for every club $D\s\kappa$,
there exists $\delta\in S$, such that $|D(\eta,\delta)| = \theta$.
The rest of the proof is now identical to that of Theorem~\ref{partitioningjkk}(3).
\end{proof}

\begin{cor} Suppose that $\vec C$ witnesses $\cg(S,T,\sigma)$ with $S\s E^\kappa_\lambda$.

If $\lambda$ is not ineffable, then $\omega\in \Theta_2(\vec C,T,\sigma)$.
\end{cor}
\begin{proof} By Theorem~\ref{apartitioningjkk} and Corollary~\ref{cor420}(3).
\end{proof}
Motivated by Fact~\ref{fact517}(1), we ask:
\begin{question} Suppose that $\cg(S,T)$ holds for stationary $S\s E^\kappa_{>\omega}$ and $T\s\kappa$.

Does there exist a cardinal $\mu<\kappa$ such that $\cg(S,T,1,\langle \ns_\delta\restriction E^\delta_\mu\mid \delta\in S\rangle)$ holds?
\end{question}

Lemma~\ref{claim521} suggests the following variation of Definition~\ref{maindefn}.

\begin{defn}
$\cg_\xi(S, T, \half ,\vec J)$ asserts the existence of
a $\xi$-bounded $C$-sequence, $\vec C=\langle C_\delta\mid\delta \in S\rangle$ such that,
for every club $D\s\kappa$ there is a $\delta \in S$ such that
    \[\{\beta < \delta \mid (\beta,\min(C_\delta\setminus(\beta+1))]\cap D\cap T\neq\emptyset\}\in J_\delta^+.\]
\end{defn}

\begin{cor}\label{prop55} For all stationary $S\s E^\kappa_{>\omega}$ and $T\s\kappa$,
if $\cg(S,T)$ holds, then so does $\cg(S,T,\half,\langle \ns_\delta\mid \delta\in S\rangle)$.\qed
\end{cor}

We now show that it is possible to upgrade $\sigma=\half$ to $\sigma=1$,
but at the cost
 of losing control over the set $T$. 
\begin{lemma}\label{lemma56} Suppose that $S\s E^\kappa_{>\omega}$ is stationary.

For every $C$-sequence $\langle C_\delta\mid \delta\in S\rangle$ witnessing
 $\cg(S,\kappa,\half,\vec J)$,
 there exists a postprocessing$^*$ function $\Phi:\mathcal K(\kappa)\rightarrow\mathcal K(\kappa)$
such that $\langle \Phi(C_\delta)\mid \delta\in S\rangle$ 
witnesses $\cg(S,\kappa,1,\vec J)$.
\end{lemma}
\begin{proof} 
We shall make use of the operator $\Phi_D$ from Definition~\ref{dropdefn}.
\begin{claim} There exists a club $D\s\kappa$ such that, 
$\langle \Phi(C_\delta)\mid \delta\in S\rangle$ 
witnesses $\cg(S,\kappa,1,\vec J)$.
\end{claim}
\begin{why} Suppose not. In this case, for every club $D\subseteq \kappa$, there is a club $F^D \subseteq \kappa$ such that for every $\delta \in S$
$$\{\beta<\delta\mid \min(\Phi_D(C_\delta)\setminus(\beta+1))\in F^D\}\in J_\delta.$$
Construct a $\subseteq$-decreasing sequence $\langle D_n \mid n<\omega \rangle$ of clubs in $\kappa$
by letting $D_0 := \kappa$ and $D_{n+1} := \acc(D_n) \cap F^{D_n}$ for every $n<\omega$.
Set $D:=\bigcap_{n<\omega}D_n$ and then pick $\delta\in S$ for which the following set is in $J_\delta^+$:
$$B:=\{ \beta\in C_\delta\mid (\beta,\min(C_\delta\setminus(\beta+1))]\cap D\neq\emptyset\}.$$
In particular, $\delta\in\acc(D)$, so that, for every $n<\omega$,
$$\Phi_{D_n}(C_\delta)=\{\sup(D_n \cap \eta )\mid \eta \in C_\delta, \eta > \min(D_n)\}.$$

Now, as $\cf(\delta)>\omega$ and $J_\delta$ is $\cf(\delta)$-complete, the following set is nonempty:
$$B\setminus\bigcup_{n<\omega}\{\beta<\delta\mid \min(\Phi_{D_n}(C_\delta)\setminus(\beta+1))\in F^{D_n}\},$$
so we may pick in it some ordinal $\beta$.
Set $\gamma:=\min(C_\delta\setminus(\beta+1))$. As $\beta\in B$, we know that $D\cap(\beta,\gamma]\neq\emptyset$.
In particular, for every $n<\omega$, $\acc(D_n)\cap(\beta,\gamma]\neq\emptyset$ and $\beta_n:=\sup(D_n\cap\gamma)$ is an element of $D_n$ greater than $\beta$,
so that $$\min(\Phi_{D_n}(C_\delta)\setminus(\beta+1))=\{\beta_n\}.$$
As $\langle D_n \mid n<\omega\rangle$ is a $\s$-decreasing chain,
we may fix $n<\omega$ such that $\beta_{n+1}=\beta_n $. 
Then $\min(\Phi_{D_n}(C_\delta)\setminus(\beta+1))=\beta_n=\beta_{n+1}\in D_{n+1}\s F^{D_n}$,
contradicting the choice of $\beta$.
\end{why}

Let $D$ be given by the preceding claim. Then $\Phi:=\Phi_D$ is as sought.
\end{proof}

\begin{cor}\label{thm42b} Suppose that $S\s E^\kappa_{>\omega}$ is stationary.

If $\cg(S,\kappa)$ holds, then so does $\cg(S,\kappa,1,\langle \ns_\delta\mid \delta\in S\rangle)$.
\end{cor}
\begin{proof} By Corollary~\ref{prop55} and Lemma~\ref{lemma56}. 
\end{proof}

Note that the preceding result is restricted to $\sigma:=1$ and $T:=\kappa$. We now provide a condition sufficient for waiving this restriction.

\begin{thm}\label{thm42} Suppose that $\xi$ is an infinite successor cardinal, and $S\s E^\kappa_{\xi}$ is stationary.
\begin{enumerate}[(1)]
\item If $\cg_\xi(S,T,\sigma)$ holds, then so does $\cg_\xi(S,T,\sigma,\langle \ns_\delta\mid \delta\in S\rangle)$;
\item If $\cg_\xi(S,T,\half)$ holds, then so does $\cg_\xi(S,T,1,\langle \ns_\delta\mid \delta\in S\rangle)$.
\end{enumerate}
\end{thm}
\begin{proof} We provide a proof of Clause~(1) and leave the modification of the argument to obtain Clause~(2) to the reader.

As $\cg_\xi(S,T,\xi)$ is equivalent to $\cg(S,T,\kappa)$, by Lemma~\ref{taillemma}, we may assume that $\sigma<\xi$.
Now, suppose that $\vec C=\langle C_\delta\mid\delta\in S\rangle$ witnesses $\cg_\xi(S,T,\sigma)$,
and let $\vec e=\langle e_\delta\mid \delta\in S\rangle$ be the corresponding $C$-sequence given by Lemma~\ref{claim521}.
For each $\delta\in S$,
by possibly shrinking $e_\delta$ as in the proof of Lemma~\ref{lemma41}, we may assume that 
for every $\gamma\in e_\delta$ and every $\beta<\gamma$, 
$\sup(\suc_\sigma(C_\delta\setminus\beta))<\gamma$. 

 Let $\mu$ be such that $\xi=\mu^+$,
and then, for all $\delta\in S$ and $\alpha\in e_\delta$, let $\varphi_{\delta,\alpha}$ be some surjection from $\mu$ to $C_\delta\cap[\alpha,\min(e_\delta\setminus(\alpha+1)))$.

For every $i<\mu$, let $C_\delta^i$ be the ordinal closure below $\delta$ of the following set:
$$\bigcup\{\suc_\sigma(C_\delta\setminus \varphi_{\delta,\alpha}(i))\mid \alpha\in e_\delta\}.$$

\begin{claim} There exists $i<\mu$, such that, for every club $D\s\kappa$,
there exists $\delta\in S$ for which the following set is stationary in $\delta$:
$$\{\beta<\delta\mid \suc_\sigma(C_\delta^i\setminus\beta)\s D\cap T\}.$$
\end{claim}
\begin{why}
 Suppose not. For each $i<\mu$, pick a counterexample $D_i\s\kappa$. Consider the club $D:=\bigcap_{i<\mu}D_i$.
By the choice of $\vec e$, pick $\delta\in S$ such that
$$A:=\{ \alpha\in e_\delta\mid \exists\beta\in C_\delta\,[\alpha\le\beta<\min(e_\delta\setminus(\alpha+1))\ \&\ \suc_\sigma(C_\delta\setminus\beta)\s D\cap T]\}$$
is stationary. For every $\alpha\in e_\delta$, pick $i_\alpha<\mu$ such that $\beta_\alpha:=\varphi_{\delta,\alpha}(i)$ witnesses that $\alpha\in A$, that is,
such that $\suc_\sigma(C_\delta\setminus\varphi_{\delta,\alpha}(i))\s D\cap T$.
As $\cf(\delta)=\xi>\mu$, there must exist some stationary $A^*\s A$ on which the map $\alpha\mapsto i_\alpha$ is constant,
with value, say, $i^*$.
For every pair $\alpha<\alpha'$ of ordinals from $e_\delta$, 
$$\sup(\suc_\sigma(C_\delta\setminus \varphi_{\delta,\alpha}(i^*)))<\alpha'\le\min(\suc_\sigma(C_\delta\setminus \varphi_{\delta,\alpha'}(i^*))),$$
so, recalling the definition of $C_\delta^{i*}$, for every $\alpha\in A^*$,
$$\suc_\sigma(C_\delta^{i^*}\setminus \alpha)=\suc_\sigma(C_\delta\setminus \varphi_{\delta,\alpha}(i^*))\s D\cap T,$$
contradicting the fact that $D\s D_{i^*}$.
\end{why}

Let $i<\mu$ be given by the preceding. Then $\langle C_\delta^i\mid \delta\in S\rangle$ witnesses  $\cg_\xi(S,T,\sigma,\allowbreak\langle \ns_\delta\mid \delta\in S\rangle)$.
\end{proof}

\begin{remark}\label{asperoremark} 
The above ordertype restriction cannot be waived, that is,
the hypothesis $\cg_\xi(S,T,\sigma)$ in Theorem~\ref{thm42} cannot be relaxed to $\cg(S,T,\sigma)$.

By Theorem~\ref{nonreflectingthm}(1), if there exists a nonreflecting stationary subset of $E^{\aleph_2}_{\aleph_0}$,
then $\cg(E^{\aleph_2}_{\aleph_1},E^{\aleph_2}_{\aleph_0})$ holds,
and then, by Theorem~\ref{thma}, using $\xi=\kappa=\aleph_2$ and $S=E^{\aleph_2}_{\aleph_1}$, so does $\cg(E^{\aleph_2}_{\aleph_1},\omega_2,\omega)$.
Now, if the ordertype restriction in Theorem~\ref{thm42} could have been waived, 
then this would imply that $\cg(E^{\aleph_2}_{\aleph_1},\omega_2,\omega,\langle \ns_\delta\mid\delta\in E^{\aleph_2}_{\aleph_1}\rangle)$ holds.
In particular, by Lemma~\ref{lemma41}, $\cg_{\omega_1}(E^{\aleph_2}_{\aleph_1},\omega_2,\omega)$ holds.
However, running the forcing of Theorem~1.6 from \cite{MR3307877} over a model of $\square_{\omega_1}$,
one gets a generic extension with a nonreflecting stationary subset of $E^{\aleph_2}_{\aleph_0}$ 
in which $\cg_{\omega_1}(E^{\aleph_2}_{\aleph_1},\omega_2,2)$ fails. This is a contradiction.
\end{remark}
\begin{remark} An obvious complication of the proof of Theorem~\ref{thm42} shows that if $\xi$ an infinite successor cardinal, $S\s E^\kappa_{\xi}$ is stationary,
and $\vec C$ is a $\xi$-bounded $C$-sequence over $S$ such that $\theta\in\Theta_2(\vec C,T,\sigma)$, 
then there exists a $\xi$-bounded $C$-sequence $\vec{C^\bullet}$ over $S$
such that $\theta\in\Theta_2(\vec{C^\bullet},T,\sigma,\langle \ns_\delta\mid \delta\in S\rangle)$.
\end{remark}

\section*{Acknowledgments}
The first author is supported by the Israel Science Foundation (grant agreement 2066/18).
The second author is partially supported by the Israel Science Foundation (grant agreement 2066/18)
and by the European Research Council (grant agreement ERC-2018-StG 802756).

\end{document}